\documentclass[11pt,a4paper]{amsbook}
\usepackage{amscd,amssymb,amsmath,amsthm,amsfonts}

\numberwithin{equation}{section}
\newtheoremstyle{mythmsty}%
  {3pt}%
  {3pt}%
  {\itshape}%
  {}%
  {\bfseries}%
  {.}%
  {.5em}%
  {#1 #2\thmnote{(#3)}}%
\theoremstyle{plain}
\newtheorem{theorem}[equation]{Theorem}
\newtheorem{lemma}[equation]{Lemma}
\newtheorem{cor}[equation]{Corollary}
\newtheorem{prop}[equation]{Proposition}
\theoremstyle{definition}
\newtheorem{defn}[equation]{Definition}
\theoremstyle{remark}
\newtheorem{remark}[equation]{Remark}

\newtheorem{ex}[equation]{Example}
\newtheoremstyle{noname}%
  {3pt}%
  {3pt}%
  {}%
  {}%
  {\scshape}%
  {}%
  {.5em}%
  {#3(#2)}%
\theoremstyle{noname}
\newtheorem{fact}[equation]{}
\newtheoremstyle{citing}%
  {3pt}%
  {3pt}%
  {\itshape}%
  {}%
  {\bfseries}%
  {.}%
  {.5em}%
  {\thmnote{#3}}%
\theoremstyle{citing}
\newtheorem*{varthm}{}%
\newtheoremstyle{citingex}%
  {3pt}%
  {3pt}%
  {}%
  {}%
  {\bfseries}%
  {.}%
  {.5em}%
  {\thmnote{#3}}%
\theoremstyle{citingex}
\newtheorem*{varex}{}%

\newcommand{\ie}{\mbox{i.e.\ }}
\newcommand{\eg}{\mbox{e.g.\ }}

\newcommand{\cf}{\mbox{cf.\ }}
\newcommand{\page}{\mbox{p.\ }}
\newcommand{\resp}{\mbox{resp.\ }}
\newcommand{\wolg}{\mbox{w.l.o.g.\ }}

\newcommand{\set}[1]{\{#1\}}
\newcommand{\sett}[2]{\{#1\,|\,#2\}}

\DeclareMathOperator{\l<}{\langle}
\DeclareMathOperator{\r>}{\rangle}
\newcommand{\gen}[1]{\l<#1\r>}

\newcommand{\field}[1]{\mathbb{#1}}
\newcommand{\Q}{\field{Q}}
\newcommand{\R}{\field{R}}
\newcommand{\C}{\field{C}}
\newcommand{\Ham}{\field{H}}
\newcommand{\N}{\field{N}}
\newcommand{\Z}{\field{Z}}
\newcommand{\Cn}[1]{\mathfrak{Z}_{#1}}
\newcommand{\Zn}[1]{\Z/{#1}}

\newcommand{\F}{\field{F}}

\newcommand{\ug}[1]{{{#1}^{\times}}}

\newcommand{\bs}{\backslash}
\newcommand{\ift}{\infty}
\newcommand{\ndiv}{\nmid}
\newcommand{\gdw}{\Leftrightarrow}

\newcommand{\impl}{\Rightarrow}

\newcommand{\pp}{{\mathfrak p}}

\newcommand{\OO}{{\mathcal O}}
\newcommand{\PP}{{\mathfrak P}}

\newcommand{\qt}[1]{\quad\text{#1}\quad}

\newcommand{\qqt}[1]{\qquad\text{#1}\qquad}

\DeclareMathOperator{\lcm}{lcm}
\DeclareMathOperator{\id}{id}

\DeclareMathOperator{\Aut}{Aut}

\DeclareMathOperator{\Gal}{Gal}
\DeclareMathOperator{\I}{I}
\DeclareMathOperator{\No}{N}

\DeclareMathOperator{\Tr}{Tr}
\DeclareMathOperator{\Val}{\mathcal{V}}
\DeclareMathOperator{\V0}{{\mathcal{V}_0}}

\DeclareMathOperator{\Br}{Br}
\DeclareMathOperator{\VBr}{VBr}
\DeclareMathOperator{\IBr}{IBr}
\DeclareMathOperator{\SBr}{SBr}

\DeclareMathOperator{\ind}{ind}
\DeclareMathOperator{\inv}{inv}

\DeclareMathOperator{\calD}{\mathcal{D}}
\DeclareMathOperator{\calA}{\mathcal{A}}
\DeclareMathOperator{\calH}{\mathcal{H}}

\DeclareMathOperator{\calF}{\mathcal{F}}

\DeclareMathOperator{\calP}{\mathcal{P}}

\DeclareMathOperator{\Fix}{Fix}
\DeclareMathOperator{\disc}{disc}
\DeclareMathOperator{\charak}{char}
\DeclareMathOperator{\Irr}{Irr}
\DeclareMathOperator{\ord}{ord}

\DeclareMathOperator{\res}{res}

\DeclareMathOperator{\im}{im}

\newcommand{\lra}{\longrightarrow}
\newcommand{\lms}{\longmapsto}

\newcommand{\wt}{\widetilde}

\newcommand{\dummy}[1]{}
\renewcommand{\Zn}{\Cn}

\begin{document}
\title{\titel}

\author{\myname}
\address{\adresse}
\email{hanke@math.rwth-aachen.de}
\thanks{The author was supported in part by the DAAD}

\date{\today}

\titlepage

\newcommand{\titel}{A Direct Approach to Noncrossed Product Division Algebras}
\newcommand{\myname}{Timo Hanke}
\newcommand{\adresse}{Timo Hanke
\\Universit\"at Potsdam
\\Institut f\"ur Mathematik
\\Postfach 60 15 53
\par 14415 Potsdam
}
\newcommand{\betreuer}{Prof. Dr. J. Gr\"ater}
\newcommand{\namekurz}{T. Hanke}
\newcommand{\betreuerkurz}{J. Gr\"ater}
\newcommand{\datum}{12. M\"arz 2001}
\newcommand{\gebdatum}{12. April 1975}
\newcommand{\gebort}{Braunschweig}
\newcommand{\disputdatum}{16. Juli 2001}
\newcommand{\druckdatum}{19. August 2001}

\mbox{}\vspace{15mm}
\begin{center}
{\LARGE\bf
\titel}

\vspace{5cm}
{\Large\bf
Dissertation}

\vspace{15mm}
\large
zur Erlangung des akademischen Grades \\
Doktor der Naturwissenschaften \\
(Dr. rer. nat.) \\
in der Wissenschaftsdisziplin Algebra und Zahlentheorie

\vspace{15mm}
eingereicht an der \\
Naturwissenschaftlichen Fakult\"at \\
der Universit\"at Potsdam

\vspace{15mm}
von \myname, \\
geboren am \gebdatum~in \gebort.

\vspace{15mm}
Potsdam, den \druckdatum

\end{center}

\titlepage
\begin{minipage}[b]{12cm}
Corrected version as of \today.\\
 \flushleft
Typographical corrections and small changes were made within the text as long as they did not significantly change the page breaks.
Newly added comments were placed within the page margins on pages 25,27,31.
\end{minipage}
\vfill
\begin{minipage}[b]{15cm}
Eingereicht am : \datum. \\
\\
Gutachter : \\
\\
Prof. Dr. Eberhard Becker, Universit\"at Dortmund, \\
Prof. Dr. Joachim Gr\"ater, Universit\"at Potsdam, \\
Prof. Dr. Adrian R. Wadsworth, University of California, San Diego. \\
\\
Tag der Disputation : \disputdatum.
\end{minipage}

\frontmatter

\chapter*{Preface}

The history of division algebras
starts with Hamilton's construction of the real quaternions in 1843,
which were the first example of a skewfield.
Over the field of real numbers, as well as over any real closed field,
they are the only finite-dimensional noncommutative division algebra.
For other ground fields the situation can be much more complicated,
and it is a major problem in the theory of division algebras 
to classify all finite-dimensional division algebras over a given field,
or more precisely, with a given centre.
The centre of an $F$-algebra is denoted by $Z(A)$
and $A$ is called a central $F$-algebra if $Z(A)=F$.
In the following all $F$-algebras are assumed to be finite-dimensional.
The finite-dimensional central $F$-division algebras are better understood
in the context of the finite-dimensional central simple $F$-algebras.
By the Wedderburn structure theorem, 
every central simple $F$-algebra is isomorphic to a matrix ring $M_n(D)$
over a unique (up to isomorphism) central $F$-division algebra $D$,
which is called the underlying division algebra of $A$.
Two central simple $F$-algebras are called similar
if their underlying division algebras are isomorphic.
This defines an equivalence relation.
The tensor product (taken over $F$) defines a group structure 
on the equivalence classes of finite-dimensional central simple $F$-algebras.
This group is called the Brauer group $\Br(F)$ of~$F$.
It classifies the finite-dimensional central $F$-division algebras, 
because 
every element from the Brauer group is represented by a unique
$F$-central division algebra.

To give more examples of finite-dimensional central simple algebras, 
Hamilton's construction has been subsequently generalized to 
quaternion algebras over arbitrary fields, symbol algebras, crossed products and generalized crossed products,
all of which will play a role in this work.
Due to its importance the construction of crossed products shall be shortly repeated here.
Let $K/F$ be a Galois extension with  Galois group $G$ 
and let $f$ be a map $f:G\times G\to K^\times$.
An associative multiplication is defined on the $F$-space
$A=\bigoplus_{\sigma\in G} Kz_\sigma$ by the rule
\begin{equation*}
  kz_\sigma\cdot k'z_\tau=k\sigma(k')f(\sigma,\tau)z_{\sigma\tau}
\qt{for all $k,k'\in K$ and $\sigma,\tau\in G$,}
\end{equation*}
if and only if $f$ is a $2$-cocycle, \ie
\[ \rho(f(\sigma,\tau))f(\rho,\sigma\tau)=f(\rho,\sigma)f(\tau,\rho) 
\qt{for all $\rho,\sigma,\tau\in G$.}\]
In this case, $A$ is a central simple $F$-algebra that contains a strictly maximal subfield isomorphic to $K$ (\ie $\dim_F A=[K:F]^2$).
The algebra $A$ is called a crossed product of $K$ with $G$
and is written $A=(K,G,f)$.

However, it is in general not easy to give examples of finite-dimensional division algebras,
since it is difficult to decide whether a given crossed product is a division algebra or not.
A construction that always leads to division algebras 
are the twisted (or skew) polynomial rings, twisted (or skew) function fields and twisted (or skew) Laurent series rings.
These were used by Hilbert to give examples of noncommutative ordered division rings, though his examples are not finite-dimensional.
We shall shortly repeat here the special case that yields finite-dimensional division algebras.
Let $K/k$ be a finite cyclic field extension with $\Gal(K/k)=\gen{\sigma}$.
Let $D$ be a finite-dimensional central $K$-division algebra
and suppose that $\sigma$ extends to a $k$-algebra automorphism $\wt\sigma$ of $D$.
The twisted polynomial ring $D[x;\wt\sigma]$ is the ring of polynomials over $D$ in the indeterminate $x$ with multiplication ``twisted'' by the rule
\begin{equation*}
 ax^i\cdot bx^j=a\wt\sigma^i(b)x^{i+j} \qt{for all $a,b\in D$, $i,j\in\N_0$.} 
\end{equation*}
The twisted function field $D(x;\wt\sigma)$ is the quotient ring of $D[x;\wt\sigma]$
and the twisted Laurent series ring $D(\!(x;\wt\sigma)\!)$ is the analog of $D[x;\wt\sigma]$ for Laurent series in $x$ over $D$.
Both $D(x;\wt\sigma)$ and $D(\!(x;\wt\sigma)\!)$ are division rings
that are finite-dimensional over their centre.

We now return to crossed products.
As a consequence of the theorem of Skolem-Noether,
if a central simple $F$-algebra $A$ contains a strictly maximal subfield $K$
which is Galois over $F$,
then $A$ is isomorphic to a crossed product $(A,G,f)$ for some $2$-cocycle $f$.
We simply say that $A$ is a crossed product if it is isomorphic to a crossed product, 
otherwise we say $A$ is a noncrossed product.
Thus, $A$ is a crossed product if and only if it contains a strictly maximal subfield that is Galois over the centre of $A$.

The crossed products provide examples of central simple algebras
whose structure is particularly nice,
since the multiplication can be described by a group action.
Besides that,
their importance lies in the fact that they already describe the whole Brauer group.
For, every central simple $F$-algebra is similar to a crossed product,
thus every element from $\Br(F)$ is represented by a crossed product.
This is equivalent to the statement that every simple algebra has a Galois splitting field,
which in turn is a consequence from the fact that every division algebra contains a maximal separable subfield.
Since the crossed products are described by cocycles,
they provide the bridge to the cohomological interpretation of the Brauer group,
which is formulated by the isomorphism $H_c^2(\Gal(F_{sep}/F),\ug{F_{sep}})\cong\Br(F)$.

Since any central simple algebra is similar to a crossed product,
the question naturally arises,
whether any central simple algebra itself is a crossed product,
\ie whether any central simple algebra contains a strictly maximal subfield which is Galois over the centre.
In general, questions about the subfields of a simple algebra $A$
are difficult,
because $A$ can contain subfields of
quite a lot of different isomorphism classes.
However, it is known that any central simple algebra of degree $2,3,4,6$ and $12$ is a crossed product.
Moreover, as a classic (and deep) result of number theory,
any central simple algebra over a local or global field $F$ is a crossed product (there is even a cyclic strictly maximal subfield).
Trivially, this is also true over finite, real closed or algebraically closed fields $F$.
But in general the question remained open for several decades till 1972
when Amitsur gave the first examples of noncrossed products in \cite{amitsur:central-simple-div-alg}.
Today it is still open for all degrees $n=2^kp_1\cdots p_r$
with $n\nmid 12$,
where $0\leq k\leq 2$, $r\geq 0$, and 
$p_1,\ldots,p_r$ are pairwise different prime numbers $\geq 3$.
We now look closer at the noncrossed product constructions presented so far
and particularly focus on 
the centres of the constructed noncrossed products and
the indices and exponents that can be realized.

Amitsur's noncrossed product examples in \cite{amitsur:central-simple-div-alg} are so-called generic division algebras $D_n=\Q(X_1,\ldots,X_m)$, $m\geq 2$,
where the $X_i$ are $n\times n$ matrices whose entries are commutative indeterminates over $\Q$.
$D_n$ is a division algebra of degree $n$.
The central theorem of \cite{amitsur:central-simple-div-alg} states that if $D_n$ is a crossed product with group $G$,
then any central simple algebra $D$ of degree $n$ over a field of characteristic zero is a crossed product with the same group $G$.
Thus, if there are two division algebras $D_1,D_2$ of degree $n$ over fields of characteristic zero
that are crossed products but not crossed products with the same group,
this proves that $D_n$ is a noncrossed product.
Amitsur constructs $D_1$ as an iterated twisted Laurent series ring
and $D_2$ as a division algebra over a $p$-adic field $\Q_p$.
The subfields of the Laurent series ring $D_1$ can be controlled using the 
standard valuation that is introduced by the indeterminates
(though Amitsur does not use valuations explicitly).
His result is that $D_n$ is a noncrossed product
if $p^2\mid n$ for an odd prime $p$ or $8\mid n$.
Unfortunately the centre of $D_n$ is unknown, and so is its Brauer group.
However, it can be shown that $\exp D_n=\ind D_n=n$.
Based on Amitsur's idea
there have been subsequent constructions by various authors
covering noncrossed products in characteristic $p$ 
and noncrossed products whose index exceeds its exponent.

Another approach was made by Jacob and Wadsworth in \cite{jacob-wadsworth:constr-noncr-prod}.
They explicitly use (noncommutative) valuations on a division algebra $D$
to control the subfields of $D$.
More precisely, division algebras $D$ are considered
that are totally and tamely ramified with respect to two valuations $v_1$ and $v_2$.
If $D$ is a crossed product with group $G$,
then $G$ is isomorphic to a subgroup of the relative value groups of $v_1$ and $v_2$.
Thus, if the relative value groups of $v_1$ and $v_2$ 
have no common subgroup of a certain order,
then $D$ is a noncrossed product.
In this way the existence of noncrossed products of index $p^m$ and exponent $p^n$ is shown
for any $n\geq 2$ if $p$ is an odd prime, $n\geq 3$ if $p=2$, and any $m\geq n$.
The noncrossed products appear as the underlying divison algebra of some tensor product $A_1\otimes A_2$, 
which itself is a crossed product.
The centres are known, since they are part of the construction,
though their transcendence degree can be relatively high
(it depends on $m$).

Brussel showed in \cite{brussel:noncr-prod}
that there are noncrossed products with centres $\Q(t)$ and $\Q(\!(t)\!)$.
These are in some sense the ``simplest'' centres over which noncrossed products exist.
The Brauer groups are well known, \eg
by a theorem of Witt,
\begin{equation}
  \label{eq:Br-Qt}
\Br(\Q(\!(t)\!))\cong\Br(\Q)\oplus X(G),
\end{equation}
where $X(G)$ is the character group of the absolute Galois group of $\Q$.
Brussel constructs elements of $\Br(\Q(\!(t)\!))$ 
via the presentation (\ref{eq:Br-Qt}) whose underlying division algebras are noncrossed products.
More precisely, a tensor product $E:=A\otimes_\Q (K(\!(t)\!)/\Q(\!(t)\!),\sigma,t)$
is constructed,
where $A$ is a central simple algebra over $\Q$ 
and $K/\Q$ is a finite cyclic extension with $\Gal(K/\Q)=\gen{\sigma}$.
He shows that the underlying division algebra of $E$ is a noncrossed product
if any splitting field $M$ of $A$ containing $K$ with degree 
$[K:\Q]\ind(A\otimes_\Q K)$ over $\Q$ is not Galois over $\Q$.
By arranging such situations of $A$ and $K$
with purely number theoretic methods,
the existence of noncrossed product division algebras of index $p^m$ and exponent $p^n$ is shown
for any $n\geq 2$ if $p$ is an odd prime, $n\geq 3$ if $p=2$, and $n\leq m< 2n$.
The noncrossed products over $\Q(t)$ appear analogously as the underlying division algebra of 
$A\otimes_\Q (K(t)/\Q(t),\sigma,t)$.
Moreover, the above is done in \cite{brussel:noncr-prod} for $\Q$ replaced by a number field $k$.
The lowest possible indices and exponents may increase if $k$ contains roots of unity.
In any case, the noncrossed products are the underlying division algebras of
tensor products, which themselves are crossed products.
A more detailed survey on the various approaches to noncrossed products is contained in \cite[\S 5]{wadsworth:survey}.

The goal of this work is to give a direct approach to the construction of noncrossed product division algebras that covers ``simple'' centres as well as ``small'' indices and exponents
and that makes it possible to construct explicit examples.
By a direct approach, we mean a construction that directly yields the noncrossed product division algebras instead of a matrix ring over it.
This is accomplished \eg by Amitsur's examples, though not for the ``simple'' centres.
By an explicit example we mean that we can effectively compute inside the noncrossed product division algebra
and that we can write down a multiplication table of some basis.
The result can be formulated as follows.
\begin{varthm}[Result]
The construction of iterated twisted function fields 
and iterated twisted Laurent series rings of the form
\[
D(x_1;\wt\sigma_1)\cdots(x_r;\wt\sigma_r) \qt{and}
D(\!(x_1;\wt\sigma_1)\!)\cdots(\!(x_r;\wt\sigma_r)\!)
\]
meets all these requirements.
An explicit example of noncrossed product division algebras of this form
is given for the case $r=2$.
\end{varthm}
It was mentioned before that the iterated twisted Laurent series rings also play a role in Amitsur's noncrossed product examples.
The iterated twisted Laurent series rings constructed by Amitsur are crossed products,
but only a certain group can occur as the Galois group of a maximal subfield.
It is now surprising that we get iterated twisted Laurent series rings
which are even noncrossed products.
The central theorem on the way to this result is the following noncrossed product criterion for so-called inertially split division algebras.
\begin{varthm}[Noncrossed Product Criterion]
Let $D$ be a valued finite-dimensional central $F$-division algebra
that is inertially split.
If $D$ is a crossed product, then the residue algebra $\bar D$ contains a maximal subfield that is Galois over the residue field $\bar F$ of $F$
(not just over $Z(\bar D)$).
\end{varthm}
This theorem is practical if the structure of $\bar D$ is well understood,
\eg if $\bar D$ is a division algebra over a number field.
The actual noncrossed product criterion, of course, is the contraposition of the statement.
The terminology ``inertially split'' is due to Jacob and Wadsworth,
who define inertially split division algebras over Henselian fields in \cite{jacob-wadsworth:div-alg-hensel-field}
as division algebras that have an inertial splitting field,
and the proof is based on results from \cite{jacob-wadsworth:div-alg-hensel-field}.
If $\bar F$ is not perfect, then further on, a lemma from Saltman \cite{saltman:noncr-prod-small-exp} on $p$-algebras is required,
which states that if a division algebra contains a normal maximal subfield
then it also contains a Galois maximal subfield.

The Noncrossed Product Criterion easily extends to non-Henselian valued division algebras.
For, if a valued division algebra is a crossed product,
then its Henselization,
as introduced by Morandi in \cite{morandi:henselization-div-alg},
is a crossed product with the same residue algebra.
Here, for convenience,
we call a valued division algebra inertially split if its Henselization 
is inertially split in the sense of \cite{jacob-wadsworth:div-alg-hensel-field},
so that the Noncrossed Product Criterion as stated holds for arbitrary valued division algebras.
The reason for this is to treat the iterated twisted function fields,
which are not Henselian valued, with this criterion.
It shall be mentioned that in the Henselian case
even the converse is true,
and this result can be used for example to show that every division algebra over $k(\!(t_1,\ldots,t_r)\!)$ for a local field $k$ of characteristic zero is a crossed product.

Results from \cite{jacob-wadsworth:div-alg-hensel-field} are further used
to compute index and exponent of inertially split division algebras.
However, the result on the exponent can not be carried over immediately to the non-Henselian case.
Only in special cases, including the considered iterated twisted function fields,
the exponent coincides with the exponent of the Henselization.
It is further investigated how the exponent of an inertially split division algebra $D$ can be computed from the residue algebra $\bar D$.
If $\bar D$ is a division algebra over a global field,
this is possible under certain circumstances using local-global principles.

The Noncrossed Product Criterion suggests to divide the proof of existence of noncrossed products over $F$, as well as their construction, into two steps.
First, one has to prove the existence of $\bar F$-division algebras $\wt D$
that do not contain a maximal subfield Galois over $\bar F$.
Note that $\bar F$ is not required to be the centre of $\wt D$,
hence $\wt D$ may well be a crossed product.
Second, one has to find a central $F$-division algebra $D$ that is inertially split with residue algebra $\bar D$ isomorphic to $\wt D$.
Then $D$ is a noncrossed product.
Besides that, one should keep track of the index and exponent in this process.
In this way, the noncrossed product division algebra $D$ is obtained ``directly''.

It turns out that if $\bar F$ is a number field,
the first step can be done with the same number theoretic arguments 
that are used by Brussel in \cite{brussel:noncr-prod}.
More precisely, as mentioned before,
Brussel arranges a situation of a cyclic extension $K/k$ of number fields
and a central simple $k$-algebra $A$ such that
any splitting field of $A$ containing $K$ with degree
$[K:k]\ind(A\otimes_k K)$ over $k$ 
is not Galois over $k$.
If instead of $A$ the underlying division algebra $D$ of $A\otimes_k K$ is considered,
this is equivalent to that any maximal subfield of $D$ is not Galois over $k$.

The second step is possible under certain circumstances.
We introduce a property of the valued field $F$ called \emph{inertial lift property}.
Roughly spoken, it means that there are inertial lifts over $F$ of both,
central $\bar F$-division algebras and Galois extensions of $\bar F$,
such that they ``fit together''.
Besides the Henselian fields,
all valued fields $F$ have the inertial lift property 
whose valuation ring $V_F$ contains a subfield that maps isomorphically onto $\bar F$ under the canonical residue map.
For example the rational function fields $F=\bar F(t_1,\ldots,t_r)$
and the Laurent series fields $F=\bar F(\!(t_1,\ldots,t_r)\!)$
regarded with respect to the composite of the $t_i$-adic valuations
have the inertial lift property.
This leads to the 

\begin{varthm}[Lift Theorem]
Suppose that $F$ has the inertial lift property.
Let $K/\bar F$ be a field extension with $[K:\bar F]=n<\ift$,
and let $\wt D$ be a finite-dimensional central $K$-division algebra.
There exists a finite-dimensional central $F$-division algebra $D$ 
that is inertially split with residue algebra $\bar D$ isomorphic to $\wt D$
if and only if the following conditions are satisfied :
\begin{enumerate}
\item $K$ is abelian over $\bar F$.
\item $\wt D$ is similar to $A\otimes_{\bar F} K$ for some central simple $\bar F$-algebra $A$.
\item $\Gal(K/\bar F)$ embeds into $\Gamma_F/m\Gamma_F$,
where $\Gamma_F$ is the value group of $F$
and $m=\exp \Gal(K/\bar F)$.
\end{enumerate}
Moreover, if (1)--(3) hold,
then $D$ can be found such that
$\ind D=n\ind\wt D$ and $\exp D=\lcm(m,\exp A)$.
\end{varthm}
This theorem is called Lift Theorem because $D$ is in some sense a lift of $\wt D$ over $F$.
The inertial lift property thereby serves two purposes.
First, it guarantees the existence of $D$,
and second, it makes it possible to keep track of the exponent
in the passage from $\wt D$ to $D$ also in the non-Henselian cases.

An important special case is $F=k(t)$ or $F=k(\!(t)\!)$ 
regarded with respect to the $t$-adic valuation, \ie $\bar F=k$.
Since the valuation is discrete, (1) and (3) of the Lift Theorem
require that $K/k$ is cyclic.
If $\Gal(K/k)=\gen{\sigma}$
then (2) is equivalent to that $\sigma$ extends to an automorphism $\wt\sigma$ of $\wt D$.
Moreover, the algebras $\wt D(x;\wt\sigma)$ and $\wt D(\!(x;\wt\sigma)\!)$
are lifts of $\wt D$ in the sense of the Lift Theorem.
As a corollary from the Noncrossed Product Criterion we get the following
crossed product characterization for twisted function fields and twisted Laurent series rings.

\begin{varthm}[Crossed Product Characterization]
The algebras $\wt D(x;\wt\sigma)$ and $\wt D(\!(x;\wt\sigma)\!)$
are crossed products if and only if $\wt D$ contains a maximal subfield 
that is Galois over $k$.
Here, $k$ is the fixed field of the restriction of $\wt\sigma$ to $Z(\wt D)$.
\end{varthm}

If $k$ is a number field,
then both steps of the noncrossed product construction are done theoretically by the above explanations.
If $k=\Q$ for example,
then noncrossed product division algebras $\wt D(x;\wt\sigma)$ and $\wt D(\!(x;\wt\sigma)\!)$ are obtained in this way
for any index $m$ and any exponent $n$
such that $p^2\mid n\mid m\mid\frac{n^2}{p}$ 
for an odd prime $p$, or 
$8\mid n\mid m\mid\frac{n^2}{2}$.
The centres are isomorphic to $\Q(t)$ and $\Q(\!(t)\!)$ respectively.
Moreover, the noncrossed products from \cite{brussel:noncr-prod} are obtained in this way.
In fact, if $\wt D$ is the underlying division algebra of $A\otimes_k K$,
then $\sigma$ extends to a $k$-automorphism $\wt\sigma$ of $\wt D$
and $\wt D(x;\wt\sigma)$ is isomorphic to the underlying division algebra of 
$A\otimes_k (K(t)/k(t),\sigma,t)$.
Analogously  $\wt D(\!(x;\wt\sigma)\!)$ is isomorphic to the underlying division algebra of 
$A\otimes_k (K(\!(t)\!)/k(\!(t)\!),\sigma,t)$.
This establishes a more direct way to look at the noncrossed products from \cite{brussel:noncr-prod},
and their construction is reduced to the construction of $\wt D$ and $\wt\sigma$.
If the construction is started with $\wt D$,
the algebra $A$ is not needed anymore,
hence no underlying division algebra of a tensor product has to be determined.
We give an explicit example of a suitable $\wt D$.
Here $k=\Q$ and $K$ is the cyclic cubic number field $\Q(\zeta+\zeta^{-1})$ for a primitve $7$-th root of unity $\zeta$.
$\wt D$ is a central $K$-division algebra of degree~$3$,
the automorphism $\sigma$ of $K/k$ extends to an automorphism $\wt\sigma$ of $\wt D$,
and $\wt D(x;\wt\sigma)$ and $\wt D(\!(x;\wt\sigma)\!)$ are noncrossed product division algebras of index and exponent $9$.
However, the automorphism $\wt\sigma$ remains unknown,
and the examples $\wt D(x;\wt\sigma)$ and $\wt D(\!(x;\wt\sigma)\!)$ are not fully explicit as desired.

The crucial point in the construction of an explicit example
is the computation of automorphisms of $\wt D$
that extend given automorphisms on the centre.
Since this is too difficult in general,
we have to find a division algebra $\wt D$ that, on the one handside,
has a simple structure so that its automorphisms can be computed.
But on the other handside, $\wt D$ shouldn't contain a maximal subfield that is Galois over $k$,
so its structure may not be ``too simple''.

A class of algebras that are ``simple enough'' are the symbol algebras.
If $K$ is a field containing a primitive $m$-th root of unity $\zeta$
and $a,b$ are elements from $\ug K$,
then the symbol algebra $A=(\frac{a,b}{K,\zeta})$ is the central simple $K$-algebra of degree $m$ that is 
generated over $K$ by elements $x,y$ satisfying $x^m=a,y^m=b$ and $yx=\zeta xy$. 
We find that the computation of an automorphism of $A$ extending a given automorphism $\sigma$ on $K$ with $\sigma(b)=b$
reduces to the solution of the relative norm equation
$\No_{K(y)/K}(\alpha)=\frac{\sigma(a)}{a}$ for $\alpha\in K(y)$.
If $K$ is a number field,
this relative norm equation can be solved with methods from computational algebraic number theory.
In concrete examples the KASH software \cite{kash} can be used.

The question arises if the symbol algebras are ``general enough'' to build noncrossed products from them.
First we stay in the setup that $K/k$ is a cyclic extension of number fields
and we suppose that $K$ contains a primitive $m$-th root of unity
and that $\wt D$ is a symbol algebra of degree $m$.
The arguments that we have taken over from \cite{brussel:noncr-prod}
to show that $\wt D$ does not contain a maximal subfield Galois over $k$
require the absence of a primitive $m$-root of unity.
It turns out that this assumption is necessary.
In fact we prove the
\begin{varthm}[Embedding Theorem]
Let $K/k$ be a cyclic extension of global fields and let $A$ be a finite-dimensional central simple $K$-algebra.
If $A$ is a symbol algebra or a $p$-algebra,
then $A$ contains a strictly maximal subfield Galois over $k$.
\end{varthm}
By a $p$-algebra we mean an algebra of $p$-power degree over a field of characteristic $p$.
The theorem is called Embedding Theorem, because the problem is to embed the cyclic extension $K/k$ of global fields
into a larger Galois extension $M/k$ of a certain degree that splits $A$.
The fact that $M$ splits $A$ is encoded in the local degrees of $M/k$.
Therefore, the theorem can be regarded as a Grunwald-Wang type of theorem.
In \cite{036.15802}, Wang discusses the similar question when a cyclic extension of number fields embeds into a larger cyclic extension of given local degrees.
However, this is too strict to be applied here, 
since we can not expect to find $M/k$ cyclic.
An immediate consequence of the Embedding Theorem and the Crossed Product Characterization is the

\begin{varthm}[Crossed Product Criterion]
Let $\wt D$ be a finite-dimensional division algebra over a global field
and let $\wt\sigma$ be an automorphism of $\wt D$.
If $\wt D$ is a symbol algebra or a $p$-algebra, 
then $\wt D(x;\wt\sigma)$ and $\wt D(\!(x;\wt\sigma)\!)$ are crossed products.
\end{varthm}

A way to bypass the Crossed Product Criterion is to consider non-cyclic abelian extensions $K/k$.
So let $K/k$ be abelian with $\Gal(K/k)=\gen{\sigma_1}\oplus\ldots\oplus\gen{\sigma_r}$.
As lifts in the sense of the Lift Theorem we can no longer take 
twisted function fields and twisted Laurent series rings in one indeterminate.
We have to generalize to iterated twisted function fields
and iterated twisted Laurent series rings in $r$ indeterminates of the form
$\wt D(x_1;\wt\sigma_1)\cdots(x_r;\wt\sigma_r)$ 
and
$\wt D(\!(x_1;\wt\sigma_1)\!)\cdots(\!(x_r;\wt\sigma_r)\!)$
respectively.
Such algebras can be constructed from so-called abelian factor sets in $\wt D$.
This construction is due to Tignol \cite{tignol:gen-cr-prod} and it
generalizes the construction of generic abelian crossed products from
Amitsur and Saltman in \cite{amitsur:central-simple-div-alg}.
The resulting algebras are written
\begin{equation}
  \label{eq:alg}
\wt D(x_1,\ldots,x_r;\wt\sigma_1,\ldots,\wt\sigma_r;u_{ij}) \qt{and}
\wt D(\!(x_1,\ldots,x_r;\wt\sigma_1,\ldots,\wt\sigma_r;u_{ij})\!),
\end{equation}
where the $\wt\sigma_i$ are automorphisms of $\wt D$ extending $\sigma_i$ respectively,
and the $u_{ij}$ are elements from $\ug D$ satisfying certain relations.
The multiplication is subject to the rules
$$x_ia=\wt\sigma_i(a)x_i \qt{and} x_ix_j=u_{ij}x_jx_i \qt{for all $a\in D$, $1\leq i,j\leq r$.} $$
For $r=1$ the algebras (\ref{eq:alg}) coincide with
$\wt D(x;\wt\sigma)$ and $\wt D(\!(x;\wt\sigma)\!)$ respectively.
For $r=2$ the elements $u_{ij}$ are determined by a single element $u$ with $x_2x_1=ux_1x_2$.
Due to Tignol \cite{tignol:gen-cr-prod},
the algebras (\ref{eq:alg}) can be represented as generalized crossed products to show that they are inertially split with residue algebra $\bar D$ and to compute their index and exponent.
The Crossed Product Characterization then holds analogously,
\ie the algebras (\ref{eq:alg}) are crossed products if and only if $\wt D$ contains a maximal subfield that is Galois over the fixed field $k$ of the restrictions of $\wt\sigma_1,\ldots,\wt\sigma_r$ to $Z(\wt D)$.

At first sight it might look more difficult to compute explicit examples of algebras (\ref{eq:alg}),
because instead of extending one automorphism from $K$ to $\wt D$,
an abelian factor set has to be computed.
In the case that $\wt D$ is a field, the abelian factor sets coincide with the abelian $2$-cocycles.
And like it is done by Amitsur and Saltman in \cite{amitsur-saltman:gen-abel-cr-prod} for the abelian $2$-cocycles,
also the abelian factor sets can be described by a few parameters with a few relations.
The parameters are essentially the elements $u_{ij}$,
and it is relatively easy to compute them as long as the
extensions $\wt\sigma_i$ of the automorphisms $\sigma_i$ to $\wt D$ can be computed.
Therefore, if $\wt D$ is a symbol algebra over a number field,
then lifts 
$\wt D(x_1;\wt\sigma_1)\cdots(x_r;\wt\sigma_r)$ 
and 
$\wt D(\!(x_1;\wt\sigma_1)\!)\cdots(\!(x_r;\wt\sigma_r)\!)$
of $\wt D$ in the sense of the Lift Theorem can be computed explicitly.

The following example shows
that if we allow two cyclic factors in $\Gal(K/k)$,
we even find a quaternion algebra $\wt D$ that does not contain a maximal subfield Galois over $k$.
In particular, this shows that in the Embedding Theorem the assumption that $K/k$ is cyclic is necessary.

\begin{varex}[Example]
Let $K=\Q(\sqrt{3},\sqrt{-7})$.
Then $[K:\Q]=4$ and $K/\Q$ is abelian with $\Gal(K/\Q)=\gen{\sigma_1}\times\gen{\sigma_2}$, where
\begin{alignat*}{2}
  \sigma_1(\sqrt{3})&=-\sqrt{3},&\qquad \sigma_1(\sqrt{-7})&=\sqrt{-7}, \\
  \sigma_2(\sqrt{3})&=\sqrt{3},& \sigma_2(\sqrt{-7})&=-\sqrt{-7}.
\end{alignat*}
Let $\wt D$ be the quaternion algebra $D=(\frac{a,b}{K})$ with 
$$ a=3+\sqrt{3}, \qquad b=\frac{-7+\sqrt{-7}}{2}. $$
Then $\wt D$ is a division algebra that does not contain a maximal subfield Galois over $\Q$.
Moreover, there are extensions $\wt\sigma_1,\wt\sigma_2$ of $\sigma_1,\sigma_2$ to $\wt D$ respectively
and an element $u\in\ug{\wt D}$, such that 
$\wt D(x_1,x_2;\wt\sigma_1,\wt\sigma_2;u)$
and $\wt D(\!(x_1,x_2;\wt\sigma_1,\wt\sigma_2;u)\!)$
are noncrossed product division algebras of index and exponent $8$.
The automorphisms $\wt\sigma_1,\wt\sigma_2$ and the element $u\in\ug{\wt D}$ are computed explicitly.
\end{varex}

To prove that $\wt D$ does not contain a maximal subfield Galois over $\Q$
it is first assumed that such a maximal subfield $M$ exists with $\Gal(M/\Q)=G$.
Then $|G|=8$,
and it is shown that all possible groups of order $8$ lead to a contradiction.
The arguments used here also require the absence of roots of unity,
in this case the primitive $4$-th roots of unity, which are not in $\Q$.
But compared to the cyclic case the requirements are relaxed.
This makes it possible that $\wt D$ is a quaternion algebra.

Finally, the Example is modified by replacing $\Q$ with $\Q(\sqrt{37})$
and by introducing a third indeterminate 
to obtain noncrossed product division algebras 
$\wt D(x_1,x_2;\wt\sigma_1,\wt\sigma_2;u)(x_3;\wt\sigma_3)$
and $\wt D(\!(x_1,x_2;\wt\sigma_1,\wt\sigma_2;u)\!)(\!(x_3;\wt\sigma_3)\!)$
of index~$16$ and exponent $8$.

\mbox{}\\
\noindent
Acknowledgements.
Most notably I would like to express my gratitude towards my teacher and advisor Joachim Gr\"ater
for his support and his suggestions throughout the work on this thesis.
I am also indebted to him 
for introducing me to the fields of algebra and number theory 
from the very beginning
and for acquainting me with the topics of this dissertation during my time in Potsdam.
I would further like to thank Darell Haile for the invitation 
to IU Bloomington in 1999
and for his support during my stay there.
Particular thanks go to Florin Nicolae for the many discussions
and his help in reading through large parts of the text and pointing out errors.
Finally, I am indebted to the DAAD for the financial support of my visit to IU Bloomington.

\tableofcontents

\mainmatter

\chapter*{Preliminaries and Notation}

\section{Central simple algebras}

Let $F$ be a field and let $A$ be an $F$-algebra.
Then $Z(A)$ denotes the centre of $A$.
If $A$ is  non-trivial we identify $F$ with $F\cdot 1_A\subseteq Z(A)$,
and $A$ is said to be a central $F$-algebra if $Z(A)=F$.
If $A$ is a simple then $Z(A)$ is a field, and 
$A$ is said to be finite-dimensional if $[A:Z(A)]<\ift$.
We write $\calA(F)$ for the set of all finite-dimensional central simple $F$-algebras and
$\calD(F)$ for the set of all finite-dimensional $F$-division algebras.
For any $A\in\calA(F)$, the dimension $[A:F]$ of $A$ is a square and
the degree of $A\in\calA(F)$ is defined as $\deg A:=\sqrt{[A:F]}$.
By the Wedderburn Structure Theorem,
$A\cong M_n(D)$ for a unique (up to isomorphism) $D\in\calD(F)$ and a unique $n\in\N$.
$D$ is called the underlying division algebra of $A$,
and the index of $A$ is defined as $\ind A:=\deg D$.
Let $\sim$ denote the similarity relation on $\calA(F)$,
\ie $A\sim B$ if and only if $A\cong M_n(D)$ and $B\cong M_m(D)$ for the same $D\in\calD(F)$,
and let $[A]$ denote the similiarity class of $A$.
The tensor product defines a group structure on the set of all similiarity classes by $[A]\cdot[B]:=[A\otimes_F B]$.
This group is called the Brauer group of $F$ and is denoted by $\Br(F)$.
The exponent of $A$, written $\exp A$, is the order of $[A]$ in $\Br(F)$.

Let $K/F$ be any field extension and $A\in\calA(F)$.
We write $A^K$ for the scalar extension $A\otimes_F K$
and $A_K$ for the underlying division algebra of $A^K$.
In particular $A_F$ is the underlying division algebra of $A\in\calA(F)$.
The restriction map is the group homomorphism 
$\res_{K/F}:\Br(F)\to\Br(K)$ given by $[A]\mapsto[A^K]$.
If $B\subseteq A$ is a subalgebra then $C_A(B)$ denotes the centralizer of $B$ in~$A$.

We will frequently deal with central simple algebras over local and global fields,
since they are the source of our examples.
By a \emph{local} field we mean a field that is complete with respect to some discrete (non-archimedian) valuation and has finite residue field. 
By a \emph{global} field we mean an algebraic number field
or a function field in one indeterminate over a finite field.

Let $k$ be a global field,
$\Val(k)$ the set of all non-trivial normalized valuations on $k$
(including archimedian valuations)
and $\V0(k)$ the subset of all non-archimedian valuations of $\Val(k)$.
For any $v\in\Val(k)$ we write $k_v$ for the completion of $k$ with respect to $v$,
and for any $v\in\V0(k)$ we write 
$\OO_v$ for the valuation ring of $v$,
$\PP_v$ for its maximal ideal
and $\bar k_v$ the its residue field $\OO_v/\PP_v$.
For a number field $k$ we write $\OO_k$ for the integral closure of $\Z$ in $k$.

Let $K/k$ be an extension of global fields,
$v\in\Val(k)$ and $w\in\Val(K)$.
We write $w\mid v$ if $w$ extends $v$.
In this case we denote the local degree $[K_w:k_v]$ also by $[K:k]_w$.
If $K/k$ is a Galois extension, 
the completions $K_w$ for the different $w\in\Val(K)$ with $w\mid v$
are all isomorphic,
therefore we can simply denote them by $K_v$.
In particular, the local degrees $[K:k]_w$ are equal for all $w\in\Val(K)$ with $w\mid v$,
and this degree shall simply be written $[K:k]_v$.

We now recall a few of the main results from the theory of central simple algebras over local and global fields that will be of frequent use.
For a reference on this theory 
see Pierce's book \cite[Chapter 17 and 18]{pierce:ass-alg} 
or Reiner's book \cite[Chapter 8 and 9]{reiner:max-orders}.

For a local field $k$ with valuation $v$ the Brauer group is $\Br(k)\cong \Q/\Z$,
and the isomorphism is given by the invariant map as follows.
Let $n\in\N$ and let $K_n/k$ be the unique inertial extension of degree $n$.
Let $\phi\in\Gal(K_n/k)$ be the Frobenius automorphism,
\ie $\phi$ induces the automorphism $x\mapsto x^q$ on $\bar k$,
where $q=|\bar k|$. %
The invariant of the cyclic crossed product $A=(K_n/k,\phi,a)$, $a\in\ug k$,
is defined as $\inv A:=\frac{v(a)}{n}+\Z$.
The mapping $[A]\mapsto\inv A$ then defines an isomorphism $\Br(k)\to\Q/\Z$
(\cf \cite[Theorem 17.10]{pierce:ass-alg}).
In particular, any $A\in\calA(k)$ is cyclic and 
$\exp A=\ind A$.
Under scalar extensions $K/k$ we have the relation
(\cf \cite[Proposition 17.10]{pierce:ass-alg})
\begin{equation}
  \label{eq:loc-scal-ext}
  \inv A^K = [K:F]\inv A.
\end{equation}

Besides that, invariants over $\R$ and $\C$ are defined by
$\inv\R=\inv\C=0$ and $\inv\Ham=\frac{1}{2}+\Z$ for
the real quaternions $\Ham$.

For a global field $k$ with $v\in\Val(k)$ and $A\in\calA(k)$
we write $A_v$ for the local completion $A^{k_v}$ and 
$\inv_v A$ for the local invariant $\inv A_v$.
Then $\inv_v A=0$ for almost all $v\in\Val(k)$.
Under scalar extensions $K/k$ we have
\begin{equation}
  \label{eq:glob-scal-ext}
 \inv_w A^K = [K:F]_w\inv_v A \qt{for any $w\in\Val(K)$ with $w\mid v$,}
\end{equation}
as an immediate consequence of (\ref{eq:loc-scal-ext}).
The global invariant is the map
\[
\inv: \Br(k)\lra \bigoplus_{v\in\Val(k)} \I_v(k), \quad [A]\lms (\inv_v A)_{v\in\Val(k)},
\]
where $\I_v(k)=\Q/\Z$ if $v\in\V0(k)$, $\I_v(k)=\frac{1}{2}\Z/\Z$ is $v$ is real,
and $\I_v(k)=0$ if $v$ is complex.
The Albert-Hasse-Brauer-Noether Theorem (\cf \cite[Proposition~18.4]{pierce:ass-alg})
states that
\begin{equation}
  \label{eq:alb-has-bra-noe}
 A\sim B \text{ if and only if } \inv_v A=\inv_v B \text{ for all $v\in\Val(k)$,} 
\end{equation}
\ie the global invariant map is injective.
It is a consequence from the Grunwald-Wang Theorem and the Albert-Hasse-Brauer-Noether Theorem that any $A\in\calA(k)$ is cyclic and $\ind A=\exp A$
(\cf \cite[Theorem 18.6]{pierce:ass-alg}).
Moreover, we have the formula 
\begin{equation}
  \label{eq:global-index-exp}
\ind A=\exp A=\lcm(\ind A_v)_{v\in\Val(k)}
\end{equation}
(\cf \cite[Corollary 18.6]{pierce:ass-alg}).
Finally, the image of the global invariant map is determined by the following theorem
(\cf \cite[Proposition 18.7b]{pierce:ass-alg}).
\begin{fact}
\label{fact:inv-surj}
If $(a_v)\in\bigoplus_{v\in\Val(k)} \I_v(k)$ such that
$\sum_{v\in\Val(k)} a_v=0$, 
then there is an $A\in\calA(k)$ with $\inv_v A=a_v$ for all $v\in\Val(k)$.
\end{fact}
The results so far can be summarized by the exact sequence (\cf \cite[Theorem 18.5]{pierce:ass-alg})
\[ 1 \lra \Br(k) \lra \bigoplus_{v\in\Val(k)} \I_v(k) \lra \Q/\Z \lra 1, \]
where the map $\bigoplus_{v\in\Val(k)} \I_v(k) \lra \Q/\Z$ is the sum of the local invariants.

\chapter{Valued Division Algebras}

\section{Noncommutative valuations}
\label{sec:def-noncom-val}

We define valuations on division rings following Schilling.
For a detailed reference see his book \cite{schilling:th-of-val}.
A ring $D$ is called a division ring if the non-zero elements $\ug D$ form a multiplicative group.

\begin{defn}
\label{def:non-com-val-fin}
\index{valuation!noncommutative}
  Let $D$ be a division ring, $\Gamma$ a totally ordered abelian group 
  (additively written), and $\widehat\Gamma = \Gamma\cup\set{\ift}$ with
  $\gamma+\infty=\infty+\gamma=\infty$ and $\gamma<\infty$ for all $\gamma\in\Gamma$.
  A map $v: D \lra \Hat\Gamma$ is called a \emph{valuation} on $D$ 
  if the following conditions are satisfied for all $x,y\in D$ :
  \begin{enumerate}
  \item[(V1)] $v(x)=\infty \Leftrightarrow x=0$.
    \label{item:def:non-com-val-fin-1}
  \item[(V2)] $v(xy)=v(x)+v(y)$ for all $x,y\in D$.
    \label{item:def:non-com-val-fin-2}
  \item[(V3)] $v(x+y) \geq \min\{v(x),v(y)\}$.
    \label{item:def:non-com-val-fin-3}
  \end{enumerate}
\end{defn}

Note that Schilling in his definition of a valuation does not require $\Gamma$ to be abelian.
But in the case that $D$ is finite-dimensional, 
which is our main interest here,
his definition implies that $\Gamma$ is abelian.
If $D$ is a field then Definition \ref{def:non-com-val-fin} coincides with the definition of a Krull valuation.

Associated to a valuation $v$ on $D$ we have
the \emph{valuation ring}
$B_v = \sett{x\in D}{v(x)\geq 0}$,
which is a subring of $D$ with the unique maximal two-sided ideal
$M_v = \sett{x\in D}{v(x)> 0}$ 
and unit group
$U_v = \sett{x\in D}{v(x)= 0}$.
Moreover we have the \emph{residue ring} or \emph{residue algebra}
$\bar D_v=\bar B_v=B_v/M_v$, which is a division ring,
the canonical \emph{residue map}
$\pi_v:B_v\to\bar B_v,x\mapsto \bar x=x+M_v$,
and the \emph{value group} $\Gamma_v=v(D^\times)$.
If a valuation $v$ on $D$ is fixed,
we index these objects also by $D$
($M_D$, $\Gamma_D$, etc.) and write $\bar D$ for $\bar B_v$.
For the valuation ring however we shall use the more common notation $V_D$.
If we say that $D$ is a \emph{valued division ring} or \emph{valued division algebra},
we mean that some valuation $v$ on $D$ is fixed.
A valuation $v$ is called \emph{discrete} if $\Gamma_v$ is order isomorphic to the additive group of integers.
Note that throughout this chapter by a \emph{valued field} we mean a field with a Krull valuation.

If $E$ is a sub-division ring of the valued division ring $D$ then $v$ clearly restricts to the valuation $v|_E$ on $E$.
Conversely we say that $v$ extends $v|_E$.
We usually denote the restriction $v|_E$ also by $v$ 
and write $V_E$, $M_E$, $U_E$, $\bar E$ and $\Gamma_E$ for the corresponding induced objects.
Then $V_E=V_D\cap E$.
The value group $\Gamma_E$ is a subgroup of $\Gamma_D$ 
and the group index $|\Gamma_D : \Gamma_E|$ is called the \emph{ramification index} of $D$ over $E$.
The factor group $\Gamma_D/\Gamma_E$ is called the \emph{relative value group} of $D$ over $E$.
Furthermore $\bar E$ can be viewed as a sub-division ring of $\bar D$
and $[\bar D : \bar E]$ is called the \emph{residue degree} of $D$ over $E$.
If $[D:E]$ is finite so are $|\Gamma_D : \Gamma_E|$ and $[\bar D : \bar E]$ 
and the fundamental inequality
\begin{equation}
  \label{eq:noc-val-fund-ineq}
  [\bar D : \bar E] \cdot |\Gamma_D : \Gamma_E| \leq [D:E]
\end{equation}
holds (cf. \cite[Chapter 1, Lemma 18]{schilling:th-of-val}).
We say $D$ is \emph{defectless} over $E$ if $[D:E]<\infty$ and equality holds in (\ref{eq:noc-val-fund-ineq}),
$D$ is \emph{immediate} over $E$ if $[\bar D:\bar E]=|\Gamma_D:\Gamma_E|=1$,
and $D$ is \emph{totally ramified} over $E$ if $|\Gamma_D:\Gamma_E|=[D:E]$.
If $E$ is a subfield of $Z(D)$, we say $D$ is \emph{inertial} over $E$ if $[D:E]=[\bar D:\bar E]<\infty$ and $Z(\bar D)$ is separable over $\bar E$.
The inequality (\ref{eq:noc-val-fund-ineq}) implies

\begin{remark}
  \label{rem:defectless-interm-div-ring}
  If $D$ is defectless over a sub-division ring $E$ of $D$, 
  then any intermediate division ring $E'$, $E\subseteq E' \subseteq D$, is also defectless over $E$.
\end{remark}

\begin{ex}
  \label{ex:t-adic-val}
Let $D$ be a division ring and let $t$ be a commutative indeterminate over $D$.
Let $D(t)$ be the rational function field over $D$,
\ie the ring of central quotients of the polynomial ring $D[t]$,
and let $D(\!(t)\!)$ be the Laurent series ring over $D$.
Both $D(t)$ and $D(\!(t)\!)$ are division rings.
Consider $D(t)$ as a subring of $D(\!(t)\!)$.
The $t$-adic valuation $v_t$ on $D(\!(t)\!)$ is defined by
$$ v_t(\sum_{i\geq k} d_i t^i) := \min\sett{i\in\Z}{d_i\neq 0}. $$
Clearly $\Gamma_{v_t}=\Z$, \ie $v_t$ is discrete.
Moreover
$ B_{v_t}=\sett{\sum_{i\geq 0} d_i t^i}{d_i\in D}, $
$ M_{v_t}=\sett{\sum_{i\geq 1} d_i t^i}{d_i\in D}, $
and $\bar B_{v_t}\cong D.$
The restriction of $v_t$ to $D(t)$, which will also be denoted by $v_t$,
has the same residue algebra and value group.
\end{ex}

Now, let $D$ be a finite-dimensional valued division ring and $F=Z(D)$.
We shortly write 
$e_D=|\Gamma_D : \Gamma_F|$
and $f_D=[\bar D : \bar F]$,
then (\ref{eq:noc-val-fund-ineq}) implies that $e_D$ and $f_D$ are finite.
We simply say that $D$ is \emph{defectless} %
if $D$ is defectless %
over the centre $F$.
It follows from (V$2$) that the valuation ring $V_D$ and its maximal ideal $M_D$ are invariant under all inner automorphisms of $D$,
\ie $aV_Da^{-1}=V_D$ and $aM_Da^{-1}=M_D$ for all $a\in D^\times$.
Thus the inner automorphism $\iota_a:x\mapsto axa^{-1}$ of $D$ defined by $a\in D^\times$ induces an automorphism $\bar\iota_a$ on the residue division ring $\bar D$ via
\[ 
  \bar\iota_a : \bar D \lra \bar D, \quad \bar x \lms \bar{axa^{-1}} . 
\]

Clearly, since $\iota_a$ fixes $F$ pointwise,
$\bar\iota_a$ fixes $\bar F$ pointwise,
thus $\bar\iota_a\in\Aut_{\bar F}\bar D$.
Furthermore $\bar\iota_a$ restricts to an automorphism of $Z(\bar D)$ 
because any automorphism of $\bar D$ maps the centre $Z(\bar D)$ onto itself.
Thus we have a map
\[
  \phi_D : D^\times \lra \Gal(Z(\bar D)/\bar F), \quad
  a \lms \bar\iota_a|_{Z(\bar D)} .
\]
It follows directly from the definition of $\bar\iota_a$ that $\phi_D$ is a group homomorphism.
Obviously -- also by definition -- $\bar\iota_a$ is the identity on $\bar D$ for all $a\in F^\times$.
Since for $a\in U_D$ we have $a, a^{-1}\in V_D$, 
it follows 
$\bar\iota_a(\bar x)=\bar{axa^{-1}}=\bar a\,\bar x\,\bar a^{-1}$,
\ie $\bar\iota_a$ is an inner automorphism of $\bar D$
and therefore fixes $Z(\bar D)$ pointwise.
Thus $U_D F^\times$ lies in the kernel of $\phi_D$,
so $\phi_D$ induces a homomorphism 
$\bar\phi_D : D^\times/{U_D F^\times} \lra \Gal(Z(\bar D)/\bar F)$.
By connecting $\bar\phi_D$ with the canonical isomorphism 
$D^\times/{U_D F^\times} \lra \Gamma_D/\Gamma_{F}$
we finally get the fundamental homomorphism
\begin{equation}
  \label{eq:theta}
  \theta_D : \Gamma_D/\Gamma_{F} \lra \Gal(Z(\bar D)/\bar F)
\end{equation}
with
\[
\theta_D(v(a)+\Gamma_{F}) = \phi_D(a) = \bar\iota_a|_{Z(\bar D)} \qt{for all $a\in D^\times$.}
\] 
This homomorphism will be of frequent use in the following.

\begin{prop}
  \label{prop:theta-surj}
  For any finite-dimensional valued division algebra \linebreak$D\in\calD(F)$,
  the homomorphism $\theta_D$ as defined in (\ref{eq:theta}) is surjective. 
  The field extension $Z(\bar D)/\bar F$ is normal,
  and if $Z(\bar D)$ is separable over $\bar F$, 
  then $Z(\bar D)$ is abelian Galois over $\bar F$.
\end{prop}
\begin{proof}
See \cite[Proposition 1.7]{jacob-wadsworth:div-alg-hensel-field}.  
\end{proof}
Since $\theta_D$ is surjective, 
every $\sigma\in\Gal(Z(\bar D)/\bar F)$ is the restriction of $\bar\iota_a$ for some $a\in D^\times$.
Hence, we get

\begin{cor}
  \label{cor:bar-F-normal}
  For any valued division algebra $D\in\calD(F)$, 
  the residue division ring $\bar D$ is \emph{$\bar F$-normal}, 
  \ie every $\sigma\in\Gal(Z(\bar D)/\bar F)$ can be extended to some $\tilde\sigma\in\Aut_{\bar F}(\bar D)$.
\end{cor}

In general, for a field $F$ with valuation $v$ and $D\in\calD(F)$,
$v$ does not necessarily extend to a valuation on $D$.
If the valuation $v$ is fixed we define
\[ \VBr(F):=\{[D]\,|\, D\in\calD(F) \text{ and $v$ extends to a valuation on $D$}\}\subseteq \Br(F), \]
which in general is not a subgroup of $\Br(F)$.

\begin{remark}
If we write $[D]\in\VBr(F)$ and $D$ is not further determined
we always mean that $D\in\calD(F)$ and $[D]\in\VBr(F)$.
Conversely, if $P$ is a predicate of a valued division algebra,
\eg ``inertial'', ``defectless'' or some other predicate defined later on,
and we say that ``$D$ is $P$'' for some $D\in\calD(F)$,
then we always mean that $[D]\in\VBr(F)$ and $D$ is $P$.
\end{remark}

It is known from \cite{wadsworth:ext-val} that
\begin{fact}
  \label{eq:unique-ext}
$[D]\in\VBr(F)$ if and only if
$v$ extends uniquely to every subfield $K\subseteq D$ with $F\subseteq K$.
In this case the extension of $v$ to $D$ is unique.
\end{fact}
This result implies that 
\begin{fact}
\label{eq:ext-val-hensel}
if $v$ is Henselian, then $\VBr(F)=\Br(F)$.
\end{fact}

In \cite{morandi:henselization-div-alg} Morandi proved the following useful criterion 
for when a tensor product of valued division algebras is again a valued division algebra.

\begin{theorem}
  \label{thm:morandi}
  Let $D$ and $E$ be $F$-division algebras with 
  valuations $v$ on $D$ and $w$ on $E$ such that $v|_F=w|_F$.
  If
  \begin{enumerate}
  \item $[D:F]<\ift$ and $D$ is defectless over $F$ with respect to $v$,
  \item $\bar D\otimes_{\bar F}\bar E$ is a division algebra,
  \item $\Gamma_D\cap\Gamma_E=\Gamma_F$,
    \footnote{Let $\Delta_F$ (\resp $\Delta_E$) be the divisible hull of the torsion-free abelian group $\Gamma_F$ (\resp $\Gamma_E$), 
\ie $\Delta_F\cong\Gamma_F\otimes_\Z\Q$ (\resp $\Delta_E\cong\Gamma_E\otimes_\Z\Q$).
Since $|\Gamma_D:\Gamma_F|$ is finite, 
the injection $\Gamma_F\to\Delta_F$ extends uniquely to an order preserving injection $\Gamma_D\to\Delta_F\subseteq\Delta_E$. 
The expressions $\Gamma_D\cap\Gamma_E$ and $\Gamma_D+\Gamma_E$ are then computed in $\Delta_E$.}
  \end{enumerate}
  then $D\otimes_F E$ is a division algebra with a valuation $u$ extending $v$ and $w$.
  Furthermore $\bar{D\otimes_F E}\cong\bar D\otimes_{\bar F}\bar E$ and $\Gamma_{D\otimes_F E}=\Gamma_D+\Gamma_E$.
\end{theorem}

\begin{defn}
  \label{def:henselization}
Let $F$ be a field with valuation $v$
and let $F^h$ denote the Henselization of $F$ with respect to $v$.
For any $D\in\calD(F)$ we call $D^h:=D^{F^h}=D\otimes F^h$ the \emph{Henselization} of $D$ with respect to $v$.
\end{defn}

In \cite{morandi:henselization-div-alg} Theorem \ref{thm:morandi} leads to

\begin{theorem}
  \label{thm:div-alg-henselization}
  Let $F$ be a field with valuation $v$ and let $D\in\calD(F)$.
  Then $[D]\in\VBr(F)$ if and only if $D^h$ is a division algebra.
  In this case $D^h$ is immediate over $D$, \ie
  $\bar{D^h}=\bar D$ and $\Gamma_{D^h}=\Gamma_D$.
\end{theorem}

\begin{remark}
\label{rem:defectless-henselization}
Let $F$ be a valued field and let $[D]\in\VBr(F)$.
Then $D$ is defectless if and only if $D^h$ is defectless.
\end{remark}

\section{Inertial division algebras and inertial lifts}

Throughout this section let $F$ be a field with fixed valuation $v$.

\subsection{Inertial division algebras and scalar extensions}

\begin{defn}
\label{def:inertial}
\index{inertial}
Let $D$ be a finite-dimensional valued division ring 
and let $E$ be a subfield of $Z(D)$.
We say that $D$ is \emph{inertial} over $E$ if 
$[D:E]=[\bar D:\bar E]<\infty$ and $Z(\bar D)$ is separable over $\bar E$.
We simply say that $D$ is inertial if $D$ is inertial over $Z(D)$.
\end{defn}

\begin{remark}
  \label{rem:inertial} 
Let $[D]\in\VBr(F)$.

(1) If $D$ is inertial, then $Z(\bar D)=\bar F$.

(2) $D$ is inertial if and only if $\ind D=\ind\bar D$.
\end{remark}
\begin{proof}
$(1)$
By Proposition \ref{prop:theta-surj}, $Z(\bar D)/\bar F$ is a Galois extension
and the group homomorphism $\theta_D:\Gamma_D/\Gamma_F\to\Gal(Z(\bar D)/\bar F)$ is surjective.
But by (\ref{eq:noc-val-fund-ineq}), $\Gamma_D=\Gamma_F$,
thus $Z(\bar D)=\bar F$. \\
$(2)$
If $D$ is inertial then $(1)$ shows $\ind D=\ind\bar D$.
Conversely, if $\ind D=\ind\bar D$, then (\ref{eq:noc-val-fund-ineq}) shows
$[D:F]=[\bar D:Z(\bar D)]\leq[\bar D:\bar F]\leq[D:F]$,
hence $[D:F]=[\bar D:\bar F]$ and $Z(\bar D)=\bar F$,
\ie $D$ is inertial.
\end{proof}

Define
\begin{gather*}
\IBr(F) := \sett{[D]\in\VBr(F)}{D \text{ is inertial}} \subseteq \VBr(F).
\end{gather*}
Note that in general $\IBr(F)$ is not a subgroup of $\Br(F)$.
By Remark \ref{rem:inertial} we have an index-preserving map
\[ \beta_F: \IBr(F)\lra\Br(\bar F), \quad [D]\lms[\bar D] .\]
This map $\beta_F$ can be viewed in a more general context
if we consider the Brauer group $\Br(V_F)$ of the valuation ring $V_F$ of $F$,
which is the group of equivalence classes of Azumaya algebras over $V_F$
(see \cite[Chapter II, \S 5]{demeyer-ingraham} for a definiton),
and the canonical group homomorphisms
\begin{gather*}
   \alpha : \Br(V_F)\lra\Br(F), \quad [A]\lms[A\otimes_{V_F} F]
\intertext{and}
   \beta : \Br(V_F)\lra\Br(\bar F), \quad [A]\lms[A/M_F A].
\end{gather*}
It is known that $\alpha$ is injective (cf. \cite[Lemma 1.2]{saltman:centre-of-generic} or \cite[Proposition~2.5]{jacob-wadsworth:div-alg-hensel-field}).
If $D\in\calD(F)$ is inertial then $V_D$ is an Azumaya algebra over $V_F$
(cf. \cite[Example 2.4]{jacob-wadsworth:div-alg-hensel-field}),
\ie $[V_D]\in\Br(V_F)$ and $\alpha([V_D])=[D]$.
This shows that $\IBr(F)\subseteq\im(\alpha)$
and $\beta_F=\beta\circ\alpha^{-1}|_{\IBr(F)}$.
In the Henselian case we get

\begin{theorem}
  \label{thm:alpha-beta-maps}
If $F$ is Henselian, then $\im(\alpha)=\IBr(F)$,
\ie $\IBr(F)$ is a subgroup of $\Br(F)$, and 
$\alpha:\Br(V_F)\to\IBr(F)$ is an isomorphism.
Moreover, also $\beta$ and $\beta_F=\beta\circ\alpha^{-1}$
are isomorphisms.
\end{theorem}
\begin{proof}
For the map $\alpha$ see \cite[Theorem 2.8]{jacob-wadsworth:div-alg-hensel-field}
and for $\beta$ see \cite[Theorem 31]{azumaya:maximally-central}.
\end{proof}

We now look at inertial scalar extensions, first in the case that $F$ is Henselian.
Let $K/F$ be any algebraic inertial field extension and let $w$ be the unique extension of $v$ to~$K$.
Then $K$ is Henselian with respect to $w$, 
hence $[D_K]\in\VBr(K)$ for any $D\in\calD(F)$.
From \cite[Theorem 3.1 and Remark 3.4]{jacob-wadsworth:div-alg-hensel-field} we have

\begin{theorem}
  \label{thm:inert-scal-ext-hensel}
Let $F$ be Henselian, $D\in\calD(F)$ and $K/F$ an algebraic inertial field extension.
Then
\begin{enumerate}
\item $Z(\bar{D_K})\cong Z(\bar D)\cdot\bar K$,
\item $\bar{D_K}\cong \bar D_{Z(\bar D)\cdot\bar K}$,
\item $\Gamma_{D_K}\subseteq \Gamma_D$ and $|\Gamma_D:\Gamma_{D_K}|=[Z(\bar D)\cap\bar K:\bar F]$,
\item $D$ defectless $\impl$ $D_K$ defectless.
\end{enumerate}
\end{theorem}

In the special case that $D$ is inertial this yields

\begin{cor}
  \label{cor:inert-scal-ext-hensel}
Let $F$ be Henselian, $D\in\calD(F)$ be inertial and $K/F$ be an algebraic inertial field extension.
Then
\begin{enumerate}
\item $Z(\bar{D_K})\cong \bar K$,
\item $\bar{D_K}\cong \bar D_{\bar K}$,
\item $\Gamma_{D_K}=\Gamma_F$,
\item $[D_K]\in\IBr(K)$.
\end{enumerate}
\end{cor}

If $F$ in Corollary \ref{cor:inert-scal-ext-hensel} is not Henselian,
then we do not necessarily have $[D_K]\in\VBr(K)$.
The following Proposition characterizes when $[D_K]\in\VBr(K)$.

\begin{prop}
  \label{prop:inert-scalar-ext-valued}
  Let $F$ be a field with valuation $v$
  and let $K/F$ be an algebraic inertial field extension.
  If $D\in\calD(F)$ is inertial
  then the following are equivalent~:
  \begin{enumerate}
  \item $[D_K]\in\VBr(K)$,
  \item $[D_K]\in\IBr(K)$,
  \item $\ind D_K=\ind\bar D_{\bar K}$.
  \end{enumerate}
  If $(1)$--$(3)$ hold, then 
  \begin{enumerate}
    \setcounter{enumi}{3}
  \item  $\bar{D_K}\cong\bar D_{\bar K}$.
  \end{enumerate}
\end{prop}
\begin{proof}
  Since $v$ extends to a valuation on $D$, $D^h$ is division algebra and
  $(D_K)^h \sim D\otimes_F K\otimes_K K^h \cong D\otimes_F K^h 
  \cong D\otimes_F F^h \otimes_{F^h} K^h \sim {D^h}_{K^h}$.
  By Corollary \ref{cor:inert-scal-ext-hensel},
  ${D^h}_{K^h}$ is inertial and $\bar{{D^h}_{K^h}}\cong\bar D_{\bar K}$, \ie
  $\ind\bar D_{\bar K}=\ind D^h_{K^h}$ by Remark \ref{rem:inertial} (2).
  This shows
  \begin{equation}
    \label{eq:2}
    \ind\bar D_{\bar K}=\ind (D_K)^h \leq \ind D_K . 
  \end{equation}
  Therefore, $[D_K]\in\VBr(K)$ iff $(D_K)^h$ is a division algebra,
  iff $\ind D_K=\ind (D_K)^h$, iff equality holds in (\ref{eq:2}).
  This shows $(1)\gdw(3)$.

  Suppose $(1)$ and $(3)$ hold.
  Then $\bar{D_K}=\bar{(D_K)^h}\cong\bar D_{\bar K}$
  and $\ind D_K=\ind\bar{D_K}$, \ie $D_K$ is inertial.
  This shows $(4)$ and $(2)$.
  The Proposition is proved since $(2)\impl(1)$ is trivial.
\end{proof}

\subsection{The inertial lift property}

Like before let $F$ be a field with fixed valuation $v$.

\begin{defn}
\label{defn:inertial-lift}
\index{inertial lift}
Let $\wt D$ be an $\bar F$-division algebra such that $Z(\wt D)$ is separable over $\bar F$.
We say that an $F$-division algebra $D$ is an \emph{inertial lift} of $\wt D$ over $F$ 
if $D$ is inertial over $F$ and $\bar D\cong\wt D$. 
\end{defn}

\begin{remark}
  \label{rem:inert-lift}
Let $\wt D\in\calD(\bar F)$ and let $D\in\calD(F)$ be inertial.
Then $D$ is an inertial lift of $\wt D$ if and only if $\beta_F([D])=[\wt D]$.
\end{remark}

\begin{defn}
\label{def:ILP}
\index{inertial lift!property (ILP)}
  Let $F$ be a valued field.
  Suppose $\IBr(F)$ contains a subgroup $X$ of $\Br(F)$ that maps onto $\Br(\bar F)$ under $\beta_F$.
  We say that $F$ has the \emph{inertial lift property} if
  for any finite Galois extension $\wt K/\bar F$,
  there exists an inertial Galois lift $K$ of $\wt K$ over $F$,
  \ie an inertial lift $K$ of $\wt K$ over $F$ with $K/F$ Galois,
  such that 
  \begin{equation}
    \label{eq:ILP}
\parbox{11cm}{$D_K$ satisfies the conditions of Proposition \ref{prop:inert-scalar-ext-valued} for any $D\in\calD(F)$ with $[D]\in X$.}    
  \end{equation}
\end{defn}

\begin{remark}
If $F$  has the inertial lift property, then :

(1) The map $\beta_F|_X: X\lra\Br(\bar F)$ is a group isomorphism.

(2) The group homomorphism $\res_{F^h/F}:\Br(F)\lra\Br(F^h)$ 
restricts to an isomorphism $\res_{F^h/F}|_X: X\lra\IBr(F^h)$.

(3) For $K/F$ as in Definition \ref{def:ILP} the following diagram commutes.
    $$  \begin{CD}
      X @>\beta_F>> \Br(\bar F) \\
      @VV\res V @VV\res V \\
      \IBr(K) @>\beta_K>> \Br(\bar K) 
    \end{CD} $$

(4) $\beta:\Br(V_F)\lra\Br(\bar F)$ is a split surjection.

(5) If $D\in\calD(F)$ with $[D]\in X$ and $\wt K$ is a Galois maximal subfield of $\bar D$,
then $D$ contains an inertial Galois lift $K$ of $\wt K$ over $F$.
\end{remark}
\begin{proof}
(1) The surjective map $\beta_F|_X$ is a group homomorphism, since
$X$ is a group and $\beta_F|_X=\beta\circ\alpha^{-1}|_X$.
As an index preserving group homomorphism, $\beta_F|_X$ is also injective. 
\\
(2) Since $X\subseteq\VBr(F)$, $\res_{F^h/F}|_X$ is an index preserving group homomorphism, hence it is injective.
To check that the image is $\IBr(F^h)$ let $D\in\calD(F^h)$ with $[D]\in\IBr(F^h)$.
By the inertial lift property there is $D_0\in\calD(F)$ with $[D_0]\in X$
and $\bar D_0\cong\bar D$.
Then $D_0^h$ is a division algebra with $[D_0^h]\in\IBr(F^h)$ and $\bar{D_0^h}\cong\bar D$.
Since $\beta_{F^h}$ is an isomorphism (\cf Theorem~\ref{thm:alpha-beta-maps}),
it follows that $D_0^h\cong D$. 
\\
(3) This follows from the implication $(2)\impl(4)$ in Proposition \ref{prop:inert-scalar-ext-valued}. 
\\
(4) The map $\alpha^{-1}\circ(\beta_F|_X)^{-1}$ is a splitting homomorphism. 
\\
(5) Suppose $D\in\calD(F)$ with $[D]\in X$ and $\wt K$ is a Galois maximal subfield of $\bar D$.
Let $K$ be an inertial Galois lift of $\tilde K$ over $F$ with the property (\ref{eq:ILP}).
Then 
$\ind D_K=\ind \bar D_{\wt K}=1$, \ie $K$ splits $D$.
Moreover $[K:F]=[\wt K:\bar F]=\ind\bar D=\ind D$, 
hence $K$ is maximal subfield of~$D$.
\end{proof}

To verify the inertial lift property for a Henselian field $F$,
first recall that  the map $K\mapsto\bar K$ 
gives a 1-1 correspondence between the inertial extensions $K/F$ and the separable extensions $\wt K/\bar F$.
Under this correspondence $K/F$ is Galois if and only if $\bar K/\bar F$ is Galois 
(cf. \cite[\S 19]{endler:val-th}).

\begin{theorem}
  \label{thm:hensel-ILP}
If $F$ is Henselian, then $F$ has the inertial lift property.
\end{theorem}
\begin{proof}
  If $F$ is Henselian, then $\IBr(F)$ is a subgroup of $\Br(F)$
  and $\beta_F$ is an isomorphism by Theorem \ref{thm:alpha-beta-maps}.
  Hence choose $X=\IBr(F)$.
  Let $\wt K/\bar F$ be any finite Galois extension.
  Since $\wt K/\bar F$ is separable there is a unique inertial lift 
  $K$ of $\wt K$ over $F$ and $K/F$ is Galois.
  Since $K/F$ is finite, $K$ is also Henselian, 
  so $D_K$ satisfies condition $(1)$ of Proposition \ref{prop:inert-scalar-ext-valued} for any $D\in\calD(F)$.
\end{proof}

If $F$ is Henselian we also have the following lift property.

\begin{theorem}
  \label{thm:hensel-ILSP}
  Let $F$ be Henselian and let $D$ be a division algebra over~$F$ with $[D:F]<\ift$.
  If $\wt E$ is an $\bar F$-subalgebra of $\bar D$ with
  $Z(\wt E)$ separable over $\bar F$,
  then $D$ contains an inertial lift $E$ of $\wt E$ over~$F$.
\end{theorem}

\begin{proof}
See \cite[Theorem 2.9]{jacob-wadsworth:div-alg-hensel-field}.
\end{proof}

We now verify the inertial lift property for another class of fields,
which are not necessarily Henselian.
This will later be used to get examples over non-Henselian fields.

\begin{lemma}
  \label{lem:bar-F-embed-inert-lift}
  Suppose that $V_F$ contains a field that maps isomorphically onto $\bar F$ under $\pi_F$.
  Then for any division algebra $\wt D$ over $\bar F$ with 
  $[\wt D:\bar F]<\ift$ and $Z(\wt D)$ separable over $\bar F$,
  $D:=\wt D\otimes_{\bar F} F$ is a division algebra that is an inertial lift of $\wt D$ over $F$.
\end{lemma}
\begin{proof}
  If we identify $\bar F$ with its isomorphic image $\pi_F^{-1}(\bar F)$ in $V_F$,
  then $\bar F^\times \subseteq U_F$. 
  Hence, $v|_{\bar F}$ is the trivial valuation.
  Let $w$ be the trivial valuation on $\wt D$.
  Then obviously $v|_{\bar F}=w|_{\bar F}$ and the conditions (1)--(3) of Theorem \ref{thm:morandi} are satisfied.
  Therefore,
  $D:=\wt D\otimes_{\bar F} F$ is a division algebra with valuation $u$ 
  such that $u|_F=v$ and $\bar D=\wt D$.
  Since $[D:F]=[\wt D:\bar F]$,
  $D$ is an inertial lift of $\wt D$ over $F$.
\end{proof}

\begin{theorem}
  \label{thm:bar-F-embed-ILP}
  Suppose that $V_F$ contains a field that maps isomorphically onto $\bar F$ under $\pi_F$.
  Then $F$ has the inertial lift property.
\end{theorem}
\begin{proof}
  Lemma \ref{lem:bar-F-embed-inert-lift} shows that $\beta_F\circ\res_{F/\bar F}=\id_{\Br(\bar F)}$.
  Hence, choose $X=\res_{F/\bar F}(\Br(\bar F))$.
  Let $\wt K/\bar F$ be a finite Galois extension.
  Then $K:=\wt K\otimes_{\bar F} F$ is Galois over $F$,
  and by Lemma \ref{lem:bar-F-embed-inert-lift}, $K$ is an inertial lift of $\wt K$ over $F$.
  For any $D\in\calD(F)$ with $[D]\in X$,
  by choice of $X$, $D=\bar D\otimes_{\bar F}F$.
  Hence 
  $D_K\sim \bar D\otimes_{\bar F}F\otimes_F K\cong\bar D\otimes_{\bar F}\wt K\otimes_{\bar F}F\sim\bar D_{\wt K}\otimes_{\bar F} F$.
  Therefore, by Lemma~\ref{lem:bar-F-embed-inert-lift},
  $D_K$ is an inertial lift of $\bar D_{\wt K}$ over $F$,
  hence $[D_K]\in\IBr(K)$.
\end{proof}

\begin{ex}
The rational function field $k(t)$ and the Laurent series field $k(\!(t)\!)$ have the inertial lift property with respect to the $t$-adic valuation $v_t$.
More generally, let $F=k(t_1,\ldots,t_r)$ or $F=k(\!(t_1,\ldots,t_r)\!)$
and let $v$ be the composite valuation of the $t_i$-adic valuations on $F$.
Then $F$ has the inertial lift property with respect to $v$.
\end{ex}

\section{Nicely semiramified division algebras}

Throughout this section let $F$ be a field with fixed valuation $v$.
We shall use the following definitions from \cite[\S 4]{jacob-wadsworth:div-alg-hensel-field}.

\begin{defn}
\label{def:TTRT}
Let $K/F$ be a finite field extension.
We say that $K$ is \emph{totally ramified of radical type} over $F$ if
$v$ extends (uniquely) to a valuation $w$ on $K$ such that
$K$ is totally ramified over $F$ and
there is a subgroup of $K^\times/F^\times$ that maps isomorphically onto $\Gamma_K/\Gamma_F$ under the map
$\bar w: K^\times/F^\times\to\Gamma_K/\Gamma_F, xF^\times\mapsto w(x)+\Gamma_F$.
\end{defn}

\begin{defn}
  \label{def:nicely-semiramified}
A  division algebra $D\in\calD(F)$ with $[D]\in\VBr(F)$
is said to be \emph{nicely semiramified}
if $D$ has a maximal subfield $L$ that is inertial over $F$,
and another maximal subfield $K$ that is totally ramified of radical type over $F$.
\end{defn}

\begin{remark}
\label{rem:NSR}
Let $D\in\calD(F)$ be nicely semiramified with inertial maximal subfield $L$
and maximal subfield $K$ that is totally ramified of radical type over $F$. 
\begin{enumerate}
\item $\bar D$ is a field, $\bar D=\bar L$, 
  $\bar D/\bar F $ is abelian and $\Gamma_D=\Gamma_K$. 
\item $\Gal(\bar D/\bar F)$ embeds into $\Gamma_F/m\Gamma_F$,
where $m=\exp\Gal(\bar D/\bar F)$.
\item If $F$ is Henselian, then $L/F$ is abelian.
\item $D^h$ is nicely semiramified.
\end{enumerate}
\end{remark}
\begin{proof}
$(1)$
Since $L$ is inertial and $K$ is totally ramified over $F$,
it follows from the fundamental inequality (\ref{eq:noc-val-fund-ineq}) that
$[\bar D:\bar F]=|\Gamma_D:\Gamma_F|=\ind D$.
Therefore, $\bar D=\bar L$ and $\Gamma_D=\Gamma_K$,
hence $\bar D$ is a field.
Since $L/F$ is inertial, $\bar D/\bar F$ is separable,
hence abelian by Proposition \ref{prop:theta-surj}.
\\
$(2)$ The surjective group homomorphism
$\theta_D:\Gamma_D/\Gamma_F\lra\Gal(\bar D/\bar F)$ is an isomorphism because $|\Gamma_D:\Gamma_F|=[\bar D:\bar F]$,
hence $\Gal(\bar D/\bar F)\cong\Gamma_D/\Gamma_F$.
Let $m=\exp\Gal(\bar D/\bar F)$.
The map $\Gamma_D\to\Gamma_F, \gamma\mapsto m\gamma$ induces an embedding $\Gamma_D/\Gamma_F\hookrightarrow\Gamma_F/m\Gamma_F$.
\\
$(3)$
If $F$ is Henselian, then $\bar L/\bar F$ Galois implies that $L/F$ is Galois
with $\Gal(L/F)\cong\Gal(\bar L/\bar F)$, hence $L/F$ is abelian.
\\
$(4)$
$L^h$ and $K^h$ are maximal subfields of the division algebra $D^h$.
Since $L^h$ and $K^h$ are immediate extensions of $L$ and $K$ respectively,
$L^h$ is inertial and $K^h$ is totally ramified of radical type over $F^h$.
Thus, $D^h$ is nicely semiramified.
\end{proof}

Recall the following important example from \cite[Example 4.3]{jacob-wadsworth:div-alg-hensel-field}.
  
\begin{ex}
\label{ex:NSR}
Let $L/F$ be an inertial abelian extension with
$[L:F]=n<\ift$, $\Gal(L/F)=G$ and $\exp G=m$.
If $G$ embeds into $\Gamma_F/m\Gamma_F$, then a nicely semiramified
$D\in\calD(F)$ with inertial maximal subfield $L$,
$\ind D=n$ and $\exp D=m$ is constructed as follows.

Let $L_1,\ldots,L_r$ be cyclic subfields of $L/F$ such that
$L\cong L_1\otimes_F\cdots\otimes_F L_r$,
and let $[L_i:F]=n_i$ and $\Gal(L_i/F)=\gen{\sigma_i}$ for $1\leq i\leq r$.
We shall regard $G$ as a subgroup of %
$\Gamma_F/m\Gamma_F$.
Let $\gamma_i\in\Gamma_F$ with %
$\sigma_i=\gamma_i+m\Gamma_F$
for $1\leq i\leq r$.
Since $\ord\sigma_i=n_i$, we have %
$n_i\gamma_i\in m\Gamma_F$.
So we choose $t_i\in F^\times$ with $v(t_i)=\frac{n_i}{m}\gamma_i$
for $1\leq i\leq r$
and set
$$ D:=(L_1/F,\sigma_1,t_1)\otimes_F\cdots\otimes_F(L_r/F,\sigma_r,t_r). $$
Let $c_i\in Z^2(\gen{\sigma_i},L_i^\times)$ be the cyclic cocycle defined by $t_i$
and let $\hat c_i\in Z^2(G,L^\times)$ be the inflation of $c_i$.
Then
$D\cong(L/F,G,c)$ where $c=\hat c_1\cdots \hat c_r$.
\end{ex}
\begin{proof}
Clearly $\deg D=n_1\cdots n_r=n$.
It is shown in \cite[Corollary~2.9]{jacob-wadsworth:constr-noncr-prod} that 
$D$ is a division algebra, 
and that $v$ extends to a valuation $w$ on $D$,
$\bar D=\bar L$, $\Gamma_D$ is generated by $\set{\frac{1}{n_i}v(t_i)}$ and $\Gamma_F$,
and $[\bar L:\bar F]=|\Gamma_D:\Gamma_F|=n=\ind D$.
Obviously, %
$L\cong L_1\otimes_F\cdots\otimes_F L_r$ is an inertial maximal subfield of $D$.
Let $x_i\in(L_i/F,\sigma_i,t_i)$ with $x_i^{n_i}=t_i$.
Then $K:=F(x_1)\otimes_F\cdots\otimes_F F(x_r)$ is a subfield of $D$.
Since $w(x_i)=\frac{1}{n_i}v(t_i)$,
it follows $\Gamma_K=\Gamma_D$,
thus $K$ is a maximal subfield of $D$ and is totally ramified over $F$.
Moreover, the subgroup $\gen{x_i\ug F}$ of $\ug K/\ug F$ maps onto $\Gamma_K/\Gamma_F$ under $\bar w$.
Since $x_i^{n_i}\in\ug F$,
we have $|\gen{x_i\ug F}|\leq n_1\cdots n_r=n$, 
hence the mapping is an isomorphism.
This shows that $K$ is totally ramified of radical type over $F$,
thus $D$ is nicely semiramified.

We have $\exp(L_i/F,\sigma_i,t_i)\mid\deg(L_i/F,\sigma_i,t_i)=n_i$ for $1\leq i\leq r$,
hence $\exp D\mid\lcm(n_1,\ldots,n_r)=\exp G$.
From \cite[Theorem 5.5]{jacob-wadsworth:div-alg-hensel-field}
it is known that $\exp D^h=\exp G$.
Therefore,
$\exp D\mid\exp G=\exp D^h\mid\exp D$,
thus $\exp D=\exp G=m$.
The last statement is immediate from the product theorem \cite[Proposition~14.3]{pierce:ass-alg}.
\end{proof}

\section{Inertially split division algebras}
\label{sec:Inert-split-division}
Throughout this section let $F$ be a field with fixed valuation $v$.
\subsection{Definition of inertially split division algebras}
\label{sec:Defin-inert-split}

In \cite[\S 5]{jacob-wadsworth:div-alg-hensel-field} inertially split division algebras over Henselian fields are defined as follows.

\begin{defn}
  \label{def:inert-split-hensel}
  Let $F$ be Henselian and let $D\in\calD(F)$.
  We say $D$ is \emph{inertially split} if $[D]\in\Br(F_{nr}/F)$
  \footnote{$F_{nr}$ denotes the maximal inertial extension of $F$ in some algebraic closure $F_{alg}$,
    \ie the compositum of all inertial extensions of $F$ in $F_{alg}$.},
  \ie if there is a splitting field of $D$ which is inertial over~$F$.
\end{defn}

\begin{prop}
  \label{prop:charac-inert-split}
  Let $[D]\in\VBr(F)$. The following are equivalent :
  \begin{enumerate}
  \item $D^h$ is inertially split.
  \item $D^h$ contains a maximal subfield which is inertial over $F^h$.
  \item \label{item:properties}
    $\theta_D$ is an isomorphism,
    $Z(\bar D)$ is separable over $\bar F$ 
    and $D$ is defectless.
  \end{enumerate}
\end{prop}
\begin{proof}
  Since $D^h$ is an immediate extension of $D$ by Theorem \ref{thm:div-alg-henselization},
  $(\ref{item:properties})$ holds for $D$ iff it holds for $D^h$.
  Therefore, the Proposition follows from \cite[Lemma 5.1]{jacob-wadsworth:div-alg-hensel-field},
  which shows the equivalence of $(1)$, $(2)$ and $(3)$ for $D^h$.
\end{proof}

\begin{defn}
  \label{def:inert-split}
Let $[D]\in\VBr(F)$.
We say $D$ is \emph{inertially split} if the conditions of Proposition \ref{prop:charac-inert-split} are satisfied.
Define
\[ \SBr(F):=\sett{[D]\in\VBr(F)}{\text{$D$ is inertially split}}\subseteq \VBr(F). \]
\end{defn}

\begin{remark} \mbox{} 
\label{rem:inert-split}

(1) $\IBr(F)\subseteq \SBr(F)\subseteq \VBr(F)$. 

(2) In general $\SBr(F)$ is not a subgroup of $\Br(F)$.
If $F$ is Henselian, then $\SBr(F)=\Br(F_{nr}/F)$, which is a subgroup of $\Br(F)$. 

(3) If $D\in\calD(F)$ is inertially split, then, by Proposition \ref{prop:theta-surj}, 
$Z(\bar D)$ is abelian over $\bar F$.
Moreover if the valuation is discrete, then $Z(\bar D)$ is cyclic over $\bar F$. 

(4) Let $F$ be a field with a discrete valuation $v$ such that the residue field is perfect.
It is known that $\Br(F^h_{nr})$ 
is trivial (cf. \cite[Chapter X, \S 7, Example~b)]{serre:local-fields}).
Therefore $\SBr(F)=\VBr(F)$.
Moreover if $F$ is Henselian, then $\SBr(F)=\Br(F)$.
\end{remark}

The following Theorem is due to Jacob and Wadsworth.

\begin{theorem}
  \label{theorem:inert-split-decomposition}
Let $D\in\calD(F)$ be inertially split.
There exist $I,N\in\calD(F^h)$ with $I$ inertial and $N$ nicely semiramified, 
such that $D^h\sim I\otimes_F N$.
Furthermore, for any such $I$ and $N$~:
\begin{enumerate}
  \item $Z(\bar D)\cong\bar N$.
  \item $\bar D  \cong \bar I_{\bar N}$. 
  \item $\Gamma_D=\Gamma_N$
  \item $\ind D = [\bar N:\bar F]\ind\bar I_{\bar N}$.
  \item $\exp D^h = \lcm(\exp\Gal(\bar N/\bar F),\exp\bar I)$.
\end{enumerate}
\end{theorem}
\begin{proof}
This is \cite[Lemma 5.14]{jacob-wadsworth:div-alg-hensel-field} 
and \cite[Theorem 5.15]{jacob-wadsworth:div-alg-hensel-field}.
The statements (1)--(4) are obtained for $D^h$ first
and then follow for $D$, since $D^h$ is an immediate extension of $D$.
\end{proof}

\begin{remark} \mbox{}
\label{rem:inert-split-index}
(1)
We can not draw back the information on $\exp D^h$ to~$D$,
except that $\exp D^h\mid\exp D.$

(2) The statements $(1)$--$(3)$ in Theorem \ref{theorem:inert-split-decomposition} imply
\begin{equation*}
\ind D = [Z(\bar D):\bar F]\ind\bar D,
\end{equation*}
\ie $\ind D$ can be expressed in terms of $\bar D$ only.
\end{remark}

In general, if $I\in\calD(F)$ is inertial and $N\in\calD(F)$ is nicely semiramified,
we do not have $[I\otimes_F N]\in\VBr(F)$.
The next proposition gives a criterion when $[I\otimes_F N]\in\VBr(F)$.
In this case also $[I\otimes_F N]\in\SBr(F)$.
Moreover, in a special case we get $\exp D=\exp D^h$, so that Theorem \ref{theorem:inert-split-decomposition} can be applied to compute $\exp D$.

\begin{prop}
  \label{prop:N-I-valued}
Suppose $I\in\calD(F)$ is inertial and
$N\in\calD(F)$ is nicely semiramified with inertial maximal subfield $K$.
Let $D\in\calD(F)$ be the underlying division algebra of $N\otimes_F I$.
  \begin{enumerate}
\item If $\ind I_K=\ind \bar I_{\bar K}$,
then $[D]\in\SBr(F)$.
\item If further $\exp N=\exp\Gal(\bar N/\bar F)$
  and $\exp I=\exp\bar I$,
  then 
\[ \exp D=\exp D^h=\lcm(\exp N,\exp I). \]
  \end{enumerate}
\end{prop}
\begin{proof}
Suppose $\ind I_K=\ind\bar I_{\bar K}$.
Since $K$ splits $N$, $\ind D\leq [K:F]\ind I_K$.
$I^h$ is inertial, $N^h$ is nicely semiramified 
and $D^h\sim N^h\otimes_{F^h}I^h$.
Since $[I^h],[N^h]\in\SBr(F^h)$ and $\SBr(F^h)$ is a subgroup of $\Br(F^h)$,
we get $[D^h]\in\SBr(F^h)$.
Theorem \ref{theorem:inert-split-decomposition} yields 
$\ind D^h=[\bar K:\bar F]\ind\bar{I}_{\bar K}$.
Hence
$$ \ind D\leq [K:F]\ind I_K=[\bar K:\bar F]\ind\bar{I}_{\bar K}=\ind D^h\leq\ind D. $$
Therefore, $\ind D=\ind D^h$ and $D^h$ is a division algebra,
\ie $[D]\in\VBr(F)$.
Since $[D^h]\in\SBr(F^h)$, $[D]\in\SBr(F)$.

Now suppose that %
$\exp N=\exp\Gal(\bar N/F)$ and $\exp I=\exp\bar I$.
Theorem~\ref{theorem:inert-split-decomposition} implies
$$\exp D\leq\lcm(\exp N,\exp I)=\lcm(\exp\Gal(\bar N/F),\exp\bar I)=\exp D^h \leq \exp D ,$$
hence $\exp D = \exp D^h %
=\lcm(\exp N,\exp I)$.
\end{proof}

\subsection{Exponents of inertially split division algebras}
\label{sec:Expon-inert-split}

Like it is done in Remark \ref{rem:inert-split-index} for the index $\ind D$,
it is also desirable to compute the exponent $\exp D^h$ directly from data of $\bar D$,
without knowing $\bar I$.
Example~\ref{ex:exp-D-explicit} will show that this is not possible in general,
but we will see in Corollary~\ref{cor:exp} 
that if $Z(\bar D)/\bar F$  is cyclic 
(\eg if the valuation is discrete),
then the formula 
$\exp D^h = \lcm(\exp\Gal(Z(\bar D)/\bar F),\exp A)$
holds for any $A\in\calD(\bar F)$ with $A_{Z(\bar D)}\cong \bar D$.
Hence, in the case that $Z(\bar D)/\bar F$ is cyclic, 
$\exp D^h$ is already determined by $\bar D$.

\begin{lemma}
\label{lemma:exp}
Let $K/k$ be any finite 
field extension of order $n$
and let $A,B\in\calD(k)$ with $A_K\cong B_K$.
Then $\lcm(n,\exp A)=\lcm(n,\exp B)$.
\end{lemma}
\begin{proof}
Since $A_K\cong B_K$, we have $[A\otimes_k B^{op}]\in\Br(K/k)$.
Hence $\exp(A\otimes_k B^{op})\mid[K:k]$.
Since $\Br(k)$ is an abelian group, 
$A = (A\otimes_k B^{op}) \otimes_k B$
implies that
$\exp A \mid \lcm(\exp(A\otimes_k B^{op}),\exp B) \mid \lcm(n,\exp B)$.
Hence $\lcm(n,\exp A)$ divides $\lcm(n, \exp B)$
and the assertion follows by symmetry reasons.
\end{proof}

\begin{cor}
\label{cor:exp}
Let $D\in\calD(F)$ be inertially split.
If $Z(\bar D)/\bar F$ is cyclic,
then $\exp D^h = \lcm([Z(\bar D):\bar F],\exp A)$
for any $A\in\calD(\bar F)$ with $\bar D\cong A_{Z(\bar D)}$.
\end{cor}

If $Z(\bar D)/\bar F$ is a cyclic extension of global fields we are able to calculate $\exp D^h$ from the local data of $\bar D$,
namely the local degrees of $Z(\bar D)/\bar F$ and the local invariants (or local indices) of $\bar D$.

\begin{lemma}
  \label{lem:ind-local-scal-ext}
Let $K/k$ be an extension of local fields, $[K:k]=n$,
and let $A\in\calA(k)$, $A_K=D$.
\begin{enumerate}
\item $\ind D=\frac{\ind A}{(n,\ind A)}$.
\item $n\ind D=\lcm(n,\ind A)$.
\item Let $l$ be the greatest divisor of $n$ prime to $\ind D$.
Then $$\frac{n\ind D}{l} \mid \ind A \mid n\ind D.$$
\item If $n$ is a $p$-power, $p$ prime, and $p\mid\ind D$,
then $\ind A=n\ind D.$
\end{enumerate}
\end{lemma}
\begin{proof}
We have $\inv D=n\inv A$ from (\ref{eq:loc-scal-ext}).
Statement $(1)$ follows, since $\ind A=\exp A$ is the order of $\inv A$ in $\Q/\Z$.
$(2)$ is immediate from $(1)$, since $\lcm(n,\ind A)=\frac{n\ind A}{(n,\ind A)}$.
Let $l$ be the greatest divisor of $n$ prime to $\ind D$.
We clearly have $\ind D\mid\ind A\mid n\ind D$.
If $\ind A=k\ind D$, then $(1)$ shows $k=(n,\ind A)=(n,k\ind D)$.
Thus $k\mid n$ and $(\frac{n}{k},\ind D)=1$.
Therefore $\frac{n}{k}\mid l$,
hence $\frac{n}{l}\ind D\mid k\ind D=\ind A$,
and $(3)$ is proved.
$(4)$ follows from $(3)$, since $l=1$.
\end{proof}

\begin{cor}
\label{cor:ind-local-scal-ext}
Let $K/k$ be an extension of global fields, $[K:k]=n$,
and let $A\in\calA(k)$, $A_K=D$.
For each $w\in\Val(K)$ let $n_w=[K:k]_w$ 
and let $l_w$ be the greatest divisor of $n_w$ prime to $\ind D_w$.
Then
\begin{equation*}
  \lcm\bigl(\frac{n_w\ind D_w}{l_w}\bigr)_{w\in\Val(K)} \mid \ind A \mid
  \lcm\bigl(n_w\ind D_w\bigr)_{w\in\Val(K)} \mid n\ind D. 
\end{equation*}
\end{cor}
\begin{proof}
This follows from Lemma \ref{lem:ind-local-scal-ext} $(3)$ and the formula (\ref{eq:global-index-exp}).
\end{proof}

\begin{prop}
  \label{prop:exp-form-num-field}
Suppose that $\bar F$ is a global field.
Let $D\in\calD(F)$ be inertially split such that $Z(\bar D)/\bar F$ is cyclic
with $[Z(\bar D):\bar F]=n$.
For each $w\in\Val(Z(\bar D))$ let $n_w=[Z(\bar D):\bar F]_w$.
Then
\[
  \exp D^h = \lcm(n,n_w\ind\bar D_w)_{w\in\Val(Z(\bar D))},
\]
\end{prop}
\begin{proof}
Let $K=Z(\bar D)$.
By Theorem \ref{theorem:inert-split-decomposition},
there is an $A\in\calA(\bar F)$ such that
$A_K\cong\bar D$ and $\exp D^h = \lcm(n,\exp A)=\lcm(n,\ind A)$.
Since $n_w\mid n$ for all $w\in\Val(K)$, 
\[
\begin{split}
\exp D^h &= \lcm(n,\lcm(\ind A_v)_{v\in\Val(\bar F)})
= \lcm(n,\lcm(n_w,\ind A_{w|_{\bar F}}))_{w\in\Val(K)} \\
&= \lcm(n,n_w\ind\bar D_w)_{w\in\Val(K)}
\end{split}
\]
by Lemma \ref{lem:ind-local-scal-ext} $(2)$.
\end{proof}

\begin{ex}
  \label{ex:exp-D-explicit}
Let $k=\Q$ and $K=\Q(\sqrt{2},\sqrt{3})$, 
thus $G=\Gal(K/k)\cong\Zn{2}\times\Zn{2}$.
Let $A=(\frac{-1,-1}{\Q})$ be the rational quaternions.
Since $K$ is real, $A^K$ is a division algebra.
For any $2$-cocycle $c\in Z^2(G,\ug K)$ and $B=A\otimes_k (K/\Q,G,c)$
we have $B_K\cong A_K$.
If $(K/\Q,G,c)$ is a division algebra,
then $\exp(K/\Q,G,c)=4$, hence $\lcm(\exp G,\exp B)=\lcm(2,4)=4$.
But $\lcm(\exp G,\exp A)=\lcm(2,2)=2$.

To show that there is a $2$-cocycle $c\in Z^2(G,\ug K)$ such that $(K/\Q,G,c)$ is a division algebra,
we construct a division algebra $D$ with centre $\Q$ that contains a maximal subfield isomorphic to $K$.
Let $L=\Q(\alpha)$ for a root $\alpha$ of $x^4-4x^2+2$.
Then $[L:\Q]=4$ and $\Q(\sqrt{2})\subset L$.
Further, $L/K$ is cyclic, $\Gal(L/K)=\gen{\phi}$ with
$\phi(\alpha)=\alpha^3-3\alpha$, and $\Fix(\phi^2)=\Q(\sqrt{2})$.
Let $v$ be the $3$-adic valuation on $\Q$.
$v$ is inertial in $L/\Q$,
since  $x^4-4x^2+2$ is irreducible modulo $3$,
and $3$ is a prime element with respect to $v$.
Therefore, $D=(L/\Q,\phi,3)$ is a division algebra.
Let $z\in\ug D$ be an element such that
$\iota_z|_L=\phi$ and $z^4=3$.
Then $z^2\in C_D(\Q(\sqrt{2}))$, thus $\Q(\sqrt{2},z^2)$ is a subfield of $D$.
Since $(z^2)^2=3$, $\Q(\sqrt{2},z^2)$ is isomorphic to $K$.
\end{ex}

\subsection{Subfields of inertially split division algebras}
\label{sec:Subf-locally-inert}

We first investigate the residue fields of subfields of inertially split division algebras.

\begin{lemma}
  \label{lem:deg-form}
Let $D\in\calD(F)$ be inertially split.
Any subfield $M$ of $\bar D$ over $\bar F$ is maximal 
if and only if $[M:\bar F]=\ind D$.
\end{lemma}
\begin{proof}
A subfield $M$ of $\bar D$ over $\bar F$ is maximal if and only if $[M:Z(\bar D)]=\ind\bar D$.
Hence by Remark \ref{rem:inert-split-index},
$M$ is maximal if and only if $[M:\bar F]=\ind D$.
\end{proof}

\begin{lemma}
\label{lemma:L-bar-form}
Let $D\in\calD(F)$ be inertially split and let $L$ be any subfield of $D$ with $F\subseteq L$.
  Then
\begin{enumerate}
\item $[Z(\bar D) : \bar L \cap Z(\bar D)] \geq [\Gamma_L:\Gamma_F]$,
\item $[\bar L Z(\bar D) : \bar F] \geq [L:F]$.
\end{enumerate}
\end{lemma}
\begin{proof}
Let $H=\theta(\Gamma_L/\Gamma_F)$.
Since $\theta$ is an isomorphism,
$H$ is a subgroup of $\Gal(Z(\bar D)/\bar F)$ of order $[\Gamma_L:\Gamma_F]$.
Any element from $H$ is of the form $\bar\iota_a|_{Z(\bar D)}$ for some $a\in L$.
Since $L$ is a field, $\iota_a$ is the identity on $L$, 
hence $\bar\iota_a|_{Z(\bar D)}$ is the identity on $\bar L \cap Z(\bar D)$
and $\bar L \cap Z(\bar D) \subseteq \Fix(H)$.
It follows $[Z(\bar D):\bar L \cap Z(\bar D)]\geq[Z(\bar D):\Fix(H)]=[\Gamma_L:\Gamma_F]$
and $(1)$ is shown.

Since $Z(\bar D)$ is Galois over $\bar F$,
$\bar L$ and $Z(\bar D)$ are linearly disjoint over $\bar L \cap Z(\bar D)$,
\ie $[\bar L Z(\bar D) : \bar L\cap Z(\bar D)]=[\bar L:\bar L \cap Z(\bar D)][Z(\bar D):\bar L \cap Z(\bar D)]$.
Hence $[\bar L Z(\bar D) : \bar F] = [\bar L:\bar F][Z(\bar D):\bar L \cap Z(\bar D)]$.
By Remark \ref{rem:defectless-interm-div-ring} $L$ is defectless over $F$,
\ie $[L:F]=[\Gamma_L:\Gamma_F] [\bar L :\bar F]$.
Therefore $[\bar L Z(\bar D) : \bar F] = [L:F]\frac{[Z(\bar D):\bar L \cap Z(\bar D)]}{[\Gamma_L:\Gamma_F]}$.
This formula shows that $(1)$ and $(2)$ are equivalent and we are done
because $(1)$ was already proved.
\end{proof}

\begin{cor}
  \label{cor:L-bar-form}
Let $D\in\calD(F)$ be inertially split and let $L$ be a maximal subfield of $D$.
Then equality holds in Lemma~\ref{lemma:L-bar-form} and $\bar L Z(\bar D)$ is a maximal subfield of $\bar D$.
\end{cor}
\begin{proof}
If $L$ is a maximal subfield of $D$, then $[L:F]=\ind D$.
Hence, by Lemma~\ref{lemma:L-bar-form} $(2)$ and Lemma~\ref{lem:deg-form},
$[L:F] \leq [\bar L Z(\bar D) : \bar F] \leq \ind D = [L:F]$
and equality holds in Lemma~\ref{lemma:L-bar-form}.
Again by Lemma~\ref{lem:deg-form}, $\bar L Z(\bar D)$ is a maximal subfield of $\bar D$.
\end{proof}

Before we can get to the main theorem,
we need 

\begin{prop}
\label{prop:normal-galois}
Let $K/k$ be a finite separable field extension and $D\in\calD(K)$.
If $D$ contains a maximal subfield that is normal over $k$,
then $D$ also contains a maximal subfield that is Galois over $k$.
\end{prop}

For $K=k$ this proposition is \cite[Lemma 3]{saltman:noncr-prod-small-exp}
and the proof given there also handles the case that $K/k$ is finite separable.
Since some lemmas on $p$-algebras are required, 
it is deferred to the appendix (see \S~\ref{sec:p-Algebras}),
as not to interrupt the exposition of this section.

\begin{lemma}
  \label{lem:galois-max-subf-res}
Let $D\in\calD(K)$ be inertially split.
Suppose $F\subseteq K$ is a subfield such that $[K:F]<\ift$ and $Z(\bar D)/\bar F$ is Galois
  \footnote{If $K=F$ then $Z(\bar D)/\bar F$ is Galois by Remark \ref{rem:inert-split} (3).}.
  If $D$ contains a maximal subfield $L$ which is Galois over $F$,
  then $\bar D$ contains a maximal subfield $\wt L$ which is Galois over $\bar F$.
\end{lemma}
\begin{proof}
Let $L$ be a maximal subfield of $D$ which is Galois over $F$.
Then $\bar L$ is normal over $\bar F$,
and $Z(\bar D)$ is normal over $\bar F$ by hypothesis.
Hence $\bar L Z(\bar D)$ is normal over $\bar F$,
and $\bar L Z(\bar D)$ is a maximal subfield of $\bar D$ by Corollary \ref{cor:L-bar-form}.
Since $Z(\bar D)$ is finite separable over $\bar F$, 
there also exists a maximal subfield $\wt L$ of $\bar D$ which is Galois over $\bar F$
by Proposition~\ref{prop:normal-galois}.
\end{proof}

\begin{lemma}
  \label{lem:galois-max-subf-lift}
Let $F$ be a Henselian field and let $K/F$ be a Galois extension, $[K:F]<\ift$.
Suppose $D\in\calD(K)$ is inertially split.
  If $\bar D$ contains a maximal subfield $\wt L$ which is Galois over $\bar F$,
  then $D$ contains a maximal subfield $L$ which is Galois over $F$ and inertial over $K$.
\end{lemma}
\begin{proof}
  Let $\wt L$ be a maximal subfield of $\bar D$ which is Galois over $\bar F$.
  By Theorem~\ref{thm:hensel-ILSP},
  $D$ contains an inertial lift $L$ of $\wt L$ over $F$.
  Since $F$ is Henselian,
  $L/F$ is Galois, hence $LK/F$ is Galois.
  Since $L/F$ is inertial, $LK/K$ is inertial.
  Furthermore,
  $[LK:K]=[L:L\cap K]=[\bar L:\bar{L\cap K}]\geq[\bar L:\bar K]$
(note that $\bar K\subseteq Z(\bar D)\subseteq\wt L=\bar L$).
  But, by Lemma~\ref{lem:deg-form}, $[\bar L:\bar K]=\ind D$,
  hence $LK$ is a maximal subfield of $D$.
\end{proof}

\begin{theorem}
  \label{thm:crossed-product}
Let $D\in\calD(F)$ be inertially split.
Consider the following properties :
  \begin{enumerate}
  \item $D$ contains a maximal subfield which is inertial and Galois over $F$.
  \item $D$ is a crossed product.
  \item $\bar D$ contains a maximal subfield which is Galois over $\bar F$.
  \end{enumerate}
Then $(1)\impl(2)\impl(3)$.
If $F$ is Henselian, then $(1)$--$(3)$ are equivalent,
and moreover, if $\bar D$ contains a maximal subfield that is Galois over $\bar F$ with Galois group $G$,
then $D$ contains a maximal subfield that is Galois over $F$ with the same Galois group $G$.
\end{theorem}
\begin{proof}
The implication (1)$\impl$(2) is trivial,
and (2)$\impl$(3) is Lemma~\ref{lem:galois-max-subf-res} applied with $K=F$.
Now suppose that $F$ is Henselian.
(3)$\impl$(1) is Lem\-ma~\ref{lem:galois-max-subf-lift} applied with $K=F$.
However, for the statement on the Galois groups, 
we have to repeat the arguments.
So let $\wt L$ be a maximal subfield of $\bar D$ Galois over $\bar F$.
By Theorem~\ref{thm:hensel-ILSP}, $D$ contains an inertial lift $L$ of $\wt L$ over~$F$.
Since $F$ is Henselian, 
$L$ is Galois over $F$ with $\Gal(L/F)\cong\Gal(\wt L/\bar F)$.
By Lemma~\ref{lem:deg-form}, $[L:F]=[\bar L:\bar F]=\ind D$,
\ie $L$ is a maximal subfield of~$D$.
This completes the proof of the theorem.
\end{proof}

The following corollary is also known from \cite[Theorem 5.15]{jacob-wadsworth:div-alg-hensel-field}.

\begin{cor}
\label{cor:special-nonr-prod-crit}
Let $D\in\calD(F)$ be inertially split.
If $D$ is a crossed product, then $\bar D$ is a crossed product.
\end{cor}

\begin{cor}
\label{cor:all-crossed-products}  
Let $F$ be a field that is Henselian with respect to a discrete valuation.
If $\bar F$ is a finite, real closed or algebraically closed field,
or a local field of characteristic zero,
then any $D\in\calD(F)$ is a crossed product.
\end{cor}
\begin{proof}
By Remark~\ref{rem:inert-split} (4),
any $D\in\calD(F)$ is inertially split,
hence Theorem~\ref{thm:crossed-product} applies to any $D\in\calD(F)$.
It remains to verify condition $(3)$, %
\ie $\bar D$ contains a maximal subfield Galois over $\bar F$.
This is obvious if $\bar F$ is finite, real closed or algebraically closed.
If $\bar F$ is a local field of characteristic zero, 
we can apply Lemma~\ref{lem:galois-max-subf-lift} again.
$Z(\bar D)/\bar F$ is Galois,
and by Remark~\ref{rem:inert-split}~(4), $\bar D$ is inertially split.
The residue algebra of $\bar D$ is a finite field,
and is therefore Galois over the residue field of $\bar F$.
Hence by Lemma~\ref{lem:galois-max-subf-lift}, 
$\bar D$ contains a maximal subfield Galois over $\bar F$,
and condition (3) of Theorem~\ref{thm:crossed-product} is verified.
\end{proof}

\begin{ex}
  \label{ex:loc-inert-split}
Let $k$ be a finite, real closed or algebraically closed field,
or a local field of characteristic zero.
By iterating the process demonstrated in the proof of Corollary \ref{cor:all-crossed-products}, 
it is shown that for $F=k(\!(t_1,\cdots,t_r)\!)$,
the Laurent series in $r$ indeterminates over $k$,
any $D\in\calD(F)$ is a crossed product.
\end{ex}

\begin{remark}
  \label{rem:noncr-prod-constr}
The implication $(2)\impl(3)$ in Theorem~\ref{thm:crossed-product}
gives a noncrossed product criterion for $D$.
Using this criterion, we can divide the construction of noncrossed products over a valued field $F$ into two steps.
The first step, which takes place on the residue level,
is to find division algebras $\wt D$ over $\bar F$, 
such that $\wt D$ does not contain a maximal subfield which is Galois over $\bar F$.
Note that $\wt D$ can well be a crossed product and $Z(\wt D)/\bar F$ is Galois.
In the second step, which is a ``lifting'',
we have to find inertially split division algebras $D\in\calD(F)$ 
with $\bar D=\wt D$ for the $\wt D$ constructed in the first step.
\end{remark}

\subsection{Existence of inertially split division algebras with given residue algebra}
\label{sec:Exist-inert-split}

This section is motivated by Remark~\ref{rem:noncr-prod-constr}.
It settles the question when the second step of the noncrossed product construction can be performed.

\begin{theorem}
  \label{theorem:lift}
Let $F$ be a valued field
and let $\wt D$ be a finite-dimensional division algebra over $\bar F$
with $[Z(\wt D):\bar F]=n$.
If there exists an inertially split $D\in\calD(F)$ with $\bar D\cong\wt D$,
then the following conditions hold :
\begin{enumerate}
  \item $Z(\wt D)$ is abelian over $\bar F$.
  \item $\wt D\cong A_{Z(\wt D)} $ for some $A\in\calD(\bar F)$.
  \item $\Gal(Z(\wt D)/\bar F)$ embeds into $\Gamma_F/m\Gamma_F$, 
where $m=\exp \Gal(Z(\wt D)/\bar F)$.
\end{enumerate}
The converse is true if $F$ has the inertial lift property.
In this case, $D$ can be found with 
$\ind D=n\ind\wt D$ and $\exp D=\exp D^h=\lcm(m,\exp A)$.
\end{theorem}
\begin{proof}
Suppose $D\in\calD(F)$ is inertially split with $\bar D\cong\wt D$.
Then $(1)$--$(3)$ follow from Remark \ref{rem:inert-split} (3), Theorem~\ref{theorem:inert-split-decomposition} (1) and (2), 
and Remark \ref{rem:NSR} (2).
Furthermore $\ind D=n\ind\wt D$ by Remark~\ref{rem:inert-split-index} (2).

Conversely, suppose $F$ has the inertial lift property and $(1)$--$(3)$ are satisfied.
Let $K$ be an inertial Galois lift of $Z(\wt D)$ over $F$ 
with the property (\ref{eq:ILP}).
Then $G=\Gal(K/F)\cong\Gal(Z(\wt D)/\bar F)$ embeds into $\Gamma_F/m\Gamma_F$.
Let $N\in\calD(F)$ be nicely semiramified with inertial maximal subfield $K$
and $\exp N=\exp G$ as constructed in Example \ref{ex:NSR}.
Let $I\in\calD(F)$ with $[I]\in X$ and $\beta_F([I])=[A]$,
where $X$ is the subgroup of $\Br(F)$ from the inertial lift property.
Then $\bar I\cong A$, $\exp I=\exp A=\exp\bar I$,
and $\ind I_K=\ind\bar I_{\bar K}$ by (\ref{eq:ILP}).
Let $D\in\calD(F)$ be the underlying division algebra of $N\otimes_F I$.
By Proposition \ref{prop:N-I-valued} and Theorem \ref{theorem:inert-split-decomposition}, 
$[D]\in\SBr(F)$,
$\bar D\cong \bar I_{\bar K}\cong A_{Z(\wt D)}\cong\wt D$,
$\ind D=n\ind\wt D$
and $\exp D=\exp D^h=\lcm(m,\exp A)$.
\end{proof}

\begin{remark}
\label{remark:cyclic}
  If $\Gamma_F$ is discrete and $\Gal(Z(\wt D)/\bar F)$ is cyclic, 
  then condition (3) of Theorem \ref{theorem:lift} always holds.
\end{remark}

\chapter{Subfields of Division Algebras over Global Fields}

In Theorem \ref{thm:crossed-product},
the question whether an inertially split division algebra $D\in\calD(F)$ is a crossed product
leads (or is equivalent in the Henselian case) to the question
whether the residue algebra $\bar D$ contains a maximal subfield that is Galois over $\bar F$.
The current chapter attacks this question for the case that $\bar F$ is a global field,
which is our major source of examples.
For convenience of notation we replace the abelian extension $Z(\bar D)/\bar F$ by $K/k$.
The Albert-Hasse-Brauer-Noether Theorem 
reduces the question to the problem
whether the abelian extension $K/k$ 
embeds into a larger Galois extension $L/k$ of given local and global degrees.
We show that in the case that $K/k$ is cyclic,
this is always true if ``enough'' roots of unity are present in $K$,
and there are counterexamples if not.
These counterexamples lead to noncrossed products according to Remark \ref{rem:noncr-prod-constr}.
But for the computation of explicit examples of noncrossed products 
it will be necessary to make $\bar D$ a symbol algebra (see \S~\ref{sec:Autom-symb}),
hence we need many roots of unity in $K$.
Only if $K/k$ is not cyclic we may allow enough roots of unity in our counterexamples
to make $\bar D$ a symbol algebra or even a quaternion algebra.
It is the consideration of non-cyclic $K/k$ that will finally lead to a noncrossed product example
that can be completely and explicitly computed in \S~\ref{sec:An-example-with-8}.

For any field $k$ we denote by $\mu_n(k)$ the group of $n$-th roots of unity contained in $k$,
and we shortly write $\mu_n\subset k$ for the statement $\mu_n(k)=\mu_n(k_{alg})$,
where $k_{alg}$ is an algebraic closure of $k$.

By a symbol algebra we mean a central simple algebra $A\in\calA(k)$ 
over a field $k$ containing a primitive $n$-th root of unity $\zeta$,
such that $A$ is generated by two elements $\alpha,\beta$ with the relations
$\alpha^n=a,\beta^n=b$ and $\beta\alpha=\zeta\alpha\beta$ for some $a,b\in\ug k$.
Then $A$ is written $A=(\frac{a,b}{k,\zeta})$.
If $A$ is a division algebra,
then $A$ is cyclic
since it contains \eg the cyclic maximal subfields
$k(\alpha)$ and $k(\beta)$.

For a prime number $p$, a $p$-algebra 
is understood to be a central simple algebra of $p$-power degree over a field of characteristic $p$.

\section{Embedding of cyclic extensions into Galois extensions}

The goal of this section is to prove

\begin{theorem}
\label{thm:ext-global-gal-ext-general}
Let $K/k$ be a cyclic extension of global fields and $m\in\N$.
In the case $\charak k=0$ let $m_0=m$,
and in the case $\charak k\neq 0$
let $m_0$ be the maximal divisor of $m$
that is not divisible by $\charak k$.
Suppose that $v_1,\ldots,v_r\in\Val(K)$ are pairwise different valuations on $K$.
If $\mu_{m_0}\subset K$
then there exists a field extension $L/K$, $[L:K]=m$,
such that $L/k$ is Galois,
$[L:K]_{v_i}=m$ for each non-archimedian $v_i$
and $[L:K]_{v_i}=\gcd(2,m)$ for each real $v_i$.
Moreover, if $\mu_{m_0}\subset k$ then $L/k$ can be found abelian,
and if $m_0=1$ then $L/k$ can be found cyclic.
\end{theorem}

Applications are the following corollaries.

\begin{cor}
\label{cor:symbol-alg-gal-max-sf}
Let $K/k$ be a cyclic extension of global fields and let $A\in\calA(K)$.
In the case $\charak k=0$ let $m_0=\deg A$, 
and in the case $\charak k\neq 0$ 
let $m_0$ be the maximal divisor of $\deg A$
that is not divisible by $\charak k$.
If $\mu_{m_0}\subset K$, then $A$ contains a strictly maximal subfield $L$ that is Galois over $k$.
Moreover, if $\mu_{m_0}\subset k$, then $L$ can be found abelian over $k$,
and if $m_0=1$, then $L$ can be found cyclic over $k$.
In particular, if $A$ is a symbol algebra then $A$ contains a strictly maximal subfield $L$ that is Galois over $k$,
and if $A$ is a $p$-algebra then $A$ contains a strictly maximal subfield $L$ that is cyclic over $k$.
\end{cor}
\begin{proof}
Let $v_1,\ldots,v_r\in\Val(K)$ be all the valuations for which
$\inv_v A\neq 0$.
If $\mu_{m_0}\subset K$, then by Theorem \ref{thm:ext-global-gal-ext-general},
there is a field $L\supseteq K$ such that $L/k$ is Galois,
$[L:K]=m$, $[L:K]_{v_i}=m$ for each non-archimedian $v_i$
and $[L:K]_{v_i}=\gcd(2,m)$ for each real $v_i$.
Then $\inv_v A^L=0$ for all $v\in\Val(K)$ by (\ref{eq:glob-scal-ext}),
hence $L$ splits $A$ by the Albert-Hasse-Brauer-Noether Theorem.
Since $[L:K]=\deg A$, $L$ is a strictly maximal subfield of $A$.
If $A$ is a symbol algebra, then %
$\mu_m\subset K$, in particular $\mu_{m_0}\subset K$.
If $A$ is a $p$-algebra, then $m$ is a power of $\charak k$, hence $m_0=1$.
\end{proof}

\begin{cor}
\label{cor:symbol-alg-cross-prod}
Let $F$ be a discrete Henselian valued field
such that the residue field $\bar F$ is a global field
and let $D\in\calD(F)$.
In the case $\charak\bar F=0$ let $m_0=\ind\bar D$, 
and in the case $\charak\bar F\neq 0$ 
let $m_0$ be the maximal divisor of $\ind\bar D$
that is not divisible by $\charak\bar F$.
If $\mu_{m_0}\subset Z(\bar D)$ then $D$ is a crossed product.
Moreover, if $\mu_{m_0}\subset\bar F$ then $D$ is an abelian crossed product,
and if $m_0=1$ then $D$ is a cyclic crossed product.
In particular, if $\bar D$ is a symbol algebra then $D$ is a crossed product,
and if $\bar D$ is a $p$-algebra
then $D$ is a cyclic crossed product.  
\end{cor}
\begin{proof}
By Remark \ref{rem:inert-split} (4) and (3),
$D$ is inertially split 
and $Z(\bar D)$ is cyclic over $\bar F$.
Therefore, if $\mu_{m_0}\subseteq Z(\bar D)$,
Corollary \ref{cor:symbol-alg-gal-max-sf}
shows that $\bar D$ contains a maximal subfield $L$ that is Galois over $\bar F$.
Hence $D$ is a crossed product by Theorem \ref{thm:crossed-product}.
Moreover, if $\mu_{m_0}\subset\bar F$ (\resp $m_0=1$)
then $L$ can be found abelian (\resp cyclic) over $\bar F$.
Theorem \ref{thm:crossed-product} %
then shows that $D$ is also abelian (\resp cyclic).
\end{proof}

The idea of the proof of Theorem \ref{thm:ext-global-gal-ext-general} is 
to construct appropriate local extensions of the completions $K_{v_i}$ 
and to combine them to a global extension of $K$ with the approximation theorem.
The proof can be reduced to the case that $m$ is a $p$-power for a prime number $p$.
It is further split into two parts treating the cases 
$\charak k=p$ and $\charak k\neq p$ separately in the following two subsections. 
Both parts will make use of

\begin{lemma}
  \label{lem:application-krasner}
Let $K$ be a Henselian field,
$\alpha$ algebraic and separable over $K$
and $f=\Irr(\alpha,K)$.
If $g\in K[x]$ is a monic polynomial with $\deg g=\deg f$
that approximates $f$ coefficientwise with sufficient precision,
then $g$ is irreducible in $K[x]$ and $K(\alpha)\cong K(\beta)$
for any root $\beta$ of $g$.
\end{lemma}
\begin{proof}
This is a standard application of Krasner's Lemma
(\cf \cite[Lemma 33.8]{reiner:max-orders}).
\end{proof}

\subsection{Extension by a $p$-extension}

We first consider the case $\charak k=p>0$.
By a $p$-extension we mean a field extension of $p$-power degree over a field of characteristic $p$.
It is common to use the notation $\calP(x):=x^p-x$.
Note that we write the valuations $v\in\Val(K)$,
which are all non-archimedian in this case, additively.

\begin{theorem}
\label{thm:cycl-p-ext}
Let $k$ be a field with $\charak k=p>0$,
and let $K/k$ be a cyclic $p$-extension with
$[K:k]=p^n$, $n\geq 1$, and $\Gal(K/k)=\gen{\sigma}$.
Then $K$ contains an element $\beta$ 
such that $\Tr_{K/k}(\beta)=1$,
and for any such $\beta$ there exists an $\alpha\in K$ such that
$\calP(\beta)=\sigma(\alpha)-\alpha$.
If $\alpha\in K$ is such an element
then $x^p-x-\alpha$ is irreducible in $K[x]$,
and if $\gamma$ is a root of $x^p-x-\alpha$ then 
$L=K(\gamma)$ is a cyclic extension of $k$ of degree $p^{n+1}$.
Conversely, any such extension $L$ can be obtained in this way.
\end{theorem}
\begin{proof}
The Theorem is due to Albert \cite[\page 194f]{albert:modern-algebra}.
A proof can also be found in \cite[Theorem 4.2.3, \page 159]{jacobson:fin-dim-div-alg}.
\end{proof}

\begin{lemma}
\label{lem:cycl-by-p-ext}
Let $k$ be a field with $\charak k=p>0$,
and $K/k$ a non-trivial cyclic $p$-extension.
Further, let $k'/k$ be any field extension 
and assume that $k'$ and $K$ are subfields of some common overfield.
Suppose that $K'=Kk'$ is a non-trivial field extension of $k'$,
\ie $K\not\subseteq k'$.
There exists a field extension $L\supseteq K$, $[L:K]=p$,
such that $L/k$ is cyclic.
For any such $L$, $LK'/k'$ is cyclic and $[LK':K']=p$.
\end{lemma}
\begin{proof}
The first statement is Theorem \ref{thm:cycl-p-ext}.
So let now $L\supseteq K$ be any extension such that $[L:K]=p$ and $L/k$ is cyclic.
Let $Z=K\cap k'$.
Since $K\not\subseteq k'$, $K'/k'$ and $K/Z$ are non-trivial cyclic extensions.
Let $\Gal(K'/k')=\gen{\sigma}$.
Then $\Gal(K/Z)\cong\Gal(K'/k')$ and $\Gal(K/Z)=\gen{\sigma|_K}$.
By Theorem~\ref{thm:cycl-p-ext} the cyclic extension $L/Z$ is of the following form.
There are $\alpha,\beta\in K$ with 
$\Tr_{K/Z}(\beta)=1$, $\calP(\beta)=\sigma(\alpha)-\alpha$,
such that $L=K(\gamma)$ where $\gamma$ is a root of $x^p-x-\alpha$.
Then also $\Tr_{K'/k'}(\beta)=1$,
hence by Theorem \ref{thm:cycl-p-ext},
$LK'=K'(\gamma)$ is cyclic over $k'$ and $[LK':K']=p$.
\end{proof}

\begin{theorem}
\label{thm:cycl-by-p-ext-global}
Let $K/k$ be a cyclic $p$-extension of global fields
with $\charak k=p>0$.
Let $v_1,\ldots,v_r\in\Val(K)$ be pairwise different valuations.
For any $m\in\N$ there exists a field extension $L\supseteq K$, $[L:K]=p^m$,
such that $L/k$ is cyclic and
$[L:K]_{v_i}=p^m$ for $i=1,\ldots,r$.
\end{theorem}
\begin{proof}
We prove the theorem for $m=1$, 
it then follows for any $m\in\N$ by induction.
We may further assume that $[K:k]=p^n, n\in\N_0$.
For, let $[K:k]=p^n n'$, $p\ndiv n'$,
and let $K_0\subseteq K$ be the unique subfield with $[K_0:k]=p^n$.
If there exists $L_0\supseteq K_0$ with $L_0/k$ cyclic
and $[{L_0}:{K_0}]_{v_i}=[L_0:K_0]=p^m$ for $i=1,\ldots,r$,
then $L=KL_0$ has the desired properties.

So let $[K:k]=p^n, n\in\N_0$, and $\Gal(K/k)=\gen{\sigma}$.
If $n=0$ let $\alpha=0$,
otherwise let $\alpha\in K$ be as in Theorem \ref{thm:cycl-p-ext},
\ie $\sigma(\alpha)-\alpha=\calP(\beta)$ for some $\beta\in K$ with
$\Tr_{K/k}(\beta)=1$.
Let \wolg $1\leq s\leq r$ such that $K_{v_i}=k_{v_i}$ for all $i=1,\ldots,s$,
and $K_{v_i}\neq k_{v_i}$ for all $i=s+1,\ldots,r$.
For $i=1,\ldots,s$
let $L_i$ be the inertial extension of $K_{v_i}$ of degree $p$
(unique up to isomorphism)
and let $L_i=K_{v_i}(\gamma_i)$ for a root $\gamma_i$ of $x^p-x-\alpha_i$,
$\alpha_i\in K_{v_i}$.
Note that the elements $\alpha_i-\alpha$ lie in $k_{v_i}$ for $i=1,\ldots,s$,
since $K_{v_i}=k_{v_i}$.
By the approximation theorem we then find a $c\in k$ such that
\[ v_i(c-(\alpha_i-\alpha)) > N \qt{for all $i=1,\ldots,s$} \]
for any $N\in\N$.
If we set $\alpha':=\alpha+c\in K$ this means
that $\alpha'$ approximates the $\alpha_i$ with respect to $v_i$ 
with arbitrary precision.
Let $L=K(\gamma)$ for a root $\gamma$ of $x^p-x-\alpha'$.
Since $c\in k$ we have
$\sigma(\alpha')-\alpha'=\sigma(\alpha)+c-\alpha'=\sigma(\alpha)-\alpha=\calP(\beta)$.
Therefore, by Theorem \ref{thm:cycl-p-ext},
$L/k$ is cyclic and $[L:K]=p$.
By Lemma \ref{lem:application-krasner},
$L_{v_i}\cong L_i$ for $i=1,\ldots,s$,
thus $[L:K]_{v_i}=p$  for $i=1,\ldots,s$.
And by Lemma \ref{lem:cycl-by-p-ext},
$[L:K]_{v_i}=p$ for $i=s+1,\ldots,r$
since $K_{v_i}\neq k_{v_i}$.
This proves the theorem.
\end{proof}

\subsection{The general case}

The following up to Theorem \ref{thm:ex-global-gal-ext-char-neq-p} is concerned with the case $\charak k\neq p$.
Then Theorem \ref{thm:ext-global-gal-ext-general} is proved by combining
Theorem~\ref{thm:cycl-by-p-ext-global} and Theorem~\ref{thm:ex-global-gal-ext-char-neq-p}.
In this section, by a discrete valuation $v$ on a field $k$
we mean a discrete non-archimedian valuation,
and we shall write it additively.
The unit group is denoted by $U_k=\sett{x\in\ug k}{v(x)=0}$.

\begin{lemma}
\label{lem:equal-radical-ext}
Let $k$ be a field, $p$ a prime number with $\charak k\neq p$,
and $a,b\in k$.
Suppose $\mu_p\subset k$ and $x^p-a,x^p-b$ are irreducible in $k[x]$.
If $k(\alpha)=k(\beta)$
for some roots $\alpha,\beta$ of $x^p-a,x^p-b$ respectively,
then there exists $c\in k$ such that
$a=c^pb^i$ for some $i\in\N$, $0<i<p$.
\end{lemma}
\begin{proof}
This lemma is a corollary of Kummer theory.
We give a proof here for completeness.
Suppose $k(\alpha)=k(\beta)$.
Since $\mu_p\subset k$, $k(\alpha)/k$ is cyclic of degree $p$
and a generating automorphism is given by $\sigma(\alpha)=\zeta\alpha$
for some fixed primitive $p$-th root of unity $\zeta$.
Then $\sigma(\beta)=\zeta^j\beta$ for some $j\in\N$, $0<j<p$.
Let $i\in\N$, $0<i<p$, with $ij\equiv 1\bmod p$.
Then $\sigma(\beta^i)=\zeta\beta^i$,
hence $\sigma(\frac{\alpha}{\beta^i})=\frac{\alpha}{\beta^i}$.
This shows $c:=\frac{\alpha}{\beta^i}\in k$,
and we have $a=c^pb^i$.
\end{proof}

\begin{lemma}
\label{lem:p-ndiv-v-irred}
Let $k$ be a field with normalized discrete valuation $v$, 
$p$ a prime number 
and $a\in k$.
If $p\ndiv v(a)$
then the polynomial $x^{p^m}-a$ is irreducible in $k[x]$
for any $m\in\N$,
and for any root $\alpha$ of this polynomial,
$k(\alpha)$ is totally ramified over $k$.
\end{lemma}
\begin{proof}
Let $\alpha$ be a root of $x^{p^m}-a$ and 
$\pi$ a prime element of $k$ with respect to $v$, \ie $v(\pi)=1$.
Let $w$ be a valuation on $k(\alpha)$ that extends $v$.
Since $p\ndiv v(a)$ there are $j,k\in\Z$ with $jv(a)+kp^m=1$.
If we set $\alpha':=\alpha^j\pi^k$
then $$w({\alpha'}^{p^m})=w(\alpha^{jp^m})+w(\pi^{kp^m})=jv(a)+kp^m=1.$$
This means that $k(\alpha)$ contains an element $\alpha'$ with $w(\alpha')=\frac{1}{p^m}$.
Therefore $$p^m\leq |w(\ug{k(\alpha)}):v(\ug k)|\leq [k(\alpha):k]\leq p^m,$$
hence equality holds here,
$k(\alpha)$ is totally ramified over $k$ 
and $x^{p^m}-a$ is irreducible in $k[x]$.
\end{proof}

\begin{lemma}
\label{lem:unit-prime-element}
Let $k$ be a field with normalized discrete valuation $v$, 
$p$ a prime number with $\charak k\neq p$,
and $a,b\in k$.
Suppose $\mu_p\subset k$.
If $p\ndiv v(a)$ and $p\mid v(b)$ then
$k(\alpha)\neq k(\beta)$ and $k(\alpha)\cap k(\beta)=k$ for any roots
$\alpha,\beta$ of $x^p-a,x^p-b$ respectively.
\end{lemma}
\begin{proof}
Let $\alpha$ and $\beta$ be roots of $x^p-a$ and $x^p-b$ respectively.
The polynomial $x^p-a$ is irreducible in $k[x]$ by Lemma \ref{lem:p-ndiv-v-irred},
hence $[k(\alpha):k]=p$.
Suppose $k(\alpha)=k(\beta)$.
Then $x^p-b$ must also be irreducible in $k[x]$
and by Lemma \ref{lem:equal-radical-ext}
there exists $c\in k$ such that
$a=c^pb^i$ for some $i\in\N$, $0<i<p$.
If $p\mid v(b)$
then $p$ also divides
$v(a)=pv(c)+iv(b)$, a contradiction.
Therefore $k(\alpha)\neq k(\beta)$.
Since $[k(\alpha):k]=p$ is a prime number
this implies $k(\alpha)\cap k(\beta)=k$.
\end{proof}

\begin{lemma}
\label{lem:mu-p-r-in-F}
Let $k$ be a field with normalized discrete valuation $v$,
$p$ a prime number with $\charak k\neq p$,
$a\in k$ with $p\ndiv v(a)$,
and $r,m\in\N$.
If $\alpha$ is a root of the polynomial $x^{p^m}-a$
then $\mu_{p^r}(k(\alpha))=\mu_{p^r}(k)$.
\end{lemma}
\begin{proof}
We first prove the lemma for the case $m=1$ by induction on $r$,
and then in general by induction on $m$.
The statement is trivial for $r=0$.
Suppose $\mu_{p^r}\subset k$ and $\mu_{p^{r+1}}\subset k(\alpha)$.
Let $\zeta_{p^{r+1}}\in k(\alpha)$ be a primitive $p^{r+1}$-th root of unity.
Then $\zeta_{p^{r+1}}$ is the root of a polynomial
$x^p-\zeta_{p^r}$ for some primitive $p^r$-th root of unity $\zeta_{p^r}\in k$.
Since $\zeta_{p^r}$ is a unit with respect to the valuation $v$ of $k$,
Lemma \ref{lem:unit-prime-element} implies $k(\zeta^{p^{r+1}})\cap k(\alpha)=k$,
\ie $\zeta_{p^{r+1}}\in k$.
This proves the case $m=1$.

Now let $m>1$ and let $a'=\alpha^p$ and $k'=k(a')\subseteq k(\alpha)$.
The element $a'$ is a root of the polynomial $x^{p^{m-1}}-a$,
thus the induction hypothesis yields $\mu_{p^r}(k')=\mu_{p^r}(k)$.
Moreover by Lemma \ref{lem:p-ndiv-v-irred}, $k'$ is totally ramified over $k$.
Let $w$ be an extension of $v$ to $k'$.
Then the normalized valuation $w'$ on $k'$ that is equivalent to $w$
is given by $w'(x):=p^{m-1}w(x)$ for all $x\in k'$.
Since ${a'}^{p^{m-1}}=a$ this implies
$v(a)=p^{m-1}w(a')=w'(a')$, thus $p\ndiv w'(a')$.
Obviously $\alpha$ is a root of $x^p-a'$,
hence the case $m=1$ yields $\mu_{p^r}(k(\alpha))=\mu_{p^r}(k')$.
This completes the proof.
\end{proof}

The following Theorem is a corollary of \cite[Theorem 3.7]{albu-nicolae:kneser}.

\begin{theorem}
\label{thm:linearly-ordered-subfields}
Let $k$ be a field and $p$ a prime number with $\charak k\neq p$.
If $\alpha$ is the root of some polynomial $x^{p^m}-a\in k[x]$, $m\in\N$,
then the subfields of $k(\alpha)$ over $k$ are linearly ordered if and only if
$\mu_p(k(\alpha))=\mu_p(k)$ and $\mu_4(k(\alpha))=\mu_4(k)$ if $p=2$.
\end{theorem}

\begin{cor}
\label{cor:radical-ext-subf-lin-ord}
Let $k$ be a field with normalized discrete valuation $v$,
$p$ a prime number with $\charak k\neq p$,
$a\in k$ and $m\in\N$.
If $p\ndiv v(a)$ then
the polynomial $x^{p^m}-a$ is irreducible in $k[x]$
and if $\alpha$ is a root of this polynomial
then the subfields of $k(\alpha)$ over $k$ are linearly ordered.
\end{cor}
\begin{proof}
Lemma \ref{lem:p-ndiv-v-irred} states that 
the polynomial $x^{p^m}-a$ is irreducible in $k[x]$.
It is also separable since $\charak k\neq p$.
The corollary then follows from Lemma \ref{lem:mu-p-r-in-F}
and Theorem \ref{thm:linearly-ordered-subfields}.
Note that Lemma \ref{lem:mu-p-r-in-F} is not needed here in full generality.
In fact, $\mu_p\subset K\impl\mu_p\subset k$ is trivial from a consideration of degrees. 
We only need Lemma \ref{lem:mu-p-r-in-F} in the case $p=2$ to show
$\mu_4\subset K\impl\mu_4\subset k$.
\end{proof}

\begin{prop}
\label{prop:ex-local-radical-ext}
Let $k$ be a discretely valued field and $p$ a prime number
such that $\charak k\neq p$ and $U_k\not\subseteq k^p$.
For any finite cyclic field extension $K/k$ and any $m\in\N$
there exists an $a\in k$ such that 
the polynomial $x^{p^m}-a$ is irreducible in $k[x]$,
and if $\alpha$ is a root of this polynomial then
$K$ and $k(\alpha)$ are linearly disjoint over $k$.
\end{prop}
\begin{proof}
We will choose $a\in k$ an appropriate element such that $p\ndiv v(a)$.
Then the polynomial $x^{p^m}-a$ is irreducible in $k[x]$ by Lemma \ref{lem:p-ndiv-v-irred},
\ie $[k(\alpha):k]=p^m$ for any root $\alpha$,
and the subfields of $k(\alpha)$ over $k$ are linearly ordered
by Corollary \ref{cor:radical-ext-subf-lin-ord}.
To show that $K$ and $k(\alpha)$ are linearly disjoint over $k$
we prove $K\cap k(\alpha)=k$.
This is sufficient since $K/k$ is Galois.
$K/k$ is cyclic, thus $K=K_0K_1$ with $[K_0:k]=p^n, p\ndiv[K_1:k]$.
It suffices to show $K_0\cap k(\alpha)=k$,
so we assume that $[K:k]=p^n$ and $n\geq 1$.
Since $K/k$ is cyclic of $p$-power degree,
the subfields of $K$ over $k$ are also linearly ordered.
Therefore both $K$ and $k(\alpha)$ contain a unique subfield of degree $p$ over $k$.
Hence we assume $n=m=1$ and show $K\neq k(\alpha)$.
\\
Case 1 : $K/k$ is not a radical extension.
Choose $a$ any prime element of $k$,
then $p\ndiv v(a)$ and trivially $k(\alpha)\neq K$ for any root $\alpha$ of $x^p-a$.
\\
Case 2 : $K=k(\beta)$ for some $\beta\in K$, $\beta^p=b\in k$ and $p\mid v(b)$.
Choose $a$ any prime element of $k$.
Then $p\ndiv v(a)$, 
hence $k(\alpha)\neq k(\beta)$ for any root $\alpha$ of $x^p-a$ 
by Lemma \ref{lem:unit-prime-element}.
\\
Case 3 : $K=k(\beta)$ for some $\beta\in K$, $\beta^p=b\in k$ and $p\ndiv v(b)$.
Choose $a=ub$ for some $u\in U_k$ with $u\not\in k^p$.
Assume $k(\alpha)=k(\beta)$.
Since $v(a)=v(b)$ and $p\ndiv v(b)$
the polynomials $x^p-a$ and $x^p-b$ are irreducible by Lemma \ref{lem:p-ndiv-v-irred}.
Lemma \ref{lem:equal-radical-ext} then shows that
there exists $c\in k$ such that $a=c^pb^i$, $0<i<p$.
Thus $c^p=ub^{1-i}$, hence $p$ divides $v(ub^{1-i})=(1-i)v(b)$.
Since $p\ndiv v(b)$, it follows $i=1$, 
thus $c^p=u$, a contradiction.
\end{proof}

\begin{remark}
\label{rem:U-F-not-in-F-p}
If $k$ is a local field then the hypothesis $U_k\not\subseteq k^p$ is 
satisfied for any $p\neq\charak k$.
\end{remark}
\begin{proof}
This is true because a local field $k$ with $\charak k\neq p$
does not contain all $p^n$-th roots of unity
for all $n\in\N$.
This can be seen as follows.
If $\charak\bar k\neq p$ then $\mu_{p^n}\not\subset\bar k$ for almost all $n\in\N$, hence $\mu_{p^n}\not\subset k$.
If $\charak\bar k=p$ we must have $\charak k=0$,
\ie $k$ is a finite extension of $\Q_p$.
But $[\Q_p(\mu_{p^n}):\Q_p]=\varphi(p^n)=(p-1)p^{n-1}$
(cf. \cite[II, Satz 7.13]{neukirch:alg-zt}),
thus $\mu_{p^n}\not\subset k$ for almost all $n\in\N$.
\end{proof}

\begin{theorem}
\label{thm:ex-global-gal-ext-char-neq-p}
Let $K/k$ be a cyclic extension of global fields,
$p$ be a prime number with $\charak k\neq p$ and $m\in\N$.
Let $v_1,\ldots,v_r\in\Val(K)$ be pairwise different valuations on $K$.
If $\mu_{p^m}\subset K$
there exists a field extension $L/K$ with $[L:K]=p^m$
such that $L/k$ is Galois,
$[L:K]_{v_i}=p^m$ for each non-archimedian $v_i$
and $[L:K]_{v_i}=\gcd(2,p)$ for each real $v_i$.
Moreover, if $\mu_{p^m}\subset k$ then $L/k$ can be found abelian.
\end{theorem}
\begin{proof}
We can assume \wolg that none of the $v_i$ is complex,
and in the case $p\neq 2$ that all the $v_i$ are non-archimedian.
For each non-archimedian $v_i$,
by Proposition \ref{prop:ex-local-radical-ext} and Remark \ref{rem:U-F-not-in-F-p}, there is $a_i\in k_{v_i}$ such that
$x^{p^m}-a_i$ is irreducible in $k_{v_i}[x]$
and $k_{v_i}(\alpha_i)$ and $K_{v_i}$ are linearly disjoint over $k_{v_i}$
for any root $\alpha_i$ of $x^{p^m}-a_i$.
Hence the polynomial $x^{p^m}-a_i$ is irreducible in $K_{v_i}[x]$.
If any of the $v_i$ is real
then we are in the case $\charak k=0, p=2$ and $m=1$,
since $\mu_{2^m}\subset K$.
Therefore, we get the same statement also for each real $v_i$ by simply choosing $a_i=-1$.

Let $a\in k$ be an element
that approximates the $a_i$ with respect to $v_i|_k$ 
with arbitrary precision for all $i=1,\ldots,r$
(by the approximation theorem).
Lemma \ref{lem:application-krasner} then shows
that the polynomial $x^{p^m}-a$ is irreducible in $K_{v_i}[x]$ and that 
$K_{v_i}(\alpha)\cong K_{v_i}(\alpha_i)$
for any root $\alpha$ of $x^{p^m}-a$.
Therefore 
\[
  [K(\alpha):K]_{v_i}=[K_{v_i}(\alpha):K_{v_i}]=p^m \qt{for $i=1,\ldots,r$.}
\]
This implies $[K(\alpha):K]=p^m$.
Since $\mu_{p^m}\subset K$ and $\charak k\neq p$,
$K(\alpha)$ is the splitting field of the separable polynomial 
$x^{p^m}-a\in k[x]$ over $K$.
By hypothesis $K/k$ is Galois,
and $x^{p^m}-a$ is a polynomial over $k$,
hence $L=K(\alpha)$ is Galois over $k$.

Moreover, if $\mu_{p^m}\subset k$ then $k(\alpha)$ is cylic over $k$.
Therefore, since $K$ and $k(\alpha)$ are linearly disjoint over $k$,
the Galois group of $L=K\cdot k(\alpha)$ over $k$
is the direct product of two cyclic groups.
\end{proof}

Combining this result with Theorem \ref{thm:cycl-by-p-ext-global} we get the

\begin{proof}[Proof of Theorem \ref{thm:ext-global-gal-ext-general}]
If $\charak k=0$ let $L_0=K$.
If $\charak k=p\neq 0$, $m=p^em_0$, $p\ndiv m_0$,
then by Theorem \ref{thm:cycl-by-p-ext-global},
there is a field $L_0\supseteq K$
such that $L_0/k$ is cyclic and $[L_0:K]=[L_0:K]_{v_i}=p^e$
for all $i=1,\ldots,r$.
Let $m_0=p_1^{e_1}\cdots p_s^{e_s}$ be the prime factorization of $m_0$.
By Theorem \ref{thm:ex-global-gal-ext-char-neq-p} there exist fields
$L_1,\ldots,L_s\supseteq K$ such that for each $1\leq j\leq s$ :
$L_j/k$ is Galois, $[L_j:K]=p_j^{e_j}$,
$[L_j:K]_{v_i}=p_j^{e_j}$ for each non-archimedian $v_i$
and $[L_j:K]_{v_i}=\gcd(2,p_j)$ for each real $v_i$.
Then $L:=L_0 L_1\cdots L_s$ is Galois over $k$, $[L:K]=m$,
$[L:K]_{v_i}=m$ for each non-archimedian $v_i$
and $[L:K]_{v_i}=\gcd(2,m_0)$ for each real $v_i$.
If one of the $v_i$ is real we are in the case $\charak k=0$,
\ie $m_0=m$.
Hence $[L:K]_{v_i}=\gcd(2,m)$ for each real $v_i$.

Moreover, if $\mu_{m_0}\subset k$ then the $L_i$, $1\leq i\leq s$, are abelian over $k$,
hence $L$ is abelian over $k$.
If $m_0=1$ then $L=L_0$ is cyclic over $k$.
\end{proof}

\section{Non-embeddable cyclic extensions}

\label{sec:Non-embedd-cycl}

We will prove in this section the existence of cyclic extensions of number fields 
that do not embed into Galois extensions of a given degree
with certain local degrees.
We know from Theorem \ref{thm:ext-global-gal-ext-general}
that such a result must require the absence of certain roots of unity.
The proofs given here are essentially the ones from \cite{brussel:noncr-prod}
and the results that we get in the case that the degrees are prime powers
are the same as in \cite{brussel:noncr-prod}.
The main theorem is

\begin{theorem}
  \label{thm:brussel-field-existence-global}
  For $k$ a number field and $p$ a prime number 
  let $r,s\in\N_0$ be maximal such that $\mu_{p^r}\subset k^\times$
  and $\mu_{p^s}\subset k(\mu_{p^{r+1}})^\times$.
For any $n_0,n',m,a\in\N$ such that $n_0=r$ or $n_0\geq s$,
$p\ndiv n', p^{r+1}\mid m$ and $a\mid m$,
there exist a cyclic extension $K/k$ such that $[K:k]=na$ for $n=p^{n_0}n'$,
and valuations $v_0,\ldots,v_3\in\V0(k)$ such that $[K:k]_{v_i}=n$ for $0\leq i\leq 3$
with the following property~
\footnote{Although the valuation $v_3$ does not appear anymore in the rest of this theorem and its proof, $v_3$ is needed in the applications.}.
\marginpar{It is in fact not required that $w_1$ extends uniquely to $L$ because this condition can be removed from Lemmas \ref{lem:brussel-abelian} and \ref{lem:brussel-no-galois-ext} below.}
If $L/K$ is a field extension with $[L:K]=m$
and there are $w_0,w_1,w_2\in\V0(K)$ with $w_i\mid v_i$ for $0\leq i\leq 2$, 
that extend uniquely to valuations on $L$,
then $L$ is not Galois over $k$.
\end{theorem}

\begin{remark}
\label{rem:r-s}
If $p\neq 2$ or $p=2$ and $r\geq 2$ then we always have $s=r+1$.
Thus the condition on $n_0$ reduces to $n_0\geq r$.
Only in the case $p=2$ and $r=1$ it is possible that $s>r+1$.
An example is the field $k=\Q(\sqrt{2})$,
since $k(i)$ contains $\frac{1}{\sqrt{2}}(1+i)$,
which is a primitive $8$-th root of unity.
\end{remark}

Before we can prove Theorem \ref{thm:brussel-field-existence-global},
we first recall some facts about global fields (\cf \cite[Kapitel~II]{neukirch:alg-zt})
and then prove a few lemmas.
Let $k$ be a global field and let $v\in\V0(k)$.
The residue field $\bar k_v$ is finite, and if $|\bar k_v|=q$ 
then $\bar k_v^\times$ consists of the $(q-1)$-th roots of unity,
hence 
\begin{fact}
\label{eq:mu-e-in-bar-k}
$\bar k_v$ contains a primitive $n$-th root of unity if and only if $q\equiv 1\bmod n$.
\end{fact}
Let $K/k$ be a finite Galois extension, $[K:k]=n$, $\Gal(K/k)=G$.
Then the residue degrees $[\bar K_w:\bar k_v]$ 
(\resp the ramification indices $|w(K^\times):v(k^\times)|$)
are equal for all $w\in\V0(K)$ with $w\mid v$.
Therefore we simply denote them by $f_v(K/k)$ (\resp $e_v(K/k)$).
We say $v$ is \emph{tamely ramified} in $K/k$
if $\charak\bar k_v\ndiv e_v(K/k)$.
The decomposition group of $w\in\V0(K)$ in $K/k$ is the subgroup
\[ G_w(K/k)=\{\sigma\in G\;|\,w\circ\sigma=\sigma\} \]
of $G$,
and the decomposition field of $w$ in $K/k$ is the fixed field
\[ Z:=Z_w(K/k)=\Fix(G_w(K/k)). \]

Now suppose that $v$ extends uniquely to a valuation $w$ on $K$.
Then $K_w/k_v$ is Galois with $\Gal(K_w/k_v)\cong G$.
We say $v$ is \emph{inertial} in $K/k$
if $f_v(K/k)=[K:k]$,
and $v$ is \emph{totally ramified} in $K/k$
if $e_v(K/k)=[K:k]$.
If $v$ is tamely and totally ramified in $K/k$,
then $K_w=k_v(\sqrt[n]{\xi})$ for some $\xi\in k_v$
(cf. \cite[Kap. II, Satz~7.7]{neukirch:alg-zt}).
It follows that $G$ is cyclic 
and that $k_v$ and $\bar k_v$ contain a primitive $n$-th root of unity,
hence $q\equiv 1\bmod n$.
The inertia group of $v$ in $K/k$ is the subgroup
$$ I:=I_v(K/k) = \sett{\sigma\in G}{\sigma(x)\equiv x\bmod\PP_w \text{ for all $x\in\OO_w$}} $$
of $G$ and the inertia field of $v$ in $K/k$ is the fixed field
$$ T:=T_v(K/k)=\Fix(I_v(K/k)). $$
Then $I$ is a normal subgroup of $G$, and
$G/I\cong\Gal(\bar K_w/\bar k_v)$
(cf. \cite[Satz~9.9]{neukirch:alg-zt}).
Since $\bar K_w$ is a finite field,
$G/I$ is cyclic,
thus $T/k$ is a cyclic extension.
$T/k$ is the maximal unramified subextension of $K/k$,
and $K/T$ is totally ramified with respect to $v$
(cf. \cite[Satz~9.11]{neukirch:alg-zt}).
Hence if $v$ is tamely ramified in $K/k$, then $I$ and $G/I$ are both cyclic.

The following proposition is known from \cite[Theorem 10]{albert:p-adic-fields-and-rational-division-algebras}.

\begin{prop}
  \label{prop:albert-local-fields}
  Let $K/k$ be a finite Galois extension of local fields 
  that is tamely ramified with ramification index $e:=e(K/k)$
  (\ie $e$ is prime to $\charak \bar k$).
  Then $K/k$ is abelian if and only if $\mu_e\subset k^\times$ 
  (if and only if $|\bar k|\equiv 1\bmod e$).
\end{prop}

Recall the following fact.

\begin{remark}
  \label{rem:mu-e-classification}
  Let $k$ be a global field with non-archimedian valuation $v\in\V0(k)$, 
  and let $e\in\N$ be prime to $\charak\bar k_{v}$.
  Then $v$ splits completely in $k(\mu_e)/k$
  if and only if $|\bar k_{v}|\equiv 1\bmod e$.
\end{remark}
\begin{proof}
Obviously $v$ splits completely in $k(\mu_e)/k$ if and only if $\mu_e\subset k_v$.
By Hensel's Lemma, $\mu_e\subset k_v$ if and only if $\mu_e\subset \bar k_v$,
and the remark follows from (\ref{eq:mu-e-in-bar-k}).
\end{proof}

\begin{lemma}
  \label{lem:brussel-abelian}
Let  $k$ be a global field, $p$ a prime number with $\charak k\neq p$,
and let $n_0,m_0\geq 1$.
Suppose $v_1,v_2\in\V0(k)$ are non-archimedian valuations on $k$,
$\charak\bar k_{v_i}\neq p$ for $i=1,2$,
such that
  \begin{enumerate}
    \item $|\bar k_{v_1}|\equiv 1 \bmod p^{n_0}$, 
    \item $|\bar k_{v_1}|\not\equiv 1 \bmod p^{n_0+1}$,
    \item $|\bar k_{v_2}|\not\equiv 1 \bmod p^{m_0}$.
  \end{enumerate}
Suppose further that $K/k$ is a cyclic extension, $[K:k]=p^{n_0}$, such that
$v_1,v_2$ extend uniquely to valuations on $K$ and
  \begin{enumerate}
    \setcounter{enumi}{3}
  \item $v_1$ is totally ramified in $K/k$,
  \item $v_2$ is inertial in $K/k$.
  \end{enumerate}
If
\marginpar{The unique extension of $w_1$ to $L$ is not necessary:
Since $K_{v_1}/k_{v_1}$ is totally ramified, using $(2)$, $K/k$ does not embed in a cyclic extension of degree $p^{n_0+1}$ (globally).
Hence, $T_2=K$, $e_2=p^{m_0}$.
$L/k$ is abelian and $v_2$ extends uniquely to $L$, 
this contradicts $(3)$.
}
If $L/K$ is a field extension, $[L:K]=p^{m_0}$,
such that $v_1,v_2$ extend uniquely to valuations on $L$,
then $L/k$ is not abelian.
\end{lemma}
\begin{proof}
Suppose $L/k$ is an abelian extension, $K\subseteq L$,
with $[L:k]=[L:k]_{v_i}=p^{n_0+m_0}$ for $i=1,2$.
We shorty write $e_i$ for $e_{v_i}(L/k)$,
$f_i$ for $f_{v_i}(L/k)$,
and $T_i$ for $T_{v_i}(L/k)$.
Since $\charak\bar k_{v_i}\neq p$, the extensions $L_{v_i}/k_{v_i}$
are tamely ramified, $i=1,2$.
By Proposition \ref{prop:albert-local-fields},
$|\bar k_{v_1}|\equiv 1 \bmod e_1$.
Therefore, $(2)$ implies $p^{n_0+1}\ndiv e_1$.
But $(4)$ shows $p^{n_0}\mid  e_1$, 
hence $e_1=p^{n_0}$ and $f_1=p^{m_0}$.

Because of $(4)$ and $(5)$ we have $T_1\cap K=k$ and $K\subseteq T_2$.
Since $T_2/k$ is cyclic this implies $T_1\cap T_2 =k$.
Therefore $v_2$ is totally ramified in $T_1/k$, $[T_1:k]=f_1=p^{m_0}$,
hence by Proposition \ref{prop:albert-local-fields},
$|\bar k_{v_2}|\equiv 1 \bmod p^{m_0}$.
This contradicts $(3)$, so $L$ can not exist.
We have not used $(1)$ in the proof,
but $(1)$ is implicit in $(4)$ by Proposition \ref{prop:albert-local-fields}.
\end{proof}

\begin{lemma}
  \label{lem:brussel-no-galois-ext}
Let  $k$ be a global field and let $p$ be a prime number with $\charak k\neq p$.
Let $n=p^{n_0}n',m=p^{m_0}m'$ with $n_0,m_0\geq 1$, $p\ndiv n',m'$,
$\charak k\ndiv m'$ if $\charak k\neq 0$,
and let $a\geq 1$.
Suppose $v_0,v_1,v_2\in\V0(k)$ are non-archimedian valuations on $k$
with $\charak\bar k_{v_0}\ndiv m$, $\charak\bar k_{v_1}\neq p$ and $\charak\bar k_{v_2}\neq p$ 
such that
  \begin{enumerate}
    \item $|\bar k_{v_0}|\equiv 1 \bmod m$,
    \item $|\bar k_{v_1}|\equiv 1 \bmod p^{n_0}$, 
    \item $|\bar k_{v_1}|\not\equiv 1 \bmod p^{n_0+1}$,
    \item $|\bar k_{v_2}|\not\equiv 1 \bmod p^{m_0}$.
  \end{enumerate}
Suppose further that $K/k$ is a cyclic extension such that
  \begin{enumerate}
    \setcounter{enumi}{4}
  \item $[K:k] = na$, 
    $[K:k]_{v_i} = n$ for $i=0,1,2$, 
  \item $K_{v_0}/k_{v_0}$ and $K_{v_2}/k_{v_2}$ are inertial,
  \item $p^{n_0}$ divides $e_{v_1}(K/k)$.
  \end{enumerate}
If
\marginpar{The unique extension of $w_1$ to $L$ is not required as it is not needed in Lemma \ref{lem:brussel-abelian}.}
$L/K$ is a field extension, $[L:K]=m$,
such that there are valuations $w_i\in\V0(K)$, $w_i\mid v_i$,
that extend uniquely to $L$ for $i=0,1,2$,
then $L/k$ is not Galois.
\end{lemma}
\begin{proof}
Suppose $L/k$ is a Galois extension, $K\subseteq L$, $[L:k]=nma$ and $[L:k]_{v_i}=nm$ for $i=0,1,2$.
The decomposition fields $Z_{w_i}(K/k)$ all have degree $a$ over $k$.
Since $K/k$ is cyclic, there is a unique subfield $Z\subseteq K$ with $[Z:k]=a$,
hence $Z$ is the common decomposition field of $w_0,w_1,w_2$ in $K/k$.
If $L/Z$ is not Galois then $L/k$ is not Galois,
so we may assume \wolg that $a=1$ and $Z=k$.

Let $G=\Gal(L/k)$, $|G|=nm$.
Since $[L:k]_{v_i}=nm$,
$\Gal(L_{v_i}/k_{v_i})\cong G$ for $i=0,1,2$.
We shortly write $e_i$ for $e_{v_i}(L/k)$,
$f_i$ for $f_{v_i}(L/k)$,
and $T_i$ for $T_{v_i}(L/k)$.
$(6)$ implies $e_0\mid m$,
thus $|\bar k_{v_0}|\equiv 1 \bmod e_0$ by $(1)$.
Since $\charak\bar k_{v_0}\ndiv m$,
Proposition \ref{prop:albert-local-fields} shows that $G$ is abelian.

There are $K_0\subseteq K$ and $K_0\subseteq L_0\subseteq L$
with $[L_0:K_0]=p^{m_0}$ and $[K_0:k]=p^{n_0}$,
and all extensions are abelian.
$(6)$ and $(7)$ show that 
$v_1$ is totally ramified and $v_2$ is inertial in $K_0/k$.
Therefore by Lemma \ref{lem:brussel-abelian},
$L_0/k$ can not be abelian, a contradiction.
Thus $L/k$ is not Galois.
Like in the proof of Lemma \ref{lem:brussel-abelian},
$(2)$ is not used here,
but it is implicit in $(7)$ by Proposition \ref{prop:albert-local-fields}.
\end{proof}

Recall the following application of the Chebotarev Density Theorem.

\begin{remark}
  \label{rem:chebotarev}
  Let $K/k$ be a Galois extension of number fields, and let $\sigma\in\Gal(K/k)$.
  The Chebotarev density theorem 
  (see \cite[Chapter VII, Theorem~13.4, \page~569]{neukirch:alg-zt}) states
  that there are infinitely many unramified valuations $v\in\V0(k)$, 
  such that $\sigma$ is the Frobenius automorphism for a valuation $w\in\V0(K)$ with $w\mid v$.

  If $\sigma$ is the Frobenius automorphism for an unramified valuation $w\in\V0(K)$,
  it means that $\sigma$ generates the decomposition group of $w$,
  hence $Z=\Fix(\sigma)$ is the decomposition field of $w$.
  Thus, if $k\subseteq Z_0\subseteq Z$ is any subfield with $Z_0/k$ Galois,
  then $v$ splits completely in $Z_0/k$.
  Furthermore, $v$ does not split completely in $K/k$ provided that $\sigma\neq\id$.

  Therefore, there are infinitely many unramified valuations $v\in\V0(k)$ that split completely in $K/k$.
  Furthermore, if $k\subseteq Z\subsetneq K$ is any proper subfield with $Z/k$ Galois,
  then there are infinitely many unramified valuations $v\in\V0(k)$ that split completely in $Z/k$ but not in $K/k$. 
\end{remark}

We can now give the

\begin{proof}[Proof of Theorem \ref{thm:brussel-field-existence-global}]
Let $m=p^{m_0}m'$ with $p\ndiv m'$.
The numbers $n_0$ and $m$ are chosen such that the extensions 
$k(\mu_{p^{m_0}})/k$, $k(\mu_{p^{n_0+1}})/k(\mu_{p^{n_0}})$ and $k(\mu_{p^m})/k$ are non-trivial.
Obviously, the extensions $k(\mu_{p^m})/k$, $k(\mu_{p^{m_0}})/k$, $k(\mu_{p^{n_0+1}})/k$ and $k(\mu_{p^{n_0}})/k$ are Galois.
By Remark \ref{rem:chebotarev}, applied to these three extensions,
there are valuations $v_0,v_1,v_2,v_3\in\V0(k)$
with $\charak\bar k_{v_i} \ndiv 2nm$ for $0\leq i\leq 3$, such that
  $v_0$ splits completely in $k(\mu_{p^m})/k$,
  $v_2$ does not split completely in $k(\mu_{p^{m_0}})/k$ 
  and $v_1$ splits completely in $k(\mu_{p^{n_0}})/k$ but not in $k(\mu_{p^{n_0+1}})/k$.
  By Remark \ref{rem:mu-e-classification}, this implies
  the properties $(1)$--$(4)$ of Lemma~\ref{lem:brussel-no-galois-ext}.

  Let $K_{v_0},K_{v_2},K_{v_3}$ be the inertial extensions (unique up to isomorphism) 
  of $k_{v_0},k_{v_2},k_{v_3}$ respectively of degree $n$.
  These extensions are cyclic.
  By the property $(2)$ of Lemma \ref{lem:brussel-no-galois-ext},
  that was already shown above, $\mu_{p^{n_0}}\subset k_{v_1}$.
  Therefore, there exists a totally ramified cyclic extension ${K_0}_{v_1}/k_{v_1}$ of degree $p^{n_0}$
  (\eg ${K_0}_{v_1}=k_{v_1}(\sqrt[{p^{n_0}}]{a})$ for any prime element $a$ of $k_{v_1}$).
  Let $L_{v_1}$ be the inertial extension of $k_{v_1}$ of degree $n'$,
  $L_{v_1}/k_{v_1}$ cyclic,
  and set $K_{v_1}:={K_0}_{v_1}L_{v_1}$.
  Then, $K_{v_1}/k_{v_1}$ is cyclic of degree $n$ because $p^{n_0}$ and $n'$ are relatively prime.

  By the Grunwald-Wang Theorem (cf. \cite[Chapter X, Theorem 5, \page 103]{artin-tate})
  there exists a cyclic extension $K/k$ of degree $na$ 
  with local completions $K_{v_0},\ldots,K_{v_3}$
  at $v_0,\ldots,v_3$ respectively.
  Note that we are not in the special case of the Grunwald-Wang Theorem, 
  since $\charak\bar k_{v_i}\ndiv 2$, $0\leq i\leq 3$.
  Then, $K/k$ satisfies the properties $(5)$--$(7)$ from Lemma \ref{lem:brussel-no-galois-ext}
  and $[K:k]_{v_3}=n$.
  Therefore $K$ and $v_0,\ldots,v_3$ have the desired property by Lemma \ref{lem:brussel-no-galois-ext}.
\end{proof}

\begin{ex}
  \label{ex:K/Q-cyclic}
Let $k=\Q$ and 
$$K=\Q(\zeta+\zeta^{-1}),$$ 
where $\zeta$ denotes a primitive $7$-th root of unity.
Then $K$ is the maximal real subfield of $\Q(\zeta)$ which is a Galois extension of degree $3$ over $\Q$.
Shortly write  $\alpha=\zeta+\zeta^{-1}$ in the following.
The minimal polynomial of $\alpha$ over $\Q$ is $f(x)=x^3+x^2-2x-1\in\Q[x]$,
and an automorphism of $K/\Q$ is given by $\sigma(\alpha)=\alpha^2-2$.

Choose $v_1$ and $v_2$ the $7$- and $2$-adic valuation on $\Q$ respectively.
It is well known from the theory of cyclotomic fields 
that $v_1$ is totally ramified
and $v_2$ is inertial in $\Q(\zeta)/\Q$,
hence also in $K/\Q$
(see \eg \cite[Satz 7.12 and 7.13]{neukirch:alg-zt}).
Thus, if we set $p=3$ and $n_0=m_0=1$,
then $K/k$ satisfies the conditions of Lemma \ref{lem:brussel-abelian}.
Therefore any field $L\supseteq K$ with $[L:K]=3$, such that $v_1,v_2$ extend uniquely to valuations on $L$, is not abelian over $\Q$.
Since any group of order $9$ is abelian,
this shows that any such $L$ is also not Galois over $\Q$.
\end{ex}

\section{Non-embeddable abelian extensions}
\label{sec:Non-embedd-abel}

We will prove in this section the existence of abelian extensions of number fields with group $\Zn{2}\times\Zn{2}$
that do not embed into Galois extensions of degree $8$
with certain local degrees $8$.
This will be done by giving two examples,
we will not prove a general existence theorem like Theorem \ref{thm:brussel-field-existence-global}.
First we need the following two lemmas.
By a real field we mean a subfield of $\R$.

\begin{lemma}
  \label{lem:real-galois-ext}
Let $k$ be a real field and $L/k$ a finite Galois extension.
Then $L$ is real or there exists a real subfield $K$, $k\subseteq K\subseteq L$,
such that $[L:K]=2$.
\end{lemma}
\begin{proof}
Choose $K=\Fix(\sigma|_L)$ where 
$\sigma\in\Aut_k(\C)$ denotes the complex conjugation.
Then $K$ is real and $[L:K]$ divides $2$.
\end{proof}

\begin{lemma}
  \label{lem:no-galois-ext}
Let $k$ be a real number field,
and let $K/k$ be a Galois extension with 
$\Gal(K/k)\cong\Zn{2}\times\Zn{2}$
and $K$ not real.
Suppose there are valuations $v_1,v_2\in\V0(k)$ such that
\begin{enumerate}
\item \label{item:5} 
$|\bar{k}_{v_i}|\equiv 3\bmod 4$ for $i=1,2$, 
\item \label{item:6} %
$v_1$ and $v_2$ extend uniquely to valuations on $K$,
\item \label{item:7}
$T_{v_1}(K/k)\neq T_{v_2}(K/k)$.
\end{enumerate}
If $L/K$ is a field extension, $[L:K]=2$, 
such that $v_1$ and $v_2$ extend uniquely to valuations on $L$,
then $L/k$ is not Galois.
\end{lemma}
\begin{proof}
Assume there exists such an $L$ with $L/k$ Galois with $\Gal(L/k)=G$.
Then $|G|=8$.
We now show that all isomorphism classes of groups of order $8$ yield a contradiction.
Note that by (\ref{item:5}), $\charak \bar k_{v_i}\neq 2$,
thus $v_i$ is tamely ramified in $L/k$ for $i=1,2$. \\
Case 1 :
$G$ is the dihedral group.
Then $G$ has only one normal subgroup $I$
such that $I$ and $G/I$ are cyclic.
Therefore $T_{v_1}(L/k)=T_{v_2}(L/k)$, 
hence $T_{v_1}(K/k)=T_{v_2}(K/k)=T_{v_1}(L/k)\cap K$
which contradicts (\ref{item:7}). \\
Case 2 :
$G$ is the quaternion group.
Since $K$ is not real, $L$ is not real.
Hence by Lemma \ref{lem:real-galois-ext}
there is a real subfield $L_0\subset L$
with $[L:L_0]=2$.
Then $L_0\neq K$ and $[L_0:k]=4$.
But $G$ contains only one subgroup of order $2$, 
a contradiction. \\
Case 3 :
$G$ is abelian.
Let $L_0$ be as in case $2$.
Since $G$ is abelian, $L_0/k$ is Galois.
We show that \wolg $v_1$ is totally ramified in $L_0/k$.
Then $\mu_4\subset \bar k_{v_1}$ by Proposition \ref{prop:albert-local-fields},
hence $|\bar k_{v_1}|\equiv 1\bmod 4$ by (\ref{eq:mu-e-in-bar-k}).
This contradicts (\ref{item:5}).

Since $\Gal(K/k)$ is not cyclic,
$[T_{v_i}(K/k):k]=2$ for $i=1,2$,
and from $L_0\neq K$ we get $[L_0\cap K:k]\leq 2$.
Therefore, because of (\ref{item:7}), we can assume \wolg that
$L_0\cap T_{v_1}(K/k)=k$.
Since $T_{v_1}(L/k)$ is cyclic over $k$
and of prime power degree, its subfields are linearly ordered,
so $T_{v_1}(K/k)$ is the unique subfield of $T_{v_1}(L/k)$ of degree $2$ over $k$.
Therefore, $L_0\cap T_{v_1}(K/k)=k$ implies $L_0\cap T_{v_1}(L/k)=k$,
\ie $v_1$ is totally ramified in $L_0/k$.
\end{proof}

It shall be mentioned here that the question whether a Galois extension $K/k$
with group $\Zn{2}\times\Zn{2}$
embeds into a Galois extension with the quaternion group is completely treated 
in \cite[\S VI]{witt:konstr-gal-koerpen}.

\begin{ex}
\label{ex:not-embed-abel}
Let $k=\Q$ and 
$$K=\Q(\sqrt{3},\sqrt{-7}).$$
Then $[K:\Q]=4$ and $K/\Q$ is abelian with 
$\Gal(K/\Q)\cong\Zn{2}\times\Zn{2}$.
Let $v_1,v_2\in\V0(\Q)$ be the $3$- and $7$-adic valuation respectively.
Obviously $v_1$ and $v_2$ are totally ramified in $\Q(\sqrt{3})/\Q$
and $\Q(\sqrt{-7})/\Q$ respectively.
Furthermore $v_1$ and $v_2$ are inertial in $\Q(\sqrt{-7})/\Q$
and $\Q(\sqrt{3})/\Q$ respectively,
since $-7$ and $3$ are not squares modulo $3$ and $7$ respectively.
Thus, $v_1$ and $v_2$ extend uniquely to valuations on $K$
and $T_{v_1}(K/\Q)\neq T_{v_2}(K/\Q)$.
Obviously, $|\bar \Q_{v_i}|\equiv 3\bmod 4$ for $i=1,2$, 
$\Q$ is real and $K$ is not real.
Therefore $K/\Q$ and $v_1,v_2$ satisfy the properties $(1)$--$(3)$ of Lemma \ref{lem:no-galois-ext}.
\end{ex}

\begin{ex}
\label{ex:not-embed-abel-2}
Let 
$$k=\Q(\sqrt{37}) \qt{and} K=k(\sqrt{3},\sqrt{-7}).$$
Let $v_1,v_2\in\V0(k)$ be any extensions of the 
$3$- and $7$-adic valuations respectively to $k$.
Since $37$ is a square modulo $3$ and $7$,
$v_1$ and $v_2$ split completely in $k/\Q$,
\ie $k_{v_1}=\Q_3$ and $k_{v_2}=\Q_7$.
Therefore we can apply here to $K/k$ the same arguments as in Example \ref{ex:not-embed-abel}.
Since $k$ is real, this shows that $K/k$ and $v_1,v_2$ satisfy the properties $(1)$--$(3)$ of Lemma \ref{lem:no-galois-ext}.
\end{ex}

\section{Existence of noncrossed product division algebras}

The following lemma shows how we will apply the results from \S\S~\ref{sec:Non-embedd-cycl}--\ref{sec:Non-embedd-abel}.

\begin{lemma}
  \label{lem:appl-noncr-prod}
Let $K/k$ be a finite extension of global fields, 
$w_1,\ldots,w_r\in\Val(K)$, and $m\in\N$.
Suppose that for any field extension $L/K$ with $[L:K]=m$,
such that $w_1,\ldots,w_r$ extend uniquely to valuations on $L$,
$L$ is not Galois over $k$.
If $\wt D\in\calD(K)$ with $\ind\wt D=m$, such that $\wt D_{w_i}$ is a division algebra for $1\leq i\leq r$,
then $\wt D$ does not contain a maximal subfield that is Galois over $k$.
Moreover if $F$ is a valued field with $\bar F=k$,
then any inertially split $D\in\calD(F)$ with $\bar D\cong\wt D$ is a noncrossed product.
\end{lemma}
\begin{proof}
If $L$ is a maximal subfield of $\wt D$,
then  $[L:K]=\ind\wt D=m$,
and since $\wt D_{w_i}$ is a division algebra,
$w_i$ extends uniquely to $L$ for $1\leq i\leq r$.
Therefore $L$ is not Galois over $k$.
If $D\in\calD(F)$ is inertially split with $\bar D\cong\wt D$,
then $D$ is a noncrossed product by Theorem~\ref{thm:crossed-product}.
\end{proof}

The main theorem of this section is

\begin{theorem}
  \label{thm:ex-global-noncr-prod}
Let $F$ be a valued field such that
$F$ has the inertial lift property,
$\bar F=k$ is a number field and 
$\Gamma_F$ is discrete.
Let $p$ be a prime number and let $r,s\in\N_0$ be maximal such that 
$\mu_{p^r}\subset k^\times$ and  $\mu_{p^s}\subset k(\mu_{p^{r+1}})^\times$.
For any $n_0,n',m,a\in\N$ such that $n_0=r$ or $n_0\geq s$,
$p\ndiv n', p^{r+1}\mid m$ and $a\mid m$,
there exists a noncrossed product division algebra over $F$ 
of index $nma$ and exponent $nm$, where $n=p^{n_0}n'$.
\end{theorem}
\begin{proof}
Choose $K/k$ and $v_0,\ldots,v_3\in\V0(k)$ by Theorem~\ref{thm:brussel-field-existence-global}.
Then $[K:k]=na$ and $[K:k]_{v_i}=n$ for $0\leq i\leq 3$.
Let $w_1,\ldots,w_r\in\Val(K)$ be all the valuations $w\in\Val(K)$ with $w\mid v_i$ for some $0\leq i\leq 2$.
Then, by Theorem~\ref{thm:brussel-field-existence-global},
$K/k$ has the property that for any field extension $L/K$ with $[L:K]=m$
such that $w_1,\ldots,w_r$ extend uniquely to~$L$,
$L/k$ is not Galois.

Let
\marginpar{The noncrossed product argument requires in fact no condition on the invariant of $\wt D$ at $w_1$,
by the respective modification of Theorem \ref{thm:brussel-field-existence-global}.}
$\wt D\in\calD(K)$ with
$\inv_w \wt D=\frac{1}{m}$ for all $w\in\Val(K)$ with $w\mid v_0$ or $w\mid v_1$,
$\inv_w \wt D=-\frac{1}{m}$ for all $w\in\Val(K)$ with $w\mid v_2$ or $w\mid v_3$,
and $\inv_w D=0$ for all other $w\in\Val(K)$.
The sum of these local invariants is zero 
because each $v_i$ has the same number of extensions to $K$
(namely $a$) for $i=0,\ldots,3$.
Such $\wt D$ then exists by (\ref{fact:inv-surj}).
In particular, $\wt D_{w_i}$ is a division algebra for $1\leq i\leq r$.

Let $A\in\calD(k)$ with 
$\inv_{v_0}A=\inv_{v_1}A=\frac{1}{nm}$,
$\inv_{v_2}A=\inv_{v_3}A=-\frac{1}{nm}$
and $\inv_{v}A=0$ for all other $v\in\Val(k)$.
Such $A$ exists 
since the sum of these local invariants is obviously zero.
The formula (\ref{eq:glob-scal-ext}) then yields
$\inv_w A_K=\inv_w \wt D$ for all $w\in\Val(K)$,
thus $A_K\cong\wt D$ by the Albert-Hasse-Brauer-Noether Theorem.

$K/k$ is cyclic and we have shown that $\wt D\cong A_K$ for some $A\in\calA(k)$,
\ie the conditions (1) and (2) of Theorem~\ref{theorem:lift} are satisfied.
Since $\Gamma_F$ is discrete,
also (3) is satisfied (\cf Re\-mark~\ref{remark:cyclic}).
Hence, by Theorem~\ref{theorem:lift}, there exists an inertially split $D\in\calD(F)$ with $\bar D=\wt D$, 
$\ind D=[K:k]\ind\wt D=nma$ and
$\exp D=\lcm([K:k],\exp A)=\lcm(na,nm)=nm$.
By Lemma~\ref{lem:appl-noncr-prod}, $D$ is a noncrossed product.
\end{proof}

\begin{cor}
\label{cor:ex-global-noncr-prod}
Let $F, p, r$ and $s$ be as in Theorem~\ref{thm:ex-global-noncr-prod}.
For any natural numbers $e$ and $d$ such that
\begin{align*}
p^{s+1}\mid e\mid d\mid \frac{e^2}{p^{s}} &\qt{if $r=0$, and} \\
p^{2r+1}\mid e\mid d\mid \frac{e^2}{p^r} &\qt{if $r>0$,}
\end{align*}
there exists a noncrossed product division algebra over $F$ 
of index $d$ and exponent $e$.
\end{cor}
\begin{proof}
Apply Theorem~\ref{thm:ex-global-noncr-prod} with
$n_0=s$ if $r=0$, $n_0=r$ if $r>0$, $n'=1$ and $m=\frac{l}{p^{n_0}}$.
\end{proof}

\begin{remark} 
  \label{rem:ex-are-brussel}
(1) The noncrossed products in Theorem~\ref{thm:ex-global-noncr-prod}
and Corollary \ref{cor:ex-global-noncr-prod}
are already obtained from cyclic extensions $K/k$.
Thus, Theorem~\ref{theorem:lift} has been used only in the special case
that $Z(\bar D)/\bar F$ is cyclic.
The advantage of this case is,
according to Remark \ref{remark:cyclic},
that the condition (3) in Theorem~\ref{theorem:lift}
($\Gal(Z(\bar D)/\bar F)$ embeds into $\Gamma_F/m\Gamma_F$)
is always satisfied if $\Gamma_F$ is discrete. 

(2) In the cases $F=k(t)$ and $F=k(\!(t)\!)$ (and $n,m,a$ $p$-powers),
the noncrossed products in Theorem~\ref{thm:ex-global-noncr-prod}
are precisely the ones from \cite{brussel:noncr-prod}.
Recall that in \cite{brussel:noncr-prod} the noncrossed products over $k(t)$, for example,
are the underlying division algebras of tensor products of the form
$A^{k(t)}\otimes_{k(t)}(K(t)/k(t),\sigma,t)$,
where $A\in\calD(k)$, $K/k$ is cyclic with $\Gal(K/k)=\gen{\sigma}$,
and any extension $L/k$ of degree $[K:k]\ind A_K$ that contains $K$ and splits $A$ is not Galois.
This is equivalent to that any maximal subfield of $A_K$ is not Galois over $k$,
thus $A_K$ coincides with the $\wt D$ here.
The tensor product
$A^{k(t)}\otimes_{k(t)}(K(t)/k(t),\sigma,t)$
is hidden in the proof of Theorem~\ref{theorem:lift}.
There, $D$ is the underlying division algebra of a tensor product $I\otimes_{k(t)} N$.
The $I$ coincides with $A^{k(t)}$,
and the $N$ coincides with $(K(t)/k(t),\sigma,t)$.
The same analogy holds for $F=k(\!(t)\!)$.
\end{remark}

\begin{ex}
\label{ex:noncr-prod-cycl}
Let $k,K,f,\alpha,v_1$ and $v_2$ be as in Example \ref{ex:K/Q-cyclic},
\ie $k=\Q, K=\Q(\alpha)$, where $\alpha$ is a root of 
$f(x)=x^3+x^2-2x-1$,
$v_1$ the $7$-adic and $v_2$ the $2$-adic valuation on $\Q$. 
We have already seen that $v_1$ is totally ramified and $v_2$ is inertial in $K/k$.
Let $w_1$ and $w_2$ be the unique extensions of $v_1,v_2$ to $K$ respectively.
Note that $\alpha\in\OO_K$ since $f\in\Z[x]$.
$v_1$ is totally ramified in $K/k$ and $f(x)\equiv (x-2)^3\bmod 7$, 
thus $\bar K_{w_1}=\bar \Q_{v_1}=\F_7$ and $\bar\alpha=2$ in $\bar K_{w_1}$.
$f(x)$ is irreducible modulo $2$,
thus $\bar K_{w_2}=\bar \Q_{v_2}(\bar\alpha)=\F_2(\bar\alpha)$.
For $\pi=\alpha^2+2\alpha-1$ we have $\No_{K/k}(\pi)=7$.
If follows
$\PP_{w_1}=\pi\OO_{w_1}$ and $\PP_{w_2}=2\OO_{w_2}$.

Now let $L=K(\beta)$ for a root $\beta$ of
\[
g(x)=x^3+(\alpha-2)x^2-(\alpha+1)x+1\in \OO_K[x].
\]
Then $\beta\in\OO_L$.
It is easily checked that $\bar g$ has no roots in $\bar K_{w_1}$ and $\bar K_{w_2}$,
thus $[L:K]=3$ and $w_1,w_2$ are inertial in $L/K$.
Moreover, for the unique extensions $w'_1,w'_2$ of $w_1,w_2$ to $L$ respectively,
$\bar L_{w'_1}=\bar K_{w_1}(\bar\beta)=\F_7(\bar\beta)$,
$\bar L_{w'_2}=\bar K_{w_2}(\bar\beta)=\F_2(\bar\alpha,\bar\beta)$,
$\PP_{w'_1}=\pi\OO_{w'_1}$ and $\PP_{w'_2}=2\OO_{w'_2}$.
To see that $L/K$ is a Galois extension,
check $g(\beta^2+(\alpha-2)\beta-\alpha)=g(-\beta^2+(-\alpha+1)\beta+2)=0$.
Let $\tau\in\Gal(L/K)$ with $\tau(\beta)=\beta^2+(\alpha-2)\beta-\alpha$ and
$\tau^2(\beta)=-\beta^2+(-\alpha+1)\beta+2$,
and let
$$\wt D=(L/K,\tau,2\pi).$$
Recall that the Frobenius automorphism of $w'_i$ in $L/K$ is the automorphism 
$\varphi\in\Gal(L/K)$ with $\bar{\varphi(\beta)}=\bar\beta^q$ in $\bar L_{w'_i}$, 
where $q=|\bar K_{w_i}|$.
A routine computation shows that
$\bar\beta^7=\bar\beta^2-2=\bar{\tau(\beta)}$ in $\bar L_{w'_1}$
and
$\bar\beta^8=\bar\beta^2+(\bar\alpha+1)\bar\beta=\bar{\tau^2(\beta)}$ in $\bar L_{w'_2}$,
thus $\tau$ is the Frobenius automorphism of $w_1$
and $\tau^2$ is the Frobenius automorphism of $w_2$.
Since $2\pi$ is a prime element of $w_1$ and $w_2$, this shows
$\inv_{w_1} \wt D=\frac{1}{3}$ and $\inv_{w_2} \wt D=\frac{2}{3}$
(\cf \cite[\S 17.10]{pierce:ass-alg}).
In particular, $\wt D$, $\wt D_{w_1}$ and $\wt D_{w_2}$ are division algebras with index $3$.
It was shown in Example \ref{ex:K/Q-cyclic} that for any $L/F$ 
with $w_1,w_2$ extend uniquely to $L$, 
$L$ is not Galois over $\Q$.
Thus by Lemma~\ref{lem:appl-noncr-prod},
any maximal subfield of $D$ is not Galois over $\Q$.
Moreover, if $F$ is a valued field with residue field $\Q$,
then any inertially split $D\in\calD(F)$
with $\bar D=\wt D$ is a noncrossed product.

We now apply Theorem~\ref{theorem:lift} to show that such $D$ exists.
It remains to verify that $\wt D\cong A_K$ for some $A\in\calD(\Q)$.
To do so we compute all local invariants of $\wt D$ as follows.
Since $\disc(g)=20\alpha^2-19\alpha+46=(\alpha^2-\alpha+7)^2$ and
$\No_{K/\Q}(\alpha^2-\alpha+7)=673$ is a prime number,
the prime ideal $\pp=(\alpha^2-\alpha+7)\OO_K$ is the only prime ideal of $K$
that possibly divides $\disc(L/K)$.
Then $v_\pp\in\V0(K)$ is the only valuation on $K$ that is possibly ramified in $L/K$.
But $2\pi$ is a unit with respect to all $w\in\V0(K)$ with $w\neq w_1,w_2$.
Therefore, $\inv_w \wt D=0$ for all $w\in\V0(K)$ with $w\neq w_1,w_2,v_\pp$.
Since $\deg \wt D=3$, $\inv_w \wt D=0$ also for all archimedian valuations $w\in\Val(K)$.
Finally, since the sum of all invariants must be zero, $\inv_{v_\pp} \wt D=0$.
The formula (\ref{eq:glob-scal-ext}) and the Albert-Hasse-Brauer-Noether Theorem then show
that $\wt D\cong A_K$ \eg for the $A\in\calD(\Q)$ with
$\inv_{v_1} A = \frac{1}{9}$, $\inv_{v_2} A = -\frac{1}{9}$
and $\inv_v A=0$ for all other $v\in\Val(\Q)$.
Thus by Theorem~\ref{theorem:lift}, there is an inertially split
$D\in\calD(\Q(t))$, $\ind D=\exp D=9$, with residue algebra $\wt D$,
and this $D$ is a noncrossed product by Theorem~\ref{thm:crossed-product}.
Such a $D$ will be constructed in Example~\ref{ex:cycl-noncr-prod} below.
\end{ex}

\begin{ex}
\label{ex:noncr-prod-abel-1}
As in Example~\ref{ex:not-embed-abel},
let $k=\Q$ and $K=\Q(\sqrt{3},\sqrt{-7})$,
and let $v_1,v_2\in\V0(\Q)$ be the $3$- and $7$-adic valuation respectively.
There are unique extensions $w_1$ and $w_2$ of $v_1$ and $v_2$ to $K$ respectively (\cf Example~\ref{ex:not-embed-abel}).
For
\begin{align*}
   \pi_1:=1+\sqrt{3} \in\Q(\sqrt{3}) \qt{and} \pi_2:=\frac{1+\sqrt{-7}}{2} \in\Q(\sqrt{-7}), 
\end{align*}
we have
\begin{align*}
  \No_{\Q(\sqrt{3})/\Q}(\pi_1)=-2 \qt{and} 
  \No_{\Q(\sqrt{-7})/\Q}(\pi_2)=2. 
\end{align*}
Let $\wt D$ be the quaternion algebra $(\frac{a,b}{K})$ with
\begin{alignat*}{2}
  a&:=\sqrt{3}\pi_1 =3+\sqrt{3} &\quad&\in\Q(\sqrt{3}), \\
  b&:=\sqrt{-7}\pi_2 =\frac{-7+\sqrt{-7}}{2} &&\in\Q(\sqrt{-7}),
\end{alignat*}
\ie $\wt D=K\oplus Ki \oplus Kj \oplus Kij$
with $i^2=a,j^2=b,ij=-ji$.
We show that $\wt D_{w_1}$ and $\wt D_{w_2}$ are division algebras.
In particular this implies that $\wt D$ is a division algebra.
We have 
$\sqrt{3},\sqrt{-7},a,b\in\OO_K$,
$[\bar K_{w_1}:\bar\Q_{v_1}]=[\bar K_{w_2}:\bar\Q_{v_2}]=2$ and
$$  \bar K_{w_1}=\bar \Q_{v_1}(\eta), \quad \bar K_{w_2}=\bar\Q_{v_2}(\theta) $$
with $\eta=\sqrt{-7}+\PP_{w_1}$ and $\theta=\sqrt{3}+\PP_{w_2}$.
Further,
$\bar b=1-\eta$ in $\bar K_{w_1}$ and
$\bar a=3+\theta$ in $\bar K_{w_2}$.
Since $\No_{\bar K_{w_1}/\bar\Q_{v_1}}(1-\eta)=-1$ 
and $\No_{\bar K_{w_2}/\bar\Q_{v_2}}(3+\theta)=-1$ are not squares
in $\bar\Q_{v_1}$ and $\bar\Q_{v_2}$ respectively,
$b$ and $a$ are not squares in $\bar K_{w_1}$ and in $\bar K_{w_2}$ respectively.
This shows that $w_1$ and $w_2$ are inertial in $K(\sqrt{b})/K$ and $K(\sqrt{a})/K$ respectively.
But since $\pi_1$ and $\pi_2$ are units with respect to $w_1$ and $w_2$,
$a$ and $b$ are prime elements with respect to $w_1$ and $w_2$,
so $a$ and $b$ cannot be norms in $K_{w_1}(\sqrt{b})/K_{w_1}$ and $K_{w_2}(\sqrt{a})/K_{w_2}$ respectively.
This implies that $(\frac{a,b}{K_{w_1}})$ and $(\frac{a,b}{K_{w_2}})$ are division algebras 
(\cf \cite[\S 1.6, Exercise 4 or Corollary 15.1d]{pierce:ass-alg}).

By Lemma~\ref{lem:no-galois-ext} and Lemma~\ref{lem:appl-noncr-prod},
$\wt D$ does not contain a maximal subfield that is Galois over $\Q$.
Moreover, if $F$ is a valued field with residue field $\Q$,
then any inertially split $D\in\calD(F)$ with residue algebra $\wt D$ is a noncrossed product.
We will not show at this point that such $D$ exists,
since an explicit example will be constructed in \S~\ref{sec:An-example-with-8}.
\end{ex}

\begin{ex}
\label{ex:noncr-prod-abel-2}
Let $k=\Q(\sqrt{37})$ and $K=k(\sqrt{3},\sqrt{-7})$.
Let $\wt D=(\frac{a,b}{K})$ with 
$a=3+\sqrt{3}$ and $b=\frac{-7+\sqrt{-7}}{2}$ as in Example~\ref{ex:noncr-prod-abel-1}.
By Example~\ref{ex:not-embed-abel-2} we can apply here the same arguments as in Example~\ref{ex:noncr-prod-abel-1}.
This shows that $\wt D$ is a division algebra that does not contain a maximal subfield that is Galois over $k$.
In particular, $\wt D$ does not contain a maximal subfield that is Galois over $\Q$.
Moreover, if $F$ is a valued field with residue field $k$,
then any inertially split $D\in\calD(F)$ with residue algebra $\wt D$ is a noncrossed product.
This is also true if $F$ is a valued field with residue field $\Q$.
We will not show at this point that such $D$ exists,
since an explicit example will be given in \S~\ref{sec:An-example-with-16}.
\end{ex}

\chapter{Direct Constructions}

The noncrossed product division algebras in Theorem \ref{thm:ex-global-noncr-prod} are obtained as the underlying division algebra of some tensor product $N\otimes_F I$,
which itself is a crossed product (\cf proof of Theorem \ref{theorem:lift}).
The motivation of this chapter is to give constructions that lead directly to the underlying division algebra of $N\otimes_F I$.

First, the construction of generalized crossed products is described
with a focus on the cyclic and abelian generalized crossed products,
which are the important cases for our purposes.
Using generalized crossed products, Theorem \ref{theorem:lift} is reproved in a direct way.
More precisely, a generalized crossed product division algebra is constructed 
that is inertially split with a given residue algebra.
Next, the construction of (iterated) twisted function fields and (iterated) twisted Laurent series rings is described,
and it is shown how these are obtained from abelian factor sets.
For the computation of examples of factor sets
it is necessary to compute automorphisms of simple algebras that extend given
automorphisms of the centre.
For certain symbol algebras this problem reduces to
the solution of a relative norm equation in a field extension.
Finally, two explicit examples of an iterated twisted function field 
(\resp an iterated twisted Laurent series ring) are given
that are noncrossed product division algebras.

\section{Generalized crossed products}
We define factor sets and generalized crossed products 
following \cite{tignol:gen-cr-prod}.
Other papers covering this topic are \cite{yanchevskii:gen-cr-prod} and \cite{jehne:sep-adel-alg}.
Throughout this section let $K/F$ be a finite Galois extension with $\Gal(K/F)=G$
and let $A\in\calA(K)$.

\subsection{Factor sets}

\begin{defn}
  \label{def:factor-set}
A \emph{factor set}
\index{factor set}
of $G$ in $A^\times$ is a pair $(\omega,f)$ of maps 
\begin{eqnarray*}
  \omega &:& G\lra \Aut_F(A), \quad \sigma\lms \omega_\sigma, \\
  f &:& G\times G\lra A^\times, \quad (\sigma,\tau)\lms f(\sigma,\tau),
\end{eqnarray*}
such that for all $\rho,\sigma,\tau\in G$ :
\begin{subequations}
\label{eq:factor-set}
\begin{gather}
\label{eq:factor-set-1}
\omega_\sigma|_K = \sigma ,\\
\label{eq:factor-set-2}
\omega_\sigma\omega_\tau = \iota_{f(\sigma,\tau)}\omega_{\sigma\tau} ,\\
\label{eq:factor-set-3}
\omega_\rho(f(\sigma,\tau))f(\rho,\sigma\tau) = f(\rho,\sigma)f(\rho\sigma,\tau) .
\end{gather}
\end{subequations}
Denote by $\calF(G,A^\times)$ the set of all factor sets of $G$ in $A^\times$.
We say that $(\omega,f)$ and $f$ are \emph{normalized} 
\index{factor set!normalized}
if $f(\id,\id)=1$.
If $G$ is cyclic (\resp abelian) we call $(\omega,f)$ a \emph{cyclic} (\resp \emph{abelian}) factor set.
\index{factor set!cyclic}
\index{factor set!abelian}
\end{defn}

Note that in general,  $\calF(G,A^\times)$ may be empty.

\begin{remark}
The relations (\ref{eq:factor-set}) imply that
$f(\id,\sigma)=f(\id,\id)$ and $f(\sigma,\id)=\omega_\sigma(f(\id,\id))$
for all $\sigma\in G$.
In particular, if $f$ is normalized, then
$f(\sigma,\id)=f(\id,\sigma)=f(\id,\id)=1$ for all $\sigma\in G$.
\end{remark}

\begin{defn}
  \label{def:cohomologous-factor-sets}
Two factor sets $(\omega,f)$ and $(\eta,g)$ from $\calF(G,A^\times)$
are said to be \emph{cohomologous},
\index{factor set!cohomologous}
written $(\omega,f)\sim(\eta,g)$,
if there exists a family $\{m_\sigma\}_{\sigma\in G}$ in $A^\times$
such that
\begin{subequations}
\label{eq:cohomologous}
\begin{align}
\label{eq:cohomologous-1}
\eta_\sigma&=\iota_{m_\sigma}\omega_\sigma && \text{for all $\sigma\in G$,} \\
\label{eq:cohomologous-2}
g(\sigma,\tau)&=m_\sigma\omega_\sigma(m_\tau)f(\sigma,\tau)m_{\sigma\tau}^{-1}
&& \text{for all $\sigma,\tau\in G$.} 
\end{align}
\end{subequations}
\end{defn}

\begin{remark}
\label{rem:cohomology-equiv-relation}
 The cohomology relation $\sim$ is an equivalence relation on the set $\calF(G,A^\times)$
and its quotient will be denoted by $\calH(G,A^\times)$.
We write $[(\omega,f)]$ for the class of $(\omega,f)\in\calF(G,A^\times)$ in $\calH(G,A^\times)$.
\end{remark}

\begin{prop}
  \label{prop:cohomologous-factor-sets}
Suppose that $(\omega,f)\in\calF(G,A^\times)$.
\begin{enumerate}
\item 
For any family $\{m_\sigma\}_{\sigma\in G}$ in $A^\times$ the relations
(\ref{eq:cohomologous}) define a factor set $(\eta,g)\in\calF(G,\ug A)$ 
that is cohomologous to $(\omega,f)$.
\item 
For any map $\eta:G\to\Aut_F(A)$ with $\eta_\sigma|_K=\sigma$ for all $\sigma\in G$,
there exists a factor set $(\eta,g)\in\calF(G,A^\times)$
such that $(\omega,f)\sim(\eta,g)$.
\item
$(\omega,f)$ is cohomologous to a normalized factor set $(\eta,g)\in\calF(G,A^\times)$.
\item 
For any $2$-cocycle $c\in Z^2(G,K^\times)$,
a factor set $(\omega,fc)\in\calF(G,A^\times)$ is defined by
$fc(\sigma,\tau):=f(\sigma,\tau)c(\sigma,\tau)$.
Moreover, $(\omega,fc)\sim(\omega,f)$ if and only if $c$ is a $2$-coboundary,  
\ie $c\in B^2(G,K^\times)$.
\item 
Let $(\eta,g)\in\calF(G,A^\times)$ be any factor set.
There exists a $2$-cocycle $c\in Z^2(G,K^\times)$ such that $(\eta,g)\sim(\omega,fc)$.
\end{enumerate}
\end{prop}
\begin{proof}
$(1)$ Let $\{m_\sigma\}_{\sigma\in G}$ be any family in $A^\times$
and define $\eta$ and $g$ by (\ref{eq:cohomologous}).
We have to show that $(\eta,g)\in\calF(G,A^\times)$.
(\ref{eq:factor-set-1}) and (\ref{eq:factor-set-2}) are easily checked.
Using $f(\rho,\sigma)\omega_{\rho\sigma}(m_\tau)=\omega_\rho(\omega_\sigma(m_\tau))f(\rho,\sigma)$, 
which follows from (\ref{eq:factor-set-2}),
and (\ref{eq:factor-set-3}) for $(\omega,f)$ we get
\begin{multline*}
g(\rho,\sigma)g(\rho\sigma,\tau)g(\rho,\sigma\tau)^{-1} \\
  \begin{split}
&= m_\rho\omega_\rho(m_\sigma)f(\rho,\sigma)\omega_{\rho\sigma}(m_\tau)
f(\rho\sigma,\tau)f(\rho,\sigma\tau)^{-1}\omega_\rho(m_{\sigma\tau})^{-1}m_\rho^{-1} \\
&= m_\rho\omega_\rho(m_\sigma)\omega_\rho(\omega_\sigma(m_\tau))
\omega_\rho(f(\sigma,\tau))\omega_\rho(m_{\sigma\tau})^{-1}m_\rho^{-1}
= \eta_\rho(g(\sigma,\tau))
  \end{split}
\end{multline*}
for all $\rho,\sigma,\tau\in G$.
This shows (\ref{eq:factor-set-3}) for $(\eta,g)$, hence $(\eta,g)\in\calF(G,A^\times)$.
\\
$(2)$
Let $\eta:G\to\Aut_F(A)$ be any map with $\eta_\sigma|_K=\sigma$ for all $\sigma\in G$.
Since $\eta_\sigma|_K=\omega_\sigma|_K$, 
there is a family $\{m_\sigma\}_{\sigma\in G}$ in $A^\times$
such that $\eta_\sigma=\iota_{m_\sigma}\omega_\sigma$ for all $\sigma\in G$
by the theorem of Skolem-Noether.
$(1)$ shows that (\ref{eq:cohomologous}) defines a factor set $(\eta,g)\in\calF(G,A^\times)$
such that $(\omega,f)\sim(\eta,g)$.
\\
$(3)$
Choose $m_{\id}=f(\id,\id)^{-1}$ and $m_\sigma=1$ for all $\sigma\in G\bs\set{\id}$.
Let $(\eta,g)$ be defined by (\ref{eq:cohomologous}).
It was shown in (1) that $(\eta,g)\in\calF(G,\ug A)$ is a factor set
that is cohomologous to $(\omega,f)$.
(\ref{eq:factor-set-2}) implies $\omega_{\id}=\iota_{f(\id,\id)}$,
thus $\omega_{\id}(m_{\id})=m_{\id}$.
It follows $g(\id,\id)=m_{\id}m_{\id}m_{\id}^{-1}m_{\id}^{-1}=1$, 
\ie $(\eta,g)$ is normalized.
\\
$(4)$
Since $c(\sigma,\tau)\in K^\times$ for all $\sigma,\tau\in G$, 
$(\omega,fc)$ satisfies (\ref{eq:factor-set-1}) and (\ref{eq:factor-set-2}).
(\ref{eq:factor-set-3}) is easily checked since 
$c$ satisfies the $2$-cocycle condition
\index{cocycle condition}
\[ \rho(c(\sigma,\tau))c(\rho,\sigma\tau)=c(\rho,\sigma)c(\rho\sigma,\tau) 
\qt{for all $\rho,\sigma,\tau\in G$.}\]
Therefore $(\omega,fc)\in\calF(G,A^\times)$.
The $2$-cocycle $c$ %
is a $2$-coboundary if and only if
there is a family $\{m_\sigma\}_{\sigma\in G}$ in $K^\times$ 
such that 
\begin{equation}
  \label{eq:coboundary}
 c(\sigma,\tau)=m_\sigma \sigma(m_\tau)m_{\sigma\tau}^{-1}
\qt{for all $\sigma,\tau\in G$.} 
\end{equation}
The assertion follows since (\ref{eq:cohomologous-1}) with $\eta=\omega$ is equivalent to that the $m_\sigma$ lie in $\ug K$,
and if the $m_\sigma$ lie in $\ug K$, 
then (\ref{eq:cohomologous-2}) with $g=fc$ is equivalent to (\ref{eq:coboundary}).
\\
$(5)$
Let $(\eta,g)\in\calF(G,A^\times)$.
By $(2)$ we may assume \wolg that $\eta=\omega$.
Define $c(\sigma,\tau):=f(\sigma,\tau)^{-1}g(\sigma,\tau)$ for all $\sigma,\tau\in G$.
Then obviously $fc=g$.
Since $\eta=\omega$, (\ref{eq:factor-set-2}) implies that
$\iota_{f(\sigma,\tau)}=\iota_{g(\sigma,\tau)}$,
thus $c(\sigma,\tau)\in K^\times$ for all $\sigma,\tau\in G$.
This also shows that $c(\sigma,\tau)=g(\sigma,\tau)f(\sigma,\tau)^{-1}$
for all $\sigma,\tau\in G$.
We get from (\ref{eq:factor-set-3}),
\begin{equation*}
  \begin{split}
\rho(c(\sigma,\tau))&=\omega_\rho(f(\sigma,\tau))^{-1}\omega_\rho(g(\sigma,\tau)) \\
&= f(\rho,\sigma\tau)f(\rho\sigma,\tau)^{-1}f(\rho,\sigma)^{-1}
  g(\rho,\sigma)g(\rho\sigma,\tau)g(\rho,\sigma\tau)^{-1} \\
&= c(\rho,\sigma)c(\rho\sigma,\tau)c(\rho,\sigma\tau)^{-1},
  \end{split}
\end{equation*}
\ie $c\in Z^2(G,K^\times)$.
\end{proof}

\begin{cor}
\label{cor:calH-H-2}
Suppose that $\calF(G,A^\times)$ is not empty
and let $(\omega,f)$ be an element from $\calF(G,A^\times)$.
There is a bijective map
\begin{equation*}
  H^2(G,K^\times)\lra\calH(G,A^\times), \qquad
  [c]\lms[(\omega,fc)].
\end{equation*}
\end{cor}
\begin{proof}
This is a well-defined injective map by Proposition \ref{prop:cohomologous-factor-sets} $(4)$ 
and it is surjective by Proposition \ref{prop:cohomologous-factor-sets} $(5)$.
This corollary is known from \cite[Corollary 1.8]{tignol:gen-cr-prod},
where it is proved as a corollary from the Product Theorem (Theorem~\ref{thm:prod-formula}) below.
\end{proof}

\subsection{Generalized crossed products}

Like before
let $K/F$ be a finite Galois extension with $\Gal(K/F)=G$
and let $A\in\calA(K)$.

\begin{lemma}
  \label{lem:z-sigma-lin-indep}
Suppose that $B$ is a finite-dimensional $F$-algebra that contains $K$ as a subfield.
If $\set{z_\sigma}_{\sigma\in G}$ is a family in $\ug{B}$
such that
\begin{equation}
\label{eq:z-sigma-set}
\iota_{z_\sigma}|_K=\sigma \qt{for all $\sigma\in G$,}
\end{equation}
then $\set{z_\sigma}_{\sigma\in G}$ is a free set (\ie left and right linearly independent) over $C_B(K)$.
\end{lemma}
\begin{proof}
Suppose $\set{z_\sigma}_{\sigma\in G}$ is left (right is analogous)
linearly dependent over $C_B(K)$.
Let $\emptyset\neq M\subseteq G$ be minimal such that there is a family 
$\set{a_\sigma}_{\sigma\in M}$ in $C_B(K)\bs\set{0}$
with $\sum_{\sigma\in M} a_\sigma z_\sigma=0$.
Since the $z_\sigma$ lie in $\ug B$, we have $|M|\geq 2$.
Let $\sigma,\tau\in M$ with $\sigma\neq\tau$,
let $a\in K$ with $\sigma(a)\neq\tau(a)$,
and let $M_0=M\bs\set{\sigma,\tau}$.
Then $\sum_{\rho\in M} a_\rho z_\rho=0$ implies
\begin{align*}  
0=(\sum_{\rho\in M} a_\rho z_\rho)a&=\sigma(a)a_\sigma z_\sigma+\tau(a)a_\tau z_\tau+\sum_{\rho\in M_0}\rho(a) a_\rho z_\rho
\intertext{and}
0=\sigma(a)(\sum_{\rho\in M} a_\rho z_\rho)&=\sigma(a)a_\sigma z_\sigma+\sigma(a)a_\tau z_\tau+\sum_{\rho\in M_0}\sigma(a) a_\rho z_\rho, 
\end{align*}
thus
\begin{gather*}
0=(\sigma(a)-\tau(a))a_\tau z_\tau+\sum_{\sigma\in M_0}(\sigma(a)-\rho(a))a_\rho z_\rho.
\end{gather*}
Since $\sigma(a)\neq\tau(a)$, this contradicts the minimality of $M$.
\end{proof}

\begin{prop}
\label{prop:gen-cr-prod-def}
Suppose that $(\omega,f)\in\calF(G,\ug A)$
and let $\set{z_\sigma}_{\sigma\in G}$ be a family of indeterminates.
A multiplication (extending the multiplication in $A$)
is defined on the $F$-space $B:=\bigoplus_{\sigma\in G} Az_\sigma$ by the rule
\begin{equation}
\label{eq:z_sigma-eq}
az_\sigma\cdot bz_\tau = a\omega_\sigma(b)f(\sigma,\tau)z_{\sigma\tau} \qt{for all $a,b\in A$ and $ \sigma,\tau\in G$.}
\end{equation}
This turns $B$ into a central simple $F$-algebra
with $1_B=f(\id,\id)^{-1}z_{\id}$
and $\deg B=[K:F]\deg A$.
The subalgebra $Af(\id,\id)^{-1}z_{\id}$ of $B$ is isomorphic to $A$ (as an $F$-algebra).
If $A$ is identified with $Af(\id,\id)^{-1}z_{\id}$ and $z_\sigma$ is identified with $1z_\sigma$, then 
\begin{equation}
\label{eq:z-sigma-relations}
\begin{gathered}
  \omega_\sigma = \iota_{z_\sigma}|_A, \\
  f(\sigma,\tau)= z_\sigma z_\tau z_{\sigma\tau}^{-1},
\end{gathered}
\end{equation}
for all $\sigma,\tau\in G$.
\end{prop}
\begin{proof}
This is shown with the same arguments as in the case when $A$ is a field
(see \cite[Proposition 14.1]{pierce:ass-alg}).
The relations (\ref{eq:factor-set}) guarantee the associativity of the multiplication,
and Lemma \ref{lem:z-sigma-lin-indep} is used to show that $B$ is simple.
\end{proof}

\begin{defn}
\label{def:gen-cr-prod}
The algebra $B\in\calA(F)$ defined in Proposition \ref{prop:gen-cr-prod-def}
is called \emph{generalized crossed product}
\index{crossed product!generalized}
\index{generalized crossed product}
and will be denoted by $(A,G,(\omega,f))$.
We routinely identify $A$ with the subalgebra $Af(\id,\id)^{-1}z_{\id}$ of $(A,G,(\omega,f))$.
If $G$ is cyclic (\resp abelian) then we call $(A,G,(\omega,f))$ a \emph{cyclic} (\resp \emph{abelian}) generalized crossed product.
\index{generalized crossed product!cyclic}
\index{crossed product! cyclic generalized}
\end{defn}

\begin{remark}
\label{rem:gen-cr-prod}
Let $(\omega,f)\in\calF(G,A^\times)$ and let $c\in Z^2(G,K^\times)$ be a $2$-cocycle.
\begin{enumerate}
\item $A=C_{(A,G,(\omega,f))}(K)$.
\item $A\sim (A,G,(\omega,f))_K$.
\item $(K,G,(\id_G,c))=(K/F,G,c)$, also written $(K,G,c)$.
\end{enumerate}
\end{remark}
\begin{proof}
Since $K=Z(A)$, $A\subseteq C_{(A,G,(\omega,f))}(K)$.
Equality follows from a consideration of degrees by the
Double Centralizer Theorem (\cf \cite[Theorem 12.7]{pierce:ass-alg}).
(2) follows from (1) by a standard argument (\cf \cite[Lemma 13.3]{pierce:ass-alg}).
(3) is just a comparison of notation with the crossed products of fields.
\end{proof}

\begin{lemma}
\label{lem:gen-cr-prod-basis}
Suppose that $B\in\calA(F)$ contains $K$ as a subfield
and let $A=C_B(K)$.
There is a family $\set{z_\sigma}_{\sigma\in G}$ in $\ug{B}$
that satisfies (\ref{eq:z-sigma-set})
and any such family forms a free set of generators of $B$ over $A$.
Moreover a factor set $(\omega,f)\in\calF(G,\ug A)$ is defined by (\ref{eq:z-sigma-relations})
such that $B\cong(A,G,(\omega,f))$.
If $z_{\id}=1$ then $(\omega,f)$ is normalized.
\end{lemma}
\begin{proof}
By the theorem of Skolem-Noether (\cf \cite[Theorem 12.6]{pierce:ass-alg})
there is, for each $\sigma\in G$,
an element $z_\sigma\in\ug B$ such that $\iota_{z_\sigma}|_K=\sigma$.
Since $A=C_B(K)$,
any so-defined family $\set{z_\sigma}_{\sigma\in G}$ is free over $A$ by Lemma \ref{lem:z-sigma-lin-indep}.
It follows from the
Double Centralizer Theorem (\cf \cite[Theorem 12.7]{pierce:ass-alg})
that $\set{z_\sigma}_{\sigma\in G}$ generates $B$ over $A$.

Now define $(\omega,f)$ by (\ref{eq:z-sigma-relations}).
Since $A=C_B(K)$, we have $\iota_{z_\sigma}(A)=A$,
thus $\omega_\sigma\in\Aut_F(A)$ for all $\sigma\in G$.
By (\ref{eq:z-sigma-set}),
$\iota_{z_\sigma z_\tau z_{\sigma\tau}^{-1}}|_K=\id_K$,
thus $f(\sigma,\tau)\in \ug A$ for all $\sigma,\tau\in G$.
The relations (\ref{eq:factor-set-1}) and (\ref{eq:factor-set-2}) are obvious from the choice of the $z_\sigma$
and (\ref{eq:factor-set-3}) follows by a routine calculation.
Thus $(\omega,f)\in\calF(G,\ug A)$.
Since $\set{z_\sigma}_{\sigma\in G}$ is a free set of generators of $B$ over~$A$
satisfying (\ref{eq:z_sigma-eq}), 
we have $B\cong(A,G,(\omega,f))$.
\end{proof}

\begin{prop}
\label{prop:cohomologous-isomorphic}
Two factor sets $(\omega,f),(\eta,g)\in\calF(G,A^\times)$ are cohomologous if and only if
$$ (A,G,(\omega,f))\cong(A,G,(\eta,g)). $$
\end{prop}
\begin{proof}
This proposition is proved just like in the case when $A$ is a field
(\cf \cite[Lemma 14.2]{pierce:ass-alg}).
\end{proof}

\begin{prop}
  \label{prop:res-classify}
Let $A\in\calA(K)$ and $B\in\calA(F)$.
The following are equivalent~:
  \begin{enumerate}
  \item $A\sim B_K$.
  \item $A$ embeds (as an $F$-algebra) into some $B'\in\calA(F)$ with
    $B'\sim B$ such that $A=C_{B'}(K)$.
 \item There exists $(\omega,f)\in\calF(G,A^\times)$ such that $(A,G,(\omega,f))\sim B$.
  \end{enumerate}
\end{prop}
\begin{proof}
$(1)\impl (2)$:
Let $A\cong M_n(D)$ with $D\in\calD(K)$.
If $D$ embeds into some $B'\in\calA(F)$ with
$B'\sim B$ such that $D=C_{B'}(K)$,
then $A$ embeds into $M_n(B')$ and $A=C_{M_n(B')}(K)$.
Therefore we assume \wolg that $A$ is a division algebra.

Let $L$ be a maximal subfield of $A$.
Since $A\sim B_K$ and $K\subseteq L$, $L$ splits $B$.
Therefore, $L$ is a strictly maximal subfield of $B'$ for some $B'\in\calA(F)$ with $B'\sim B$
(\cf \cite[Theorem 13.3]{pierce:ass-alg}).
Then $L$ is a strictly maximal subfield of $C_{B'}(K)$.
Since $C_{B'}(K)\sim B'_K\sim B_K\sim A$
and $L$ is also a strictly maximal subfield of $A$,
it follows $A\cong C_{B'}(K)$.
\\
$(2)\impl (3)$: 
By Lemma \ref{lem:gen-cr-prod-basis},
there exists $(\omega,f)\in\calF(G,A^\times)$ such that 
$(A,G,(\omega,f))\cong B'\sim B$. \\
$(3)\impl (1)$: 
By Remark \ref{rem:gen-cr-prod} (2),
  $A\sim (A,G,(\omega,f))_K\sim B_K$.
\end{proof}

\subsection{The Product Theorem}

Like before
let $K/F$ be a finite Galois extension with $\Gal(K/F)=G$.
Let $A,B\in\calA(K)$, $(\omega,f)\in\calF(G,A^\times)$ and $(\eta,g)\in\calF(G,B^\times)$.
A factor set 
$(\omega\otimes\eta,f\otimes g)\in\calF(G,(A\otimes_K B)^\times)$ is then defined by 
\begin{gather*}
(\omega\otimes\eta)_\sigma:=\omega_\sigma\otimes\eta_\sigma, \\
(f\otimes g)(\sigma,\tau):=f(\sigma,\tau)\otimes g(\sigma,\tau),
\end{gather*}
for all $\sigma,\tau\in G$.
The following theorem is called the Product Theorem.
A proof can be found in \cite[Theorem 1.6]{tignol:gen-cr-prod}.

\begin{theorem}
  \label{thm:prod-formula}
Let $A,B\in\calA(K)$, $(\omega,f)\in\calF(G,A^\times)$ and $(\eta,g)\in\calF(G,B^\times)$.
Then
  $$ (A,G,(\omega,f)) \otimes_F (B,G,(\eta,g)) \sim (A\otimes_K B,G,(\omega\otimes\eta,f\otimes g)). $$
\end{theorem}

Let $F'/F$ be any field extension such that $F'$ is linearly disjoint to $K$ over $F$,
\ie $K':=K\otimes_F F'$ is a field.
Then $K'/F'$ is Galois, $\Gal(K'/F')\cong G$
and $A':=A\otimes_F F'=A\otimes_K K'\in\calA(K')$.
A factor set $(\omega',f')\in\calF(G,\ug{A'})$ is defined by
$\omega'_\sigma:=\omega_\sigma\otimes\id$ and
$f'(\sigma,\tau):=f(\sigma,\tau)\otimes 1$
for all $\sigma,\tau\in G$.
We denote $(\omega',f')$ by $(\omega,f)^{F'}$.

\begin{prop}
  \label{prop:gen-scal-ext}
Let $A\in\calA(K)$ and $(\omega,f)\in\calF(G,A^\times)$.
Let $F'/F$ be any field extension such that $F'$ is linearly disjoint to $K$ over~$F$.
Then 
\begin{equation*}
(A,G,(\omega,f))\otimes_F F'\cong (A\otimes_F F',G,(\omega,f)^{F'}).
\end{equation*}
\end{prop}
\begin{proof}
Let $\set{z_\sigma}_{\sigma\in G}$ be a family in $(A,G,(\omega,f))$ satisfying (\ref{eq:z-sigma-relations}).
Consider the family $\set{z_\sigma\otimes 1}_{\sigma\in G}$ in $(A,G,(\omega,f))\otimes_F F'$.
The assertion then follows from Lemma \ref{lem:gen-cr-prod-basis}.
\end{proof}

\subsection{Cyclic factor sets}

Suppose that  $K/F$ is a finite cyclic extension with $[K:F]=n$ and $\Gal(K/F)=G=\gen{\sigma}$.
As before let $A\in\calA(K)$.

\begin{lemma}
  \label{lem:ex-alpha}
Suppose that $\sigma$ extends to an $F$-automorphism $\wt\sigma$ of $A$.
There exists an element $\alpha\in A^\times$ such that 
\begin{equation}
  \label{eq:rel-alpha}
\wt\sigma^n=\iota_\alpha \qqt{and} \wt\sigma(\alpha)=\alpha.
\end{equation}
The relations (\ref{eq:rel-alpha}) determine $\alpha$ up to multiplication with elements from~$F^\times$.
\end{lemma}
\begin{proof}
See \cite[Lemma 19.7]{pierce:ass-alg}.
\end{proof}

\begin{prop}
  \label{prop:cycl-factor-set}
Suppose that $\sigma$ extends to an $F$-automorphism $\wt\sigma$ of $A$
and let $\alpha\in A^\times$ be an element satisfying (\ref{eq:rel-alpha}).
A normalized factor set $(\omega,f)\in\calF(G,A^\times)$ is defined by
\begin{equation}
\label{eq:cycl-factor-set}
\begin{gathered}
  \omega_{\sigma^i} := \wt\sigma^i,\\
  f(\sigma^i,\sigma^j) :=
  \begin{cases}
    1 & \text{if $i+j<n$}, \\
    \alpha & \text{if $i+j\geq n$}.
  \end{cases} 
\end{gathered}
\end{equation}
In particular $\calF(G,A^\times)$ is not empty.
\end{prop}
\begin{proof}
The relations (\ref{eq:factor-set-1})--(\ref{eq:factor-set-3}) are readily checked.
\end{proof}

\begin{remark}
  \label{rem:write-cycl-factor-set}
We shortly write $(\wt\sigma,\alpha)$ for the cyclic factor set $(\omega,f)$ defined in (\ref{eq:cycl-factor-set}),
and $(A,\wt\sigma,\alpha)$ for the cyclic generalized crossed product $(A,G,(\wt\sigma,\alpha))$.
Then $(A,\wt\sigma,\alpha)=\bigoplus Az^i$ for an element $z\in(A,\wt\sigma,\alpha)^\times$ with $z^n=\alpha$ and $\iota_z|_A=\wt\sigma$.
\end{remark}

We now give the ``cyclic version'' of Lemma \ref{lem:gen-cr-prod-basis}.
The proof is only a special case of the proof of Lemma \ref{lem:gen-cr-prod-basis}.

\begin{lemma}
  \label{lem:cyclic-cr-prod-basis}
Suppose that $B\in\calA(F)$ contains $K$ as a subfield
and let $A=C_B(K)$.
There is $z\in\ug B$ such that
\begin{equation}
\label{eq:iota-z-i}
\iota_{z}|_K=\sigma.
\end{equation}
For any such $z\in\ug B$
the family $\set{z^{i}}_{0\leq i<n}$ forms a free set of generators of $B$ over $A$.
Fix the extension $\wt\sigma:=\iota_z|_A$ of $\sigma$ to $A$.
The element $\alpha:=z^n$ lies in $\ug A$ and satisfies the relations (\ref{eq:rel-alpha}),
and $B\cong(A,G,(\wt\sigma,\alpha))$.
\end{lemma}

\begin{remark}
  \label{rem:cycl-equiv}
Let $(\wt\sigma,\alpha)\in\calF(G,\ug A)$
and let $\gamma\in\ug A$.
Then $$(\wt\sigma,\alpha)\sim(\iota_\gamma\wt\sigma,N(\gamma)\alpha)$$
where $N(\gamma)=\gamma\wt\sigma(\gamma)\cdots\wt\sigma^{n-1}(\gamma)$.
\end{remark}
\begin{proof}
This follows from Lemma \ref{lem:cyclic-cr-prod-basis} if 
in $(A,\wt\sigma,\alpha)$ the element $z$ is replaced by $\gamma z$.
\end{proof}

\begin{remark}
\label{rem:cyclic-factor-set-form}
Lemma \ref{lem:cyclic-cr-prod-basis} shows that
any cyclic factor set $(\omega,f)\in\calF(G,A^\times)$ is cohomologous 
to some factor set $(\wt\sigma,\alpha)$.
We shall give without proof a formula to compute $\alpha$ from $(\omega,f)$.
If we assume (by Proposition \ref{prop:cohomologous-factor-sets} (2)) that
$\omega$ is of the form $\omega_{\sigma^i}=\wt\sigma^i$ for $0\leq i<n$, then 
$$  \alpha:=f(\sigma^0,\sigma)\cdots f(\sigma^{n-1},\sigma) $$
satisfies (\ref{eq:rel-alpha}) and $(\omega,f)\sim(\wt\sigma,\alpha)$.
\end{remark}

\begin{cor}
  \label{cor:ext-auto-charac}
For any $A\in\calA(K)$ the following are equivalent~:
\begin{enumerate}
  \item $\sigma$ extends to an $F$-automorphism $\wt\sigma$ of $A$.
  \item $\calF(G,A^\times)$ is not empty.
  \item $A\sim B_K$ for some $B\in\calA(F)$.
  \item $A$ embeds (as an $F$-algebra) into some $B'\in\calA(F)$ with
    $B'\sim B$ such that $A=C_{B'}(K)$.
\end{enumerate}
  In particular, these properties depend only on the class $[A]$ of $A$ in $\Br(K)$.
\end{cor}
\begin{proof}
  $(1)\Rightarrow(2)$ is Proposition \ref{prop:cycl-factor-set}
  and $(2)\Rightarrow(1)$ is trivial. 
  The equivalence of $(2), (3)$ and $(4)$ is due to Proposition \ref{prop:res-classify}. 
  Obviously $(3)$ depends only on the class $[A]$ of $A$.
\end{proof}

\begin{remark}
\label{rem:ext-auto}
The equivalence of (1) and (3) in Corollary~\ref{cor:ext-auto-charac}
is known from Eilenberg-McLane \cite[Theorem 5.4 and Corollary 7.3]{eil-mclane:normality-alg},
where the implication $(1)\impl(3)$ is derived from the fact that the third cohomology group $H^3(G,\ug K)$
is trivial for cyclic $G$.

Over number fields $K$
there is a nice and useful criterion for the conditions of Corollary \ref{cor:ext-auto-charac} in terms of the local invariants of $A$,
which we shall call \emph{Deurings's criterion}.
The theorem of Deuring \cite[Satz 4]{deuring:einb-v-alg} states that condition $(4)$ in Corollary \ref{cor:ext-auto-charac} holds if and only if
the local invariants of $A$ are fixed under conjugation by $\sigma$, \ie
\begin{equation*}
  \label{eq:deuring}
\inv_v A=\inv_{v\circ\sigma} A \qt{for all $v\in\Val(K)$.}
\end{equation*}
\end{remark}

\begin{ex}\label{ex:sigma-extends}
Let $\wt D$ be as in Example~\ref{ex:noncr-prod-cycl}.
It was computed that
$\inv_{w_1} \wt D=\frac{1}{3}$ and $\inv_{w_2} \wt D=\frac{2}{3}$
for two valuations $w_1,w_2\in\V0(K)$
that are unique extensions of valuations on $F$, and $\inv_w \wt D=0$ for all other $w\in\Val(K)$.
Therefore, Deuring's criterion is satisfied,
\ie in this example $\sigma$ extends to an automorphism of $\wt D$.
\end{ex}

\subsection{Abelian factor sets}

Suppose that $K/F$ is finite abelian with
$G=\Gal(K/F)=S_1\times\cdots\times S_r$,
where the $S_i$ are cyclic with $S_i=\gen{\sigma_i}$ and $|S_i|=n_i$.
As above let $A\in\calA(K)$.
Suppose that $\sigma_1,\ldots,\sigma_r$ extend to $F$-automorphisms
$\wt\sigma_1,\ldots,\wt\sigma_r$ of $A$ respectively.
For convenience we introduce the notation
\begin{align*}
  N_i(x)&:=x\wt\sigma_i(x)\cdots\wt\sigma_i^{n_i-1}(x), \\
  {}_iN(x)&:=\wt\sigma_i^{n_i-1}(x)\cdots\wt\sigma_i(x)x.
\end{align*}
Obviously for all $x\in A$,
\begin{equation}
  \label{eq:N_i-inv}
  {}_iN(x^{-1})=N_i(x)^{-1}.
\end{equation}

Unlike in the cyclic case, $\calF(G,\ug A)$ can be empty
even though we have the extensions $\wt\sigma_i$ of $\sigma_i$ to $A$.
But if an abelian factor set $(\omega,f)\in\calF(G,\ug A)$ exists 
it can also be described with fewer parameters,
like a cyclic factor set can be described by a single parameter $\alpha$
(if $\wt\sigma$ is fixed).
This is a straightforward generalization from the discussion of abelian $2$-cocycles in \cite{amitsur-saltman:gen-abel-cr-prod}.
The following lemma is the ``abelian version'' of Lemma~\ref{lem:gen-cr-prod-basis}.

\begin{lemma}
  \label{lem:abelian-cr-prod-basis}
Suppose that $B\in\calA(F)$ contains $K$ as a subfield
and let $A=C_B(K)$.
There are $z_1,\ldots,z_r\in\ug B$ such that
\begin{equation}
\label{eq:iota-z-i}
\iota_{z_i}|_K=\sigma_i \qt{for all $1\leq i\leq r$.}
\end{equation}
If $z_1,\ldots,z_r\in\ug{B}$ satisfy (\ref{eq:iota-z-i}),
then the family $\set{z_\sigma}_{\sigma\in G}$,
defined by $z_{\sigma_1^{i_1}\cdots\sigma_r^{i_r}}:=z_1^{i_1}\cdots z_r^{i_r}$
for $0\leq i_k<n_k, 1\leq k\leq r$,
forms a free set of generators of $B$ over $A$.
Fix the extensions $\wt\sigma_i:=\iota_{z_i}|_A$ of $\sigma_i$ to $A$.
The elements 
\begin{equation}
  \label{eq:def-u-alpha}
  \begin{gathered}
    u_{ij}:=z_iz_jz_i^{-1}z_j^{-1}, \\
    \alpha_i:=z_i^{n_i} \\
  \end{gathered}
\end{equation}
for $1\leq i,j\leq r$ lie in $\ug A$ and satisfy the relations
\begin{subequations}
\label{eq:uij-rel}
\begin{gather}
\label{eq:uij-rel-1}
u_{ii}=1, \qquad u_{ij}=u_{ji}^{-1}, \\
\label{eq:uij-rel-2}
\wt\sigma_i^{n_i}=\iota_{\alpha_i}, \qquad
\wt\sigma_i\wt\sigma_j=\iota_{u_{ij}}\wt\sigma_j\wt\sigma_i, \\
\label{eq:uij-rel-3}
\wt\sigma_j(\alpha_i)=N_i(u_{ji})\alpha_i, \\
\label{eq:uij-rel-4}
\wt\sigma_k(u_{ij})u_{kj}\wt\sigma_j(u_{ki})u_{ji}\wt\sigma_i(u_{jk})u_{ik}=1
\end{gather}
\end{subequations}
for all $1\leq i,j,k\leq r$.
A normalized factor set $(\omega,f)\in\calF(G,\ug A)$ is defined by
\begin{equation}
\label{eq:z-simga-factor-set}
\begin{gathered}
  \omega_{\sigma_1^{i_1}\cdots\sigma_r^{i_r}} := \wt\sigma_1^{i_1}\cdots\wt\sigma_r^{i_r}, \\
 f(\sigma,\tau):= z_\sigma z_\tau z_{\sigma\tau}^{-1}, 
\end{gathered}
\end{equation}
such that $B\cong(A,G,(\omega,f))$.
It holds $f(\sigma_i,\sigma_j)=1$ for all $1\leq i\leq j\leq r$.
\end{lemma}
\begin{proof}
The relations (\ref{eq:uij-rel}) are readily checked,
where (\ref{eq:uij-rel-3}) is shown by induction on $n_i$.
The rest of the lemma is a special case of Lemma~\ref{lem:gen-cr-prod-basis}.  
\end{proof}

\begin{remark}
  \label{rem:write-abelian-factor-set}
The map $f$, as defined in (\ref{eq:z-simga-factor-set}),
can be expressed in terms of the elements $u_{ij}$ and $\alpha_i$ only
(not using the $z_\sigma$).
The resulting formula is omitted here, since it is complicated and will not be further used.
But this shows that the factor set $(\omega,f)$ defined in (\ref{eq:z-simga-factor-set})
is completely determined by the $\wt\sigma_i$, $u_{ij}$ and  $\alpha_i$.
It will therefore be denoted shortly by $(\wt\sigma,u,\alpha)$.
Further, we write $(A,\wt\sigma,u,\alpha)$ for the abelian generalized crossed product $(A,G,(\wt\sigma,u,\alpha))$. 
\end{remark}

\begin{remark}
\label{rem:abelian-factor-set-form}
Lemma~\ref{lem:abelian-cr-prod-basis} shows that
any abelian factor set $(\omega,f)$ is cohomologous to some factor set
$(\wt\sigma,u,\alpha)$.
\end{remark}

The following proposition shows that the relations (\ref{eq:uij-rel}) are also sufficient to define a factor set.
The commutative version is known as
\cite[Theorem 1.3]{amitsur-saltman:gen-abel-cr-prod}.

\begin{prop}
\label{prop:ex-abelian-factor-set}
Suppose that $\sigma_1,\ldots,\sigma_r$ extend to $F$-algebra automorphisms
$\wt\sigma_1,\ldots,\wt\sigma_r$ of $A$ respectively
and there are elements $u_{ij}, \alpha_i \in\ug A$, $1\leq i,j\leq r$,
satisfying (\ref{eq:uij-rel}) for all $1\leq i,j,k\leq r$.
Then there is a factor set $(\wt\sigma,u,\alpha)\in\calF(G,A^\times)$.
\end{prop}
\begin{proof}
We prove by induction on $r$
that there is a $B\in\calA(F)$ that contains $A$ as a subalgebra, $A=C_B(K)$,
and there are elements $z_1,\ldots,z_r\in\ug B$ satisfying
\begin{equation}
  \label{eq:iota-z-i-on-A}
  \iota_{z_i}|_A=\wt\sigma_i 
\end{equation}
and (\ref{eq:def-u-alpha}) for all $1\leq i,j\leq r$.
The proposition then follows from Lem\-ma~\ref{lem:abelian-cr-prod-basis}.
As opposed to the proof of \cite[Theorem 1.3]{amitsur-saltman:gen-abel-cr-prod}
we use cyclic generelized crossed products here instead of twisted polynomial rings.

If $r=0$ we trivially choose $B=A$.
Now assume $r>0$ and let
$K'=\Fix(\sigma_1,\ldots,\sigma_{r-1})$.
By induction hypothesis there is an $A'\in\calA(K')$ that contains $A$ as a subalgebra, $A=C_{A'}(K)$,
and there are elements $z_1,\ldots,z_{r-1}\in\ug{A'}$ satisfying
~(\ref{eq:iota-z-i-on-A}) and (\ref{eq:def-u-alpha}) for all $1\leq i,j< r$.
The extension $K'/F$ is cyclic and $\Gal(K'/F)=\gen{\sigma_r|_{K'}}$.
We can extend $\sigma_r$ to an $F$-automorphism $\sigma_r^*$ of $A'$ as follows.
Since $\set{z_1^{i_1}\cdots z_{r-1}^{i_{r-1}}}_{0\leq i_j<n_j}$ is a free set of generators of $A'$ over $A$ (by Lemma~\ref{lem:abelian-cr-prod-basis}),
$\sigma_r^*$ is completely determined if we set
\begin{equation}
  \label{eq:sigma_r-def}
  \sigma_r^*(a):=\wt\sigma_r(a), \quad
  \sigma_r^*(z_i):=u_{ri}z_i \qt{for all $a\in A$ and $1\leq i<r$.}
\end{equation}
To see that (\ref{eq:sigma_r-def}) defines an automorphism we have to check that $\sigma^*_r$ preserves the relations (\ref{eq:def-u-alpha}).
By definition of $\sigma^*_r$ and (\ref{eq:iota-z-i-on-A}),
\begin{gather*}
\sigma^*_r(z_i)\sigma^*_r(z_j)=u_{ri}z_iu_{rj}z_j=u_{ri}\wt\sigma_i(u_{rj})z_iz_j, \\
\sigma^*_r(u_{ij})\sigma^*_r(z_j)\sigma^*_r(z_i)=\wt\sigma_r(u_{ij})u_{rj}z_ju_{ri}z_i
=\wt\sigma_r(u_{ij})u_{rj}\wt\sigma_j(u_{ri})u_{ji}z_iz_j \\
\intertext{and}
\wt\sigma_r(z_i)^{n_i}=(u_{ri}z_i)^{n_i}=N_i(u_{ri})z_i^{n_i}=N_i(u_{ri})\alpha_i
\end{gather*}
for all $1\leq i,j<r$.
Thus $\sigma^*_r(z_i)\sigma^*_r(z_j)=\sigma^*_r(u_{ij})\sigma^*_r(z_j)\sigma^*_r(z_i)$ by (\ref{eq:uij-rel-4}) and (\ref{eq:uij-rel-1}),
and $\sigma^*_r(\alpha_i)=\sigma^*_r(z_i)^{n_i}$ by (\ref{eq:uij-rel-3}).

We next show that ${\sigma_r^*}^{n_r}=\iota_{\alpha_r}$.
By (\ref{eq:uij-rel-2}), for all $a\in A$, 
$$ {\sigma^*_r}^{n_r}(a)={\wt\sigma_r}^{n_r}(a)=\iota_{\alpha_r}(a). $$
By (\ref{eq:sigma_r-def}), (\ref{eq:N_i-inv}), (\ref{eq:uij-rel-1}) and (\ref{eq:iota-z-i-on-A}),
for all $1\leq i< r$, 
\begin{gather*}
{\sigma_r^*}^{n_r}(z_i)={\sigma_r^*}^{n_r-1}(u_{ri})\cdots\sigma_r^*(u_{ri})u_{ri}z_i={}_rN(u_{ri})z_i=N_r(u_{ir})^{-1}z_i 
\intertext{and}
\iota_{\alpha_r}(z_i)=\alpha_rz_i\alpha_r^{-1}=\alpha_r\wt\sigma_i(\alpha_r)^{-1}z_i.
\end{gather*}
Thus ${\sigma_r^*}^{n_r}(z_i)=\iota_{\alpha_r}(z_i)$ by (\ref{eq:uij-rel-3})
and ${\sigma_r^*}^{n_r}=\iota_{\alpha_r}$ is proved.
Moreover, $\sigma_r^*(\alpha_r)=\alpha_r$ 
by (\ref{eq:uij-rel-3}) and (\ref{eq:uij-rel-1}).
Therefore we can form the cyclic generalized crossed product
$B:=(A',\sigma_r^*,\alpha_r)$.
Then $B$ contains an element $z_r$ with $z_r^{n_r}=\alpha_r$ and
$\iota_{z_r}|_{A'}=\sigma_r^*$.
The relations (\ref{eq:iota-z-i-on-A}) and (\ref{eq:def-u-alpha}) are then satisfied for all $1\leq i,j,k\leq r$ by definition of $\sigma_r^*$.
\end{proof}

\begin{remark}
  \label{rem:abelian-r-1}
In the case $r=1$ the elements $u_{ij}$ and $\alpha_i$ are determined by the single element $\alpha=\alpha_1$.
The relations (\ref{eq:uij-rel}) are equivalent to (\ref{eq:rel-alpha}).
This shows that the abelian factor sets generalize the cyclic factor sets.
\end{remark}

\begin{remark}
  \label{rem:abelian-r-2}
In the case $r=2$ the elements $u_{ij}$ are determined by the single element $u=u_{21}$,
thus the abelian factor set $(\wt\sigma,u,\alpha)$ is determined by 
the extensions $\wt\sigma_1,\wt\sigma_2$ 
and the three parameters $u_{21}$, $\alpha_1$ and $\alpha_2$.
We shall therefore denote it also by
$(\wt\sigma_1,\wt\sigma_2;\alpha_1,\alpha_2;u)$ where $u=u_{21}$.
It is easily checked that the relation (\ref{eq:uij-rel-4}) is always satisfied.
Therefore (\ref{eq:uij-rel}) are equivalent to 
\begin{subequations}  
\label{eq:abel-r=2}
\begin{gather}
\label{eq:abel-r=2-1}
\wt\sigma_1^{n_1}=\iota_{\alpha_1}, \quad \wt\sigma_1(\alpha_1)=\alpha_1, \\
\label{eq:abel-r=2-2}
\wt\sigma_2^{n_2}=\iota_{\alpha_2}, \quad \wt\sigma_2(\alpha_2)=\alpha_2, \\
\label{eq:abel-r=2-3}
\wt\sigma_2\wt\sigma_1=\iota_u\wt\sigma_1\wt\sigma_2, \\
\label{eq:abel-r=2-4}
\wt\sigma_1(\alpha_2)=N_2(u^{-1})\alpha_2, \\
\label{eq:abel-r=2-5}
\wt\sigma_2(\alpha_1)=N_1(u)\alpha_1. 
\end{gather}
\end{subequations}
\end{remark}

The following is a repetition of Proposition \ref{prop:gen-scal-ext} for abelian factor sets.
\begin{prop}
  \label{prop:gen-scal-ext-abelian}
Let $(\wt\sigma,u,\alpha)\in\calF(G,A^\times)$ be an abelian factor set.
Let $F'/F$ be any field extension such that $F'$ is linearly disjoint to $K$ over~$F$.
Then $K\otimes_F F'$ is Galois over $F'$ with $\Gal(K\otimes_F F'/F')\cong G$ and
\begin{equation*}
(A,\wt\sigma,u,\alpha)\otimes_F F'\cong (A\otimes_F F',\wt\sigma\otimes\id,u\otimes 1,\alpha\otimes 1).
\end{equation*}
\end{prop}
\begin{proof}
Let $z_1,\ldots,z_r$ be elements from $(A,\wt\sigma,u,\alpha)$ satisfying (\ref{eq:def-u-alpha}).
Consider the elements $z_1\otimes 1,\ldots,z_r\otimes 1$ from $(A,\wt\sigma,u,\alpha)\otimes_F F'$.
The assertion then follows from Lemma~\ref{lem:abelian-cr-prod-basis}.
\end{proof}

\subsection{Generalized crossed products of inertial algebras}

The goal of this subsection is to give a proof of Theorem \ref{theorem:lift} that is based on generalized crossed products.
Let $F$ be a field with valuation $v$.
The first lemma shows, roughly spoken,
that if $F$ has the inertial lift property then generalized crossed products can be lifted
from the residue level.

\begin{lemma}
\label{lem:inert-lift-factor-set}
Let $\wt K/\bar F$ be a Galois extension with $\Gal(\wt K/\bar F)=G$,
let $\wt D\in\calD(\wt K)$, and let $(\wt\omega,\wt f)\in\calF(G,\ug{\wt D})$.
If $F$ has the inertial lift property,
there is an inertial lift $K$ of $\wt K$ over $F$,
an inertial lift $E$ of $\wt D$ over $F$ with centre $K$,
and a factor set $(\omega,f)\in\calF(G,\ug E)$ such that
$[(E,G,(\omega,f))]\in\IBr(F)$ and
$\beta_F([(E,G,(\omega,f))])=[(\wt D,G,(\wt\omega,\wt f))]$.
Moreover, $\exp (E,G,(\omega,f))=\exp (\wt D,G,(\wt\omega,\wt f))$.
\end{lemma}
\begin{proof}
Let $A\in\calD(\bar F)$ be the underlying division algebra of $(\wt D,G,(\wt\omega,\wt f))$.
Let $K$ be an inertial Galois lift of $\wt K$ over $F$ with the property (\ref{eq:ILP}),
and let $I\in\calD(F)$ be an inertial lift of $A$ over $F$ with $[I]\in X$,
where $X\subseteq\IBr(F)$ is the subgroup of $\Br(F)$ from the inertial lift property.
Then $\exp I=\exp A$,
$[I_K]\in\IBr(K)$ and $\bar{I_K}\cong\bar I_{\bar K}\cong A_{\wt K}\cong\wt D$.
This shows that $E:=I_K$ is an inertial lift of $\wt D$ with centre $K$.
By Proposition \ref{prop:res-classify} there is $(\omega,f)\in\calF(G,\ug E)$
with $(E,G,(\omega,f))\sim I$.
Therefore
$[(E,G,(\omega,f))]\in\IBr(F)$ and
$\beta_F([(E,G,(\omega,f))])=[\bar I]=[A]=[(\wt D,G,(\wt\omega,\wt f))]$.
Moreover
$\exp (E,G,(\omega,f))=\exp I=\exp A=\exp (\wt D,G,(\wt\omega,\wt f))$.
\end{proof}

We call $(E,G,(\omega,f))$ in Lemma \ref{lem:inert-lift-factor-set}
an \emph{inertial lift} of the generalized crossed product $(\wt D,G,(\wt\omega,\wt f))$.
Now let $K/F$ be a finite inertial Galois extension with $\Gal(K/F)=G$.

\begin{lemma}
  \label{lem:inert-gen-cr-prod}
Let $A\in\calA(K)$ with $[A]\in\IBr(K)$.
Let $(\omega,f)\in\calF(G,\ug A)$
such that the underlying division algebra $I\in\calD(F)$ of $(A,G,(\omega,f))$ is inertial,
and let $c\in Z^2(G,\ug K)$ such that $N:=(K,G,c)\in\calD(F)$ is nicely semiramified.
Then for $B:=(A,G,(\omega,fc))$~:
\begin{enumerate}
\item $B\sim I\otimes_F N$.
\item $\ind B=[K:F]\ind A$.
\item $[B]\in\SBr(F)$ and $\bar{B_F}\cong \bar{A_K}$.~
\footnote{Recall that $B_F,A_K$ stand for the underlying division algebras of $B,A$ respectively.}
\item If $\exp I=\exp\bar I$ and $\exp N=\exp\Gal(\bar K/\bar F)$ then 
  $$\exp B=\exp B^h=\lcm(\exp N,\exp I) . $$
\end{enumerate}
In particular, if $A$ is a division algebra, so is $B$.
\end{lemma}
\begin{proof}
Theorem \ref{thm:prod-formula} yields
$$B=(A\otimes_{K}K,G,(\omega,fc))\sim(A,G,(\omega,f))\otimes_F(K,G,c)\sim I\otimes_F N .$$
By Proposition \ref{prop:inert-scalar-ext-valued},
since $[I_K]=[A]\in\IBr(K)$,
$\ind\bar I_{\bar K}=\ind I_K=\ind A$ and 
$\bar{A_K}=\bar{I_K}\cong \bar I_{\bar K}$.
Therefore, Proposition \ref{prop:N-I-valued} $(1)$ shows $[B]\in\SBr(F)$,
and Theorem \ref{theorem:inert-split-decomposition} shows
$\ind B=[\bar K:\bar F]\ind \bar I_{\bar K}=[K:F]\ind A$ and
$\bar{B_F}\cong\bar I_{\bar K}\cong\bar{A_K}$.
$(4)$ follows from Proposition \ref{prop:N-I-valued} $(2)$.
\end{proof}

We can now construct the underlying division algebra $D$ in Theorem~\ref{theorem:lift} directly as a generalized crossed product.

\begin{theorem}
  \label{thm:underlying-div-alg}
  Let $F$ be a valued field that has the inertial lift property,
  and let $\wt D$ be a finite-dimensional division algebra over $\bar F$
  such that the following properties are satisfied~:
  \begin{enumerate}
  \item $Z(\wt D)$ is abelian over $\bar F$ with $[Z(\wt D):\bar F]=n<\ift$.
  \item There exists $(\wt\omega,\wt f)\in\calF(G,{\wt D}^\times)$.~
    \footnote{Proposition \ref{prop:res-classify} shows that this condition is actually equivalent to condition $(2)$ of Theorem \ref{theorem:lift}.}
  \item $G=\Gal(Z(\wt D)/\bar F)$ embeds into $\Gamma_F/m\Gamma_F$,
where $m=\exp G$.
  \end{enumerate}
There is an inertial lift
$(E,G,(\omega,f))$ of %
$(\wt D,G,(\wt\omega,\wt f))$
and a $2$-cocycle $c\in Z^2(G,Z(E)^\times)$ such that
$D:=(E,G,(\omega,fc))\in\calD(F)$ is inertially split with $\bar D\cong\wt D$.
Furthermore,
$\ind D=n\ind\wt D$ and
$\exp D=\exp D^h=\lcm(m,\exp (\wt D,G,(\wt\omega,\wt f)))$.
\end{theorem}
\begin{proof}
Let $(E,G,(\omega,f))$ be an inertial lift of the generalized crossed product 
$(\wt D,G,(\wt\omega,\wt f))$
as in Lemma \ref{lem:inert-lift-factor-set}.
Let $K=Z(E)$.
By Example \ref{ex:NSR}, since $\Gal(K/F)\cong G$ embeds into $\Gamma_F/m\Gamma_F$,
there is a $c\in Z^2(G,\ug K)$ such that $N:=(K,G,c)\in\calD(F)$ 
is nicely semiramified with inertial maximal subfield $K$ and 
$\exp N=\exp G=m$.
Define $D:=(E,G,(\omega,fc)$.
The statements of the theorem then follow from Lemma \ref{lem:inert-gen-cr-prod}
applied with $I\in\calD(F)$ the underlying division algebra of $(E,G,(\omega,f))$.
\end{proof}

\begin{remark}
\label{rem:explicit-constr}
\label{rem:underl-div-alg}
(1) Theorem \ref{thm:underlying-div-alg} reduces the explicit construction of $D$
to the explicit construction of the inertial lift
$(E,G,(\omega,f))$ of $(\wt D,G,(\wt\omega,\wt f))$.
The cocycle $c$ is already given explicitly. 

(2) The computation of $\exp D$ in Theorem \ref{thm:underlying-div-alg} is reduced to the computation of
$\exp (\wt D,G,(\wt\omega,\wt f))$,
which is still a difficult task in general.
But by Corollary \ref{cor:exp},
if $G$ is cyclic, then 
$\exp D$ is independent of the choice of $(\wt\omega,\wt f)\in\calF(G,{\wt D}^\times)$, 
\ie $\exp D$ is determined by $\wt D$ only. 

(3) If $Z(\wt D)$ is a global field and $G$ is cyclic,
then $\exp D$ can be computed by the formula 
\[
  \exp D =\exp D^h = \lcm(n,n_w \ind\wt D_w)_{w\in\Val(Z(\wt D))}
\]
of Proposition \ref{prop:exp-form-num-field},
where $n_w=[Z(\wt D):\bar F]_w$.
\end{remark}

The following example shows how $(E,G,(\omega,f))$ in Lemma \ref{lem:inert-lift-factor-set} is obtained if
$V_F$ contains a field that maps isomorphically onto $\bar F$ under $\pi_F$.

\begin{ex}
\label{ex:explicit-lift}
Suppose that $V_F$ contains a field that maps isomorphically onto $\bar F$ under $\pi_F$.
The inertial lifts of $\wt K$ and $\wt D$ in the proof of Lemma \ref{lem:inert-lift-factor-set} are
$K=\wt K\otimes_{\bar F} F$ and $E=\wt D\otimes_{\bar F} F$ (\cf the inertial lift property in the proof of Theorem \ref{thm:bar-F-embed-ILP}).
Moreover,
$(\omega,f)=(\wt\omega,\wt f)^F\in\calF(G,\ug A)$, \ie 
$\omega_\sigma=\wt\omega_\sigma\otimes\id_F$ 
and $f(\sigma,\tau)=\wt f(\sigma,\tau)\otimes 1$
for all $\sigma,\tau\in G$.
For, by Proposition \ref{prop:gen-scal-ext},
since $K=\wt K\otimes_{\bar F} F$ is a field, 
$$I\sim(\wt D,G,(\wt\omega,\wt f))\otimes_{\bar F} F
\cong(\wt D\otimes_{\bar F} F,G,(\omega,f))
=(E,G,(\omega,f)),$$
where $I$ is the inertial lift of the underlying division algebra of $(\wt D,G,(\wt\omega,\wt f))$ as in the proof of Lemma \ref{lem:inert-lift-factor-set}.
\end{ex}

\section{Twisted function fields and Laurent series rings}
\subsection{One twisted indeterminant}

\label{sec:One-twist-indet}
In this section we describe the construction of twisted function fields
and twisted Laurent series rings.
Let $K/F$ be a finite cyclic field extension
with $\Gal(K/F)=\gen{\sigma}$ and $[K:F]=n$.
Let $A\in\calA(K)$
and suppose that $\sigma$ extends to an $F$-algebra automorphism $\wt\sigma$ of $A$.
Denote by $A[x;\wt\sigma]$ the set of all polynomials
\begin{equation*}
  A[x;\wt\sigma] := \sett{\sum_{i=0}^k d_i x^i}{k\in\N_0, d_i\in A} ,
\end{equation*}
and by $A(\!(x;\wt\sigma)\!)$ the set of all formal series 
\begin{equation*}
  A(\!(x;\wt\sigma)\!) := \sett{\sum_{i\geq k} d_i x^i}{k\in\Z, d_i\in A} .
\end{equation*}
A ring structure is given on $A[x;\wt\sigma]$ and $A(\!(x;\wt\sigma)\!)$ by componentwise addition and by multiplication with the rule
\begin{equation*}
 ax^i\cdot bx^j=a\wt\sigma^i(b)x^{i+j} \qt{for all $a,b\in A$, $i,j\in\Z$.} 
\end{equation*}
Obviously $A[x;\wt\sigma]\subset A(\!(x;\wt\sigma)\!)$ is a subring
and we identify $A$ with the subring $Ax^0$ of $A[x;\wt\sigma]$ and $A(\!(x;\wt\sigma)\!)$.
Denote by $A(x;\wt\sigma)$ the ring of central quotients of $A[x;\wt\sigma]$, \ie
\begin{equation*}
  A(x;\wt\sigma):=\sett{f/g}{f\in A[x;\wt\sigma], g\in Z(A[x;\wt\sigma])} .  
\end{equation*}
Let $\alpha\in A^\times$ be an element satisfying (\ref{eq:rel-alpha}).
Then 
\begin{equation*}
t:=\alpha^{-1}x^n
\end{equation*}
is a commutative indeterminate over $A$ and  
the centres of $A[x;\wt\sigma]$,  $A(\!(x;\wt\sigma)\!)$ and $A(x;\wt\sigma)$ are  
\begin{alignat*}{2}
  &Z(A[x;\wt\sigma])&&=F[t]=\sett{\sum_{i=0}^k a_i(\alpha^{-1}x^n)^i}{a_i\in F, k\in\N_0} , \\
  &Z(A(\!(x;\wt\sigma)\!))&&=F(\!(t)\!)=\sett{\sum_{i\geq k} a_i(\alpha^{-1}x^n)^i}{a_i\in F, k\in\Z} 
\intertext{and}
  &Z(A(x;\wt\sigma))&&=Q(Z(A[x;\wt\sigma]))=F(t),
\end{alignat*}
where $Q(R)$ denotes the quotient field of an integral domain $R$.
Note that $Z(A(x;\wt\sigma))$ and $Z(A(\!(x;\wt\sigma)\!))$ are fields.
Therefore
all elements from $Z(A[x;\wt\sigma])$ are already invertible in $A(\!(x;\wt\sigma)\!)$,
hence $A(x;\wt\sigma)$ can be regarded as a subring of $A(\!(x;\wt\sigma)\!)$.
We call
$A[x;\wt\sigma]$ \emph{twisted polynomial ring},
$A(x;\wt\sigma)$ \emph{twisted function field}
and $A(\!(x;\wt\sigma)\!)$ \emph{twisted Laurent series ring}.

To see that $A(x;\wt\sigma)$ and $A(\!(x;\wt\sigma)\!)$ are central simple $F(t)$- and $F(\!(t)\!)$-algebras respectively
and to compute their index and exponent
they are presented as (cyclic) generalized crossed products.
This is due to \cite[Theorem~2.3]{tignol:gen-cr-prod}.
We regard $A(t)$, the ring of central quotients of $A[t]$,
as a subring of $A(x;\wt\sigma)$.

\begin{lemma}
  \label{lem:twisted-alg-cycl-cr-prod}
$A(x;\wt\sigma)$ and $A(\!(x;\wt\sigma)\!)$ are central simple $F(t)$- and $F(\!(t)\!)$-algebras respectively and
\[
  A(x;\wt\sigma)\cong(A(t),\wt\sigma,\alpha t), \qquad
  A(\!(x;\wt\sigma)\!)\cong (A(\!(t)\!),\wt\sigma,\alpha t).
\]
Here $\wt\sigma$ also denotes the extension of $\wt\sigma$ to $A(t)$ and $A(\!(t)\!)$ that fixes $t$.
\end{lemma}
\begin{proof}
We prove this lemma for $A(x;\wt\sigma)$,
it is analogous for $A(\!(x;\wt\sigma)\!)$.
The set $\{x^0,\ldots,x^{n-1}\}$ clearly generates 
$A[x;\wt\sigma]$ over $A[t]$.
Since $Z(A[x;\wt\sigma])=F[t]\subseteq A[t]$, it also generates 
$A(x;\wt\sigma)$ over $A(t)$.
Further, we have $\iota_{x}|_{K(t)}=\sigma$,
where $\sigma$ denotes the extension of $\sigma$ to $K(t)$ fixing $t$,
and $A(t)\subseteq C_{A(x;\wt\sigma)}(K(t))$.
Lemma~\ref{lem:cyclic-cr-prod-basis} then shows that
$\{x^0,\ldots,x^{n-1}\}$ is free over $A(t)$,
thus $A(x;\wt\sigma)=\bigoplus_{i=0}^{n-1} A(t)x^i$.
It is now readily seen that actually
$C_{A(x;\wt\sigma)}(K(t))=A(t)$.
Therefore, by Lemma~\ref{lem:cyclic-cr-prod-basis},
$A(x;\wt\sigma)\cong(A(t),\wt\sigma, \alpha t)$,
since $x^n=\alpha t$.
\end{proof}

We have to be careful when regarding $A(x;\wt\sigma)$ and $A(\!(x;\wt\sigma)\!)$ as $F(t)$- and $F(\!(t)\!)$-algebras respectively,
since in this notation the choice of $\alpha$ is lost.
We should keep in mind that $t$ depends on $\alpha$,
and that different choices of $\alpha$ lead to different actions of $F(t)$ and $F(\!(t)\!)$ on $A(x;\wt\sigma)$ and $A(\!(x;\wt\sigma)\!)$ respectively.
However, in the Brauer group of $F(t)$ and $F(\!(t)\!)$,
this affects the elements 
$[A(x;\wt\sigma)]$ and $[A(\!(x;\wt\sigma)\!)]$ only up to an isomorphism
which is induced by an $F$-isomorphism of $F(t)$ and $F(\!(t)\!)$ respectively.
In Lemma~\ref{lem:twisted-alg-cycl-cr-prod} and in the following theorems 
we assume that $\alpha$ is fixed and $t=\alpha^{-1}x^n$.

\begin{theorem}
  \label{thm:twisted-alg-properties}
$A(x;\wt\sigma)$ and $A(\!(x;\wt\sigma)\!)$ are central simple $F(t)$- and $F(\!(t)\!)$-algebras respectively with
\begin{enumerate}
\item $\deg A(x;\wt\sigma)=\deg A(\!(x;\wt\sigma)\!)=n\deg A,$
\item $\ind A(x;\wt\sigma)=\ind A(\!(x;\wt\sigma)\!)=n\ind A,$
\item $\exp A(x;\wt\sigma)=\exp A(\!(x;\wt\sigma)\!)=\lcm(n,\exp (A,\wt\sigma,\alpha)),$
\item $[A(x;\wt\sigma)]\in\SBr(F(t))$ and $[A(\!(x;\wt\sigma)\!)]\in\SBr(F(\!(t)\!)),$
\item $\bar{A(x;\wt\sigma)_{F(t)}}\cong\bar{A(\!(x;\wt\sigma)\!)_{F(\!(t)\!)}}\cong A_K.$
\end{enumerate}
Here, $F(t)$ and $F(\!(t)\!)$ are regarded with respect to the $t$-adic valuation.
Moreover, for any $C\in\calA(F)$ with $C^K\sim A$,
\begin{enumerate}
\setcounter{enumi}{5}
\item $\exp A(x;\wt\sigma)=\exp A(\!(x;\wt\sigma)\!)=\lcm(n,\exp C).$
\end{enumerate}
In particular, if $A$ is a division algebra then
$A(x;\wt\sigma)$ and $A(\!(x;\wt\sigma)\!)$ are division algebras.
\end{theorem}
\begin{proof}
Again we prove the theorem for $A(x;\wt\sigma)$ only.
Lemma~\ref{lem:twisted-alg-cycl-cr-prod} states
$A(x;\wt\sigma)\cong(A(t),\wt\sigma,\alpha t)$
and (1) follows from Proposition~\ref{prop:gen-cr-prod-def}.
In the following $F(t)$ is regarded with respect to the $t$-adic valuation.
We want to apply Lemma~\ref{lem:inert-gen-cr-prod} with
$(\omega,f)=(\wt\sigma,\alpha)\in\calF(G,A(t)^\times)$,
$I$ the underlying division algebra of $(A(t),\wt\sigma,\alpha)$,
$c=t\in Z^2(G,K(t)^\times)$
and $N=(K(t),G,t)$.
By Proposition~~\ref{prop:gen-scal-ext-abelian},
$$ (A,\wt\sigma,\alpha)\otimes_F F(t)\cong (A\otimes_F F(t),\wt\sigma\otimes\id,\alpha\otimes 1)=(A(t),\wt\sigma,\alpha). $$
This shows that $I$ is an inertial lift of the underlying division algebra of $(A,\wt\sigma,\alpha)$ over $F(t)$
and $\exp I=\exp\bar I=\exp (A,\wt\sigma,\alpha)$
(cf. proof of Theorem~\ref{thm:bar-F-embed-ILP}).
The cyclic extension $K(t)/F(t)$ is inertial and 
$t$ is a prime element of $F(t)$.
Example~\ref{ex:NSR} then shows that $N=(K(t),G,t)$ is a nicely semiramified division algebra with $\exp N=n$.
Since $(\omega,fc)=(\wt\sigma,\alpha t)$,
Lemma~\ref{lem:inert-gen-cr-prod} yields (2)--(5).
Moreover, Corollary~\ref{cor:exp} shows (6).
It is clear from (1) and (2)
that if $A$ is a division algebra, then $A(x;\wt\sigma)$ is a division algebra.
\end{proof}

\begin{cor}
  \label{cor:twist-alg-glob-exp-form}
If $F$ is a global field, then we have the formula
$$ \exp A(x;\wt\sigma)=\exp A(\!(x;\wt\sigma)\!)=
\lcm(n,n_w\ind A_w)_{w\in\Val(K)}, $$
where $n_w=[K:F]_w$.
\end{cor}
\begin{proof}
  Theorem \ref{thm:twisted-alg-properties} and Proposition \ref{prop:exp-form-num-field}.
\end{proof}

\begin{cor}
  \label{cor:twist-alg-cross-prod}
If $A$ contains a strictly maximal subfield that is Galois over~$F$ with Galois group $G$,
then $A(x;\wt\sigma)$ and $A(\!(x;\wt\sigma)\!)$ are crossed products
with group $G$.
Now suppose that $F$ is a global field.
If $A$ is a symbol algebra, 
then $A(x;\wt\sigma)$ and $A(\!(x;\wt\sigma)\!)$ are crossed products.
If $A$ is a $p$-algebra, 
then $A(x;\wt\sigma)$ and $A(\!(x;\wt\sigma)\!)$ are cyclic crossed products.
\end{cor}
\begin{proof}
Let $L$ be a strictly maximal subfield of $A$ which is Galois over $F$ with $\Gal(L/F)=G$.
Then $L(t)$ and $L(\!(t)\!)$ are subfields of $A(x;\wt\sigma)$ and $A(\!(x;\wt\sigma)\!)$ 
that are Galois over $F(t)$ and $F(\!(t)\!)$ respectively with the same Galois group $G$.
By Theorem~\ref{thm:twisted-alg-properties} (1), they are also strictly maximal subfields,
thus $A(x;\wt\sigma)$ and $A(\!(x;\wt\sigma)\!)$ are crossed products with group $G$.
Now suppose that $F$ is a global field.
If $A$ is a symbol algebra, 
then Corollary~\ref{cor:symbol-alg-gal-max-sf} shows that $A$ contains a strictly maximal subfield $L$ that is Galois over $F$,
thus $A(x;\wt\sigma)$ and $A(\!(x;\wt\sigma)\!)$ are crossed products.
If $A$ is a $p$-algebra, then $L$ can be found cyclic over $F$ by Corollary~\ref{cor:symbol-alg-gal-max-sf},
thus $A(x;\wt\sigma)$ and $A(\!(x;\wt\sigma)\!)$ are cyclic crossed products.
\end{proof}

Besides that, if $A$ is a division algebra we can apply Theorem~\ref{thm:crossed-product} and achieve the following crossed product characterization.

\begin{cor}
  \label{cor:twist-alg-nonr-prod}
Suppose that $A\in\calD(K)$ is a division algebra.
Then the division algebra $A(x;\wt\sigma)$ is a crossed product
if and only if 
$A$ contains a strictly maximal subfield that is Galois over $F$.
In particular, if $A(x;\wt\sigma)$ is a crossed product, 
then $A$ is a crossed product.
The same holds for $A(\!(x;\wt\sigma)\!)$.
\end{cor}
\begin{proof}
Suppose  $A(x;\wt\sigma)$ is a crossed product.
Theorem~\ref{thm:twisted-alg-properties} states that $A(x;\wt\sigma)$ is an inertially split division algebras with residue algebra $A$.
Since $F$ is the residue field of $F(t)$,
Theorem~\ref{thm:crossed-product} shows that $A$ contains a maximal subfield that is Galois over $F$.
The converse was already shown in Corollary~\ref{cor:twist-alg-cross-prod}.
The same proof holds for $A(\!(x;\wt\sigma)\!)$.
\end{proof}

\begin{ex}
\label{ex:cycl-noncr-prod}
Let $\wt D$ be as in Example~\ref{ex:noncr-prod-cycl}.
We know from Example \ref{ex:sigma-extends} 
that 
$\sigma$ extends to an automorphism $\wt\sigma$ of $\wt D$.
Since $\wt D$ does not contain a maximal subfield which is Galois over $F$,
as shown in Example~\ref{ex:noncr-prod-cycl},
the algebras $\wt D(x;\wt\sigma)$ and $\wt D(\!(x;\wt\sigma)\!)$
are noncrossed products by Corollary \ref{cor:twist-alg-nonr-prod}.
It was also shown in Example~\ref{ex:noncr-prod-cycl}
that there is an $A\in\calD(F)$ of index and exponent $9$ with $A^K\sim\wt D$.
Therefore by the formulas of Theorem \ref{thm:twisted-alg-properties}, 
the index and exponent of $\wt D(x;\wt\sigma)$ and $\wt D(\!(x;\wt\sigma)\!)$
are equal to $9$.
\end{ex}

\begin{ex}
  \label{ex:twisted-scal-ext}
Let $A=B\otimes_F K$ for some $B\in\calA(F)$ and let $\wt\sigma=\id\otimes\sigma$.
Then by Theorem~\ref{thm:twisted-alg-properties},
\begin{align*}
  \deg A(x;\wt\sigma)&=\deg A(\!(x;\wt\sigma)\!)%
  =n\deg B, \\
  \ind A(x;\wt\sigma)&=\ind A(\!(x;\wt\sigma)\!)=n\ind A, \\
  \exp A(x;\wt\sigma)&=\exp A(\!(x;\wt\sigma)\!)=\lcm(n,\exp B).
\end{align*}
In particular, if $n\mid\exp B$, then 
$$ \exp A(x;\wt\sigma)=\exp A(\!(x;\wt\sigma)\!)=\exp B. $$
This can be used to construct new noncrossed products from old
such that the index increases while the exponent remains the same
(cf. \cite[Theorem~2]{saltman:noncr-prod-small-exp}).
For, if $B$ is a noncrossed product division algebra with $n\mid\exp B$
such that $A=B^K$ is a noncrossed product division algebra,
then $A(x;\wt\sigma)$ and $A(\!(x;\wt\sigma)\!)$ are noncrossed products by Corollary~\ref{cor:twist-alg-nonr-prod}.
Moreover, since $\ind A=\deg B$, they are division algebras with index $n\ind B$ and exponent $\exp B$.
\end{ex}

\begin{remark}
\label{rem:brussel-ex}
In Theorem~\ref{thm:twisted-alg-properties} the algebras
$A(x;\wt\sigma)$ and $A(\!(x;\wt\sigma)\!)$
are presented as generalized crossed products and then decomposed into a tensor product $I\otimes N$.
Of course, this process can be reversed.
As an application, we can show that the noncrossed products from \cite{brussel:noncr-prod} are of the form $\wt D(x;\wt\sigma)$ and $\wt D(\!(x;\wt\sigma)\!)$.
We switch to the notation that was used in Chapter 2.
The noncrossed products in \cite{brussel:noncr-prod} are the underlying division algebras $D$ of tensor products
$A\otimes_k(K(t)/k(t),\sigma,t)$,
where $A\in\calD(k)$ and $K/k$ is cyclic with $\Gal(K/k)=\gen{\sigma}$.
Let $\wt D=A_K$.
By Proposition~\ref{prop:res-classify},
there is a cyclic factor set $(\wt\sigma,\alpha)\in\calF(G,\ug{\wt D})$
such that $(\wt D,\wt\sigma,\alpha)\sim A$.
Then, using Proposition~\ref{prop:gen-scal-ext}, the Product Theorem
and Lemma~\ref{lem:twisted-alg-cycl-cr-prod},  
\begin{align*}
D&\sim A^{k(t)}\otimes_{k(t)} (K(t)/k(t),\sigma,t)
\sim (\wt D(t),\wt\sigma,\alpha)\otimes_{k(t)} (K(t)/k(t),\sigma,t) \\
&\sim (\wt D(t),\wt\sigma,\alpha t)\cong \wt D(x;\wt\sigma).
\end{align*}
By Theorem~\ref{thm:twisted-alg-properties}, 
$\wt D(x;\wt\sigma)$ is a division algebra,
thus $\wt D(x;\wt\sigma)\cong D$.
Analogously, the underlying division algebra of
$A\otimes_k (K(\!(t)\!)/k(\!(t)\!),\sigma,t)$ 
is $\wt D(\!(x;\wt\sigma)\!)$.
\end{remark}

\subsection{Iterated twisted function fields}

We shall now iterate the process of building twisted function fields and twisted Laurent series rings from $A$.
Let $K/F$ be a finite abelian extension with Galois group
$G=S_1\times\cdots\times S_r$,
where the $S_i$ are cyclic with $S_i=\gen{\sigma_i}$ and $|S_i|=n_i$,
and let $|G|=n_1\cdots n_r=n$.
Suppose that $\sigma_1,\ldots,\sigma_r$ extend to $F$-automorphisms
$\wt\sigma_1,\ldots,\wt\sigma_r$ of $A$ respectively.
By an \emph{iterated twisted function field} 
and an \emph{iterated twisted Laurent series ring}
we mean rings
\[
A(x_1;\sigma^*_1)(x_2;\sigma^*_2)\cdots(x_r;\sigma^*_r)
\qt{and}
A(\!(x_1;\sigma^*_1)\!)(\!(x_2;\sigma^*_2)\!)\cdots(\!(x_r;\sigma^*_r)\!)
\]
respectively,
where the $\sigma^*_i$ are automorphisms extending the $\wt\sigma_i$ respectively.
The following Theorem shows how such rings can be build from abelian factor sets.

\begin{theorem}
  \label{thm:iter-twisted}
Suppose there are elements $u_{ij},\alpha_i\in\ug A$, \mbox{$1\leq i,j\leq r$},
that satisfy the relations (\ref{eq:uij-rel}).
Then there exists an iterated twisted function field 
$$R=A(x_1;\sigma^*_1)(x_2;\sigma^*_2)\cdots(x_r;\sigma^*_r)$$
and an iterated twisted Laurent series ring
$$S=A(\!(x_1;\sigma^*_1)\!)(\!(x_2;\sigma^*_2)\!)\cdots(\!(x_r;\sigma^*_r)\!)$$
such that
\begin{equation}
  \label{eq:mult-in-iter-twisted}
  x_i a=\wt\sigma_i(a)x_i, \quad x_ix_j=u_{ij}x_jx_i
\end{equation}
for all $a\in A$, $1\leq i,j\leq r$.
The centres of these rings are the fields
\begin{equation*}
  \label{eq:centre-iter-twisted}
  Z(R)=F(t_1,\ldots,t_r) \qt{and} Z(S)=F(\!(t_1,\ldots,t_r)\!)
\end{equation*}
respectively, where $t_i=\alpha_i^{-1}x_i^{n_i}$ for $1\leq i\leq r$.
\end{theorem}
\begin{proof}
For simplicity of exposition this will be formulated only for $R$,
and is in complete analogy for $S$.
The construction can be found in \cite[Theorem 1.3]{amitsur-saltman:gen-abel-cr-prod} for the case $A=K$,
and it is the same in the general case. 
However we repeat it here for completeness.

We prove the theorem by induction on $r$.
It is trivial for $r=0$, 
so let $r>0$ and suppose that
$R':=A(x_1;\sigma^*_1)(x_2;\sigma^*_2)\cdots(x_r;\sigma^*_{r-1})$
is constructed such that (\ref{eq:mult-in-iter-twisted}) holds for all $1\leq i,j<r$
and $Z(R')=F'(t_1,\ldots,t_{r-1})$,
where $F'=\Fix(\sigma_1,\ldots,\sigma_{r-1})$.
Define the automorphism $\sigma^*_r$ of $R'$ by
\begin{equation}
  \label{eq:sigma-*-def}
  \sigma_r^*(a):=\wt\sigma_r(a), \quad
  \sigma_r^*(x_i):=u_{ri}x_i \qt{for all $a\in A$, $i<r$.}
\end{equation}
To see that (\ref{eq:sigma-*-def}) defines an automorphism 
we have to check that $\sigma^*_r$ preserves the relations (\ref{eq:mult-in-iter-twisted}).
The calculation is the same as in the proof of Proposition~\ref{prop:ex-abelian-factor-set}.
Next we show that $\sigma^*_r$ fixes $F(t_1,\ldots,t_{r-1})$.
By (\ref{eq:mult-in-iter-twisted}) and (\ref{eq:uij-rel-3}) we get
\begin{align*}
\sigma^*_r(t_i)&=\sigma^*_r(\alpha_i^{-1}x_i^{n_i})
=\sigma^*_r(\alpha_i)^{-1}(u_{ri}x_i)^{n_i} \\
&=\wt\sigma_r(\alpha_i)^{-1} N_i(u_{ri})x_i^{n_i}=\alpha_i^{-1}x_i^{n_i}=t_i
\end{align*}
for all $1\leq i<r$.
Therefore $\Gal(Z(R')/F(t_1,\ldots,t_{r-1}))=\gen{\sigma^*_r|_{Z(R')}}$.
Like in the proof of Proposition~\ref{prop:ex-abelian-factor-set} we show
${\sigma_r^*}^{n_r}=\iota_{\alpha_r}$ and $\sigma_r^*(\alpha_r)=\alpha_r$.
Therefore we can build $R:=R'(x_r;\sigma^*_r)$ as in \S~\ref{sec:One-twist-indet},
and get $Z(R)=F(t_1,\ldots,t_r)$ for $t_r=\alpha_r^{-1}x_r^{n_r}$.
\end{proof}

\begin{remark}
  \label{rem:write-iter-twist}
The rings $R$ and $S$ in Theorem~\ref{thm:iter-twisted} are completely described over $A$ by the rules (\ref{eq:mult-in-iter-twisted}).
We therefore also write
$$R=A(x;\wt\sigma;u)=A(x_1,\ldots,x_r;\wt\sigma_1,\ldots,\wt\sigma_r;u_{ij})$$
and
$$S=A(\!(x;\wt\sigma;u)\!)=A(\!(x_1,\ldots,x_r;\wt\sigma_1,\ldots,\wt\sigma_r;u_{ij})\!).$$
In the case $r=2$ the $u_{ij}$ are determined by the single element $u=u_{21}$
(\cf Remark~\ref{rem:abelian-r-2}).
We then write
\[ A(x_1,x_2;\wt\sigma_1,\wt\sigma_2;u) \qt{and}
A(\!(x_1,x_2;\wt\sigma_1,\wt\sigma_2;u)\!) \]
for $R$ and $S$ respectively.
\end{remark}

\begin{lemma}
  \label{lem:iter-twisted-abel-cr-prod}
$A(x;\wt\sigma;u)$ and $A(\!(x;\wt\sigma;u)\!)$ are central simple $F(t_1,\ldots,t_r)$- and $F(\!(t_1,\ldots,t_r)\!)$-algebras respectively and
\begin{align*}
  A(x;\wt\sigma;u)&\cong(A(t_1,\cdots,t_r),\wt\sigma,u,\alpha t), \\  
  A(\!(x;\wt\sigma;u)\!)&\cong(A(\!(t_1,\ldots,t_r)\!),\wt\sigma,u,\alpha t).
\end{align*}
Here $\wt\sigma_i$ also denotes the extension of $\wt\sigma_i$ to $A(t_1,\ldots,t_r)$ and $A(\!(t_1,\ldots,t_r)\!)$ that fixes $t_1,\ldots,t_r$,
and $(\alpha t)_i$ stands for $\alpha_i t_i$.
\end{lemma}
\begin{proof}
We give the proof here for $A(x;\wt\sigma;u)$.
It is essentially the same as the proof of Lemma~\ref{lem:twisted-alg-cycl-cr-prod}.
The set \mbox{$\sett{x_1^{i_1}\cdots x_r^{i_r}}{0\leq i_k<n_k,1\leq k\leq r}$}
generates $A(x;\wt\sigma;u)$ over $A(t_1,\ldots,t_r)$ 
and is free over $A(t_1,\ldots,t_r)$ %
by Lem\-ma~\ref{lem:abelian-cr-prod-basis}.
Hence $A(x;\wt\sigma;u)=\bigoplus_{0\leq i_k\leq n_k} A(t_1,\ldots,t_r)x_1^{i_1}\cdots x_r^{i_r}$.
This shows that $C_{A(x;\wt\sigma;u)}(K(t_1,\ldots,t_r))=A(t_1,\ldots,t_r)$.
The assertion follows from Lemma~\ref{lem:abelian-cr-prod-basis} since the $x_i$ satisfy the relations (\ref{eq:mult-in-iter-twisted})
and $x_i^{n_i}=\alpha_i t_i$ for all $1\leq i,j\leq r$.
\end{proof}

\begin{theorem}
  \label{thm:iter-twist-properties}
$A(x;\wt\sigma;u)$ and $A(\!(x;\wt\sigma;u)\!)$ are central simple algebras over
$F(t_1,\ldots,t_r)$ and $F(\!(t_1,\ldots,t_r)\!)$ respectively with
\begin{enumerate}
\item $\deg A(x;\wt\sigma;u)=A(\!(x;\wt\sigma;u)\!)=n\deg A,$
\item $\ind A(x;\wt\sigma;u)=A(\!(x;\wt\sigma;u)\!)=n\ind A,$
\item $\exp A(x;\wt\sigma;u)=A(\!(x;\wt\sigma;u)\!)=\lcm(\exp G,\exp (A,\wt\sigma,u,\alpha)),$
\item $[A(x;\wt\sigma;u)]\in\SBr(F(t_1,\ldots,t_r)), 
[A(\!(x;\wt\sigma;u)\!)]\in\SBr(F(\!(t_1,\ldots,t_r)\!)),$
\item $\bar{A(x;\wt\sigma;u)_{F(t_1,\ldots,t_r)}}\cong\bar{A(\!(x;\wt\sigma;u)\!)_{F(\!(t_1,\ldots,t_r)\!)}}\cong A_K.$
\end{enumerate}
Here, $F(t_1,\ldots,t_r)$ and $F(\!(t_1,\ldots,t_r)\!)$ are regarded with respect to the composite valuation of the $t_i$-adic valuations.
In particular, if $A$ is a division algebra then
$A(x;\wt\sigma;u)$ and $A(\!(x;\wt\sigma;u)\!)$ are division algebras.
\end{theorem}
\begin{proof}
This theorem is proved just like Theorem~\ref{thm:twisted-alg-properties}
using Lemma \ref{lem:iter-twisted-abel-cr-prod}.
Here Lemma~\ref{lem:inert-gen-cr-prod} is applied
with $I$ the underlying division algebra of $(A(t_1,\ldots,t_r),\wt\sigma,u,\alpha)$
and $N=(K(t_1,\ldots,t_r),G,\hat c_1\cdots\hat c_r)$,
where the $\hat c_i$ are inflations of cyclic cocycles defined by $t_i$.
Then $N$ is a nicely semiramified division algebra by Example~\ref{ex:NSR}
because $v(t_1),\ldots,v(t_r)$ form a base of the value group of $F(t_1,\ldots,t_r)$,
where $v$ denotes the composite valuation of the $t_i$-adic valuations.
\end{proof}

Analogously to Corollary  \ref{cor:twist-alg-nonr-prod} we get

\begin{cor}
  \label{cor:iter-alg-nonr-prod}
Suppose that $A\in\calD(K)$ is a division algebra.
Then the division algebra $A(x;\wt\sigma;u)$ is a crossed product
if and only if $A$ contains a maximal subfield that is Galois over $F$.
In particular, if $A(x;\wt\sigma;u)$ is a crossed product,
then $A$ is a crossed product.
The same holds for $A(\!(x;\wt\sigma;u)\!)$.
\end{cor}

The following is an application of Proposition~\ref{prop:gen-scal-ext-abelian}.
\begin{prop}
  \label{prop:gen-scal-ext-iterated}
Let $(\wt\sigma,u,\alpha)\in\calF(G,A^\times)$ be an abelian factor set.
Let $F'/F$ be any field extension such that $F'$ is linearly disjoint to $K$ over~$F$.
Then $K\otimes_F F'$ is Galois over $F'$ with $\Gal(K\otimes_F F'/F')\cong G$ and
\begin{equation*}
A(x;\wt\sigma;u)\otimes_{F(t_1,\ldots,t_r)} F'(t_1,\ldots,t_r)\cong (A\otimes_{F} F')(x;\wt\sigma\otimes\id,u\otimes 1).
\end{equation*}
The same relation holds for the Laurent series rings.
\end{prop}
\begin{proof}
We shortly write $t$ for $t_1,\ldots,t_r$.
By Lemma~\ref{lem:iter-twisted-abel-cr-prod} and Proposition~\ref{prop:gen-scal-ext-abelian}, we have
\begin{multline*}
  A(x;\wt\sigma,u)\otimes_{F(t)} F'(t)\cong(A(t),\wt\sigma,u,\alpha t)\otimes_{F(t)} F'(t) \\
\cong(A(t)\otimes_{F(t)}F'(t),\wt\sigma\otimes 1,u\otimes 1,\alpha t\otimes 1)
\cong(A\otimes_F F')(x;\wt\sigma\otimes 1;u\otimes 1).
\end{multline*}
\end{proof}

\section{Automorphisms of symbol algebras}
\label{sec:Autom-symb}

The first (and major) problem 
in the computation of examples of factor sets $(\omega,f)\in\calF(G,\ug A)$
lies in the computation of extensions $\wt\sigma$ of the automorphism $\sigma\in G$ to $A$.
The present section discusses this problem for the case that $A$ is a symbol algebra.
Let $K$ be a field that contains a primitive $n$-th root of unity $\zeta$.
For $a,b\in\ug K$, the symbol algebra $(\frac{a,b}{K,\zeta})$
is a central simple $K$-algebra of degree $n$ that is generated
by elements $i,j$ such that $i^n=a, j^n=b$ and $ji=\zeta ij$.

\begin{lemma}
  \label{lem:autom-symbol}
Let $A=(\frac{a,b}{K,\zeta})$ be a symbol algebra, $a,b\in K^\times$,
and let $\sigma$ be an automorphism of $K$ fixing $\zeta$.
If $\wt\sigma$ is an automorphism of $A$ that extends $\sigma$,
then for $i':=\wt\sigma(i)$ and $j':=\wt\sigma(j)$,
\begin{equation}
  \label{eq:sigma-u-v}
{i'}^n=\sigma(a), \quad {j'}^n=\sigma(b), \quad j'i'=\zeta i'j'.
\end{equation}
Conversely, if $i',j'\in A$ satisfy (\ref{eq:sigma-u-v}),
then an extension $\wt\sigma$ of $\sigma$ to~$A$ is defined by
$\wt\sigma(i):=i'$ and $\wt\sigma(j):=j'$.
\end{lemma}

The proof is obvious.
Now assume that $A$ is a division algebra.
More than Lemma \ref{lem:autom-symbol} can be stated in the special case that $\sigma$ fixes $\zeta$ and $b$.

\begin{lemma}
  \label{lem:autom-symbol-special}
Let $A=(\frac{a,b}{K,\zeta})$ be a symbol algebra that is a division algebra, 
$a,b\in K^\times$, 
and let $\sigma$ be an automorphism of $K$ fixing $\zeta$ and $b$.
Then $\sigma$ extends to an automorphism $\wt\sigma$ of $A$ 
if and only if there exists $\lambda\in K(j)$ with
$$\No_{K(j)/K}(\lambda)=\frac{\sigma(a)}{a}.$$
For any such $\lambda$, an extension $\wt\sigma$ of $\sigma$ is defined by
$$ \wt\sigma(i):=\lambda i, \quad \wt\sigma(j):=j. $$
\end{lemma}
\begin{proof}
Suppose that $\sigma$ extends to an automorphism $\wt\sigma$ of $A$.
Then $\wt\sigma(j)^n=\wt\sigma(b)=b=j^n$,
hence by the theorem of Skolem-Noether there is a $d\in A^\times$ such that
$\iota_d\wt\sigma(j)=j$.
Since $\iota_d\wt\sigma$ also extends $\sigma$,
we can assume \wolg that $\wt\sigma(j)=j$.
Let $\lambda=\wt\sigma(i)i^{-1}$, $i'=\wt\sigma(i)$ and $j'=\wt\sigma(j)$,
\ie $i'=\lambda i$ and $j'=j$.
Since $j'i'=j\lambda i$ and
$\zeta i'j'=\zeta\lambda ij = \lambda ji$,
\begin{equation}
  \label{eq:lambda-in-K(j)}
j'i'=\zeta i'j' \iff
j\lambda=\lambda j \iff
\lambda\in K(j),
\end{equation}
so $\lambda\in K(j)$ by Lemma \ref{lem:autom-symbol}.
Conjugation by $i$ induces a generating automorphism of $\Gal(K(j)/K)$,
thus ${i'}^n=(\lambda i)^n=\No_{K(j)/K}(\lambda)a$.
Therefore
\begin{equation} 
  \label{eq:norm-lambda}
{i'}^n=\sigma(a) \iff \No_{K(j)/K}(\lambda)=\frac{\sigma(a)}{a},
\end{equation} 
so $\No_{K(j)/K}(\lambda)=\frac{\sigma(a)}{a}$ by Lemma \ref{lem:autom-symbol}.
Conversely, suppose there is $\lambda\in K(j)$ with 
$\No_{K(j)/K}(\lambda)=\frac{\sigma(a)}{a}$.
Let $i'=\lambda i$ and $j'=j$.
Then (\ref{eq:sigma-u-v}) follows from 
(\ref{eq:lambda-in-K(j)}), (\ref{eq:norm-lambda}) and $\sigma(b)=b$.
Lemma \ref{lem:autom-symbol} states that an extension $\wt\sigma$ of $\sigma$ is defined by
$\wt\sigma(i):=i'=\lambda i$ and $\wt\sigma(j):=j'=j$.
\end{proof}

\begin{remark}
Lemma \ref{lem:autom-symbol-special} reduces the computation of $\wt\sigma$ to the solution of a relative norm equation.
If $K$ is a number field, this problem can be handled with methods from computational algebraic number theory. 
In concrete examples the KASH software \cite{kash} can be used.
\end{remark}

\section{Examples of noncrossed product division algebras}
\label{sec:Exampl-noncr-prod}
The examples of this section were computed with the help of the KASH software \cite{kash}. 
However, we will give hints how all necessary calculations can be performed by hand.

\subsection{A noncrossed product division algebra with index $8$ and exponent $8$}
\label{sec:An-example-with-8}

We take over the situation from Example~\ref{ex:not-embed-abel} and~\ref{ex:noncr-prod-abel-1}.
So let $K=\Q(\sqrt{3},\sqrt{-7})$.
Then $K/\Q$ is abelian with $[K:\Q]=4$ and
$\Gal(K/\Q)=\gen{\sigma_1}\oplus\gen{\sigma_2}$, where
\begin{alignat*}{2}
  \sigma_1(\sqrt{3})&=-\sqrt{3},&\qquad \sigma_1(\sqrt{-7})&=\sqrt{-7}, \\
  \sigma_2(\sqrt{3})&=\sqrt{3},& \sigma_2(\sqrt{-7})&=-\sqrt{-7}.
\end{alignat*}
For
\begin{align*}
   \pi_1:=1+\sqrt{3} \in\Q(\sqrt{3}) \qt{and} \pi_2:=\frac{1+\sqrt{-7}}{2} \in\Q(\sqrt{-7}), 
\end{align*}
we have
\begin{align}
  \label{eq:norm-pi}
  \No_{\Q(\sqrt{3})/\Q}(\pi_1)=\pi_1\sigma_1(\pi_1)=-2,\quad
  \No_{\Q(\sqrt{-7})/\Q}(\pi_2)=\pi_2\sigma_2(\pi_2)=2. 
\end{align}
Let $D$ be the quaternion algebra $(\frac{a,b}{K})$ with
\begin{equation*}
\begin{aligned}
  a&:=\sqrt{3}\pi_1 =3+\sqrt{3} &\quad&\in\Q(\sqrt{3}), \\
  b&:=\sqrt{-7}\pi_2 =\frac{-7+\sqrt{-7}}{2} &&\in\Q(\sqrt{-7}),
\end{aligned}
\end{equation*}
\ie $D=K\oplus Ki \oplus Kj \oplus Kk$
with $i^2=a,j^2=b,ij=-ji=k$.
We know from Example~\ref{ex:noncr-prod-abel-1}
that $D$ is a division algebra.

Now extend $\sigma_1$ and $\sigma_2$ to $F$-autmorphisms of $D$.
Since $\sigma_1(b)=b$ and $\sigma_2(a)=a$
we can use Lemma~\ref{lem:autom-symbol-special}.
Set 
$\lambda:=\frac{\lambda_0}{\pi_2}$ and $\mu:=\frac{\mu_0}{\pi_1}$ for
\begin{align*}
  \lambda_0 = \sigma_1(\pi_1)(-1+i) \in K(i) \qt{and}
  \mu_0 = \pi_2 + j \in K(j) .
\end{align*}
We use the notation $\bar\lambda_0=\sigma_1(\pi_1)(-1-i)$
and $\bar\mu_0=\pi_2-j$.
Then
\begin{equation}
  \label{eq:norm-0}
 \No_{K(i)/K}(\lambda_0) = \lambda_0\bar\lambda_0=-2 \qt{and}  
 \No_{K(j)/K}(\mu_0) = \mu_0\bar\mu_0 = 2 ,
\end{equation}
so by (\ref{eq:norm-pi}),
\begin{align*}
 \No_{K(j)/K}(\mu) &= \frac{2}{{\pi_1}^2} = -\frac{\pi_1\sigma_1(\pi_1)}{{\pi_1}^2} = -\frac{\sigma_1(\pi_1)}{\pi_1}=\frac{\sigma_1(a)}{a} \\
\intertext{and}
 \No_{K(i)/K}(\lambda) &= \frac{-2}{{\pi_2}^2} = -\frac{\pi_2\sigma_2(\pi_2)}{{\pi_2}^2} = -\frac{\sigma_2(\pi_2)}{\pi_2}=\frac{\sigma_2(b)}{b}.
\end{align*}
Hence by Lemma~\ref{lem:autom-symbol-special},
extensions $\wt\sigma_1$, $\wt\sigma_2$ of $\sigma_1$,$\sigma_2$ to $D$ are defined by
\begin{alignat*}{2}
  \wt\sigma_1(i)&=\mu i,&\qquad \wt\sigma_1(j)&=j, \\
  \wt\sigma_2(i)&=i,& \wt\sigma_2(j)&=\lambda j.
\end{alignat*}

To build an iterated twisted function field $D(x_1,x_2;\wt\sigma_1,\wt\sigma_2;u)$
and an iterated twisted Laurent series ring $D(\!(x_1,x_2;\wt\sigma_1,\wt\sigma_2;u)\!)$ over $D$
according to Theorem \ref{thm:iter-twisted} (\cf also Remark \ref{rem:write-iter-twist}),
we give elements $\alpha_1, \alpha_2, u\in D^\times$ 
that satisfy the relations (\ref{eq:abel-r=2}).
These are
\begin{align*}
  \alpha_1& := \frac{1}{\pi_2}\mu_0j = \sqrt{-7}+j,\\
  \alpha_2& := \sqrt{3}\lambda_0i = -6+(3-\sqrt{3})i,\\
  u& := \frac{1}{2\sigma_1(\pi_1)}(\lambda_0\bar\mu_0-2) 
= \frac{1}{4}(1+2\sqrt{3}-\sqrt{-7}+(1+\sqrt{-7})i+2j-2k).
\end{align*}
The verification of the relations (\ref{eq:abel-r=2}) is deferred to the end of this section,
in order to not interrupt the exposition here.
By Theorem~\ref{thm:iter-twist-properties},
$D(x_1,x_2;\wt\sigma_1,\wt\sigma_2;u)$
and $D(\!(x_1,x_2;\wt\sigma_1,\wt\sigma_2;u)\!)$
are division algebras that are inertially split with residue algebra $D$ and
\begin{align*}
  Z(D(x_1,x_2;\wt\sigma_1,\wt\sigma_2;u))&=\Q(t_1,t_2), \\
  Z(D(\!(x_1,x_2;\wt\sigma_1,\wt\sigma_2;u)\!))&=\Q(\!(t_1,t_2)\!),
\end{align*}
where $t_1=\alpha_1^{-1}x_1^2$, $t_2=\alpha_2^{-1}x_2^2$.
We have seen in Example~\ref{ex:noncr-prod-abel-1} that $D$ does not contain a maximal subfield which is Galois over $\Q$,
hence $D(x_1,x_2;\wt\sigma_1,\wt\sigma_2;u)$
and $D(\!(x_1,x_2;\wt\sigma_1,\wt\sigma_2;u)\!)$
are noncrossed products by Corollary~\ref{cor:iter-alg-nonr-prod}.
Moreover by Theorem~\ref{thm:iter-twist-properties},
\begin{align*}
 \ind D(x_1,x_2;\wt\sigma_1,\wt\sigma_2;u)&=\ind D(\!(x_1,x_2;\wt\sigma_1,\wt\sigma_2;u)\!) = 4\ind D=8, \\
 \exp D(x_1,x_2;\wt\sigma_1,\wt\sigma_2;u)&=\exp D(\!(x_1,x_2;\wt\sigma_1,\wt\sigma_2;u)\!) = \lcm(2,\exp A),
\end{align*}
with $A=(D,\wt\sigma,u,\alpha)$.
To compute $\exp A=\ind A$ we make use of Corollary~\ref{cor:ind-local-scal-ext}.
We know from Example~\ref{ex:not-embed-abel} that the $3$-adic valuation $v_1\in\V0(\Q)$ extends uniquely to a valuation $w_1\in\V0(K)$,
and we know from Example~\ref{ex:noncr-prod-abel-1} that $D_{w_1}$ is a divison algebra.
This means $n_{w_1}=[K:\Q]_{w_1}=4$  and $\ind D_{w_1}=2$.
Since $n=[K:\Q]=4$ and $\ind D=2$,
Corollary~\ref{cor:ind-local-scal-ext} yields 
$8\mid\ind A\mid 8$.
Therefore
\begin{equation*}
 \exp D(x_1,x_2;\wt\sigma_1,\wt\sigma_2;u)=\exp D(\!(x_1,x_2;\wt\sigma_1,\wt\sigma_2;u)\!) = \lcm(2,8)=8.
\end{equation*}

The rest of this section gives hints
how to verify that $\alpha_1,\alpha_2,u$ satisfy (\ref{eq:abel-r=2}).
Note that $n_1=n_2=2$ and that $\alpha_1\in K(j)$ and $\alpha_2\in K(i)$.
The following relations, as well as (\ref{eq:norm-0}), 
will be frequently used without further mention :
\begin{align*}
  \wt\sigma_1(\mu_0)& =\mu_0,& \wt\sigma_2(\lambda_0)& =\lambda_0,&
  i\lambda_0& =\lambda_0 i,& j\lambda_0&=\bar\lambda_0 j, \\
  \wt\sigma_1(\bar\mu_0)& =\bar\mu_0,& \wt\sigma_2(\bar\lambda_0)&=\bar\lambda_0,&
  i\mu_0& =\bar\mu_0 i,& j\mu_0&=\mu_0 j.
\end{align*}
We first check (\ref{eq:abel-r=2-1}--\ref{eq:abel-r=2-3}).
Since $\wt\sigma_1$ fixes $\pi_2,\mu_0$ and $j$, we have $\wt\sigma_1(\alpha)=\alpha$,
and since $\wt\sigma_2$ fixes $\sqrt{3},\lambda_0$ and $i$, we have $\wt\sigma_2(\beta)=\beta$.
To verify the identities about the inner automorphisms we will make use of

\begin{lemma}
  \label{lem:comp-inner-autom}
Let $D$ be a quaternion algebra that is a division algebra and 
let $\varphi$ be an inner automorphism of $D$.
If $x,y\in D\bs K$ with $x\notin K(y)$ such that $\varphi(x)=x$ and $\varphi(y)=xyx^{-1}$,
then $\varphi=\iota_x$.
\end{lemma}
\begin{proof}
If $x\notin K(y)$ then $\{1,x\}$ is a $K(y)$-basis of $D$.
Therefore $\varphi$ is already determined by $\varphi(x)$ and $\varphi(y)$.
\end{proof}

We first apply Lemma~\ref{lem:comp-inner-autom} to $\wt\sigma_1^2$ and $\wt\sigma_2^2$.
Obviously,
$\wt\sigma_1^2(\alpha_1)=\alpha_1$ and $\wt\sigma_2^2(\alpha_2)=\alpha_2$.
Furthermore,
\begin{align*}
  \wt\sigma_1^2(i)\alpha_1&= \wt\sigma_1(\frac{\mu_0}{\pi_1}i)\alpha_1
=\frac{\mu_0}{\sigma_1(\pi_1)}\frac{\mu_0}{\pi_1}i\frac{\mu_0j}{\pi_2} 
=\frac{\mu_0 j}{\pi_2}i
=\alpha_1 i
\intertext{and}
  \wt\sigma_2^2(j)\alpha_2&= \wt\sigma_2(\frac{\lambda_0}{\pi_2}j)\alpha_2
=\frac{\lambda_0}{\sigma_2(\pi_2)}\frac{\lambda_0}{\pi_2}j\sqrt{3}\lambda_0i
=\sqrt{3}\lambda_0ij
=\alpha_2 j.
\end{align*}
Since $\alpha_1\notin K(i)$ and $\alpha_2\notin K(j)$
Lemma \ref{lem:comp-inner-autom} shows $\wt\sigma_1^2=\iota_{\alpha_1}$
and $\wt\sigma_2^2=\iota_{\alpha_2}$,
thus (\ref{eq:abel-r=2-1}) and (\ref{eq:abel-r=2-2}) are verified.
Next, we apply Lemma \ref{lem:comp-inner-autom}
to $\wt\sigma_2\wt\sigma_1\wt\sigma_2^{-1}\wt\sigma_1^{-1}$
with $x=u$ and $y=\wt\sigma_1\wt\sigma_2(i)$.
To show $\wt\sigma_2\wt\sigma_1\wt\sigma_2^{-1}\wt\sigma_1^{-1}(u)=u$
we check that 
$\wt\sigma_1\wt\sigma_2(u')=\wt\sigma_2\wt\sigma_1(u')=u$
for $u'=\frac{-1}{2\pi_2}(\bar\lambda_0\mu_0+2)$.
Using 
\begin{equation}
  \label{eq:tau-mu_0}
  \wt\sigma_2(\mu_0)=\frac{2+\lambda_0j}{\pi_2} \qqt{and}
  \wt\sigma_1(\bar\lambda_0)=-(\pi_1+\mu_0 i)
\end{equation}
we first get
\begin{gather*}
  \wt\sigma_2(u')
=\frac{-1}{2\sigma_2(\pi_2)}(\bar\lambda_0\frac{2+\lambda_0j}{\pi_2}+2)
=\frac{-1}{2\sigma_2(\pi_2)\pi_2}(2\bar\lambda_0-2j+2\pi_2)
=-\frac{1}{2}(\bar\lambda_0+\bar\mu_0)
\intertext{and}
  \wt\sigma_1(u')=\frac{1}{2\pi_2}((\pi_1+\mu_0 i)\mu_0-2)
=\frac{1}{2\pi_2}(\pi_1\mu_0+2i-2)
=\frac{\pi_1}{2\pi_2}(\mu_0-\lambda_0).
\end{gather*}
Thus, again using (\ref{eq:tau-mu_0}),
\begin{gather*}
  \wt\sigma_1\wt\sigma_2(u')
=\frac{1}{2}(\pi_1+\mu_0 i-\bar\mu_0)
=\frac{1}{2\sigma_2(\pi_1)}(\sigma_2(\pi_1)\pi_1+\sigma_2(\pi_1)(-1+i)\bar\mu_0)
=u
\intertext{and}
  \wt\sigma_2\wt\sigma_1(u')
=\frac{\pi_1}{2\sigma_2(\pi_2)}(\frac{2+\lambda_0j}{\pi_2}-\lambda_0)
=\frac{\pi_1}{4}(2+\lambda_0(v-\pi_2))
=u,
\end{gather*}
showing $\wt\sigma_1\wt\sigma_2(u')=\wt\sigma_2\wt\sigma_1(u')=u$,
hence $\wt\sigma_2\wt\sigma_1\wt\sigma_2^{-1}\wt\sigma_1^{-1}(u)=u$.
Since we have shown $\wt\sigma_2^{-1}\wt\sigma_1^{-1}(u)=u'$,
$u\notin K(y)$ is equivalent to $u'\notin K(i)$,
and $u'=\frac{-1}{2\pi_2}(\bar\lambda_0\mu_0+2)\notin K(i)$ is obvious.
It remains to show 
$\wt\sigma_2\wt\sigma_1(i)u=u\wt\sigma_1\wt\sigma_2(i)$.
Then (\ref{eq:abel-r=2-3}) follows from Lemma~\ref{lem:comp-inner-autom}.
Because $\wt\sigma_2\wt\sigma_1(i)u=-\frac{1}{4}\wt\sigma_2(\mu_0)(\lambda_0\mu_0-2)i$ 
and $u\wt\sigma_1\wt\sigma_2(i)=-\frac{1}{2}(\lambda_0-\mu_0)i$,
we show 
\begin{equation}
  \label{eq:commutator}
\wt\sigma_2(\mu_0)(\lambda_0\mu_0-2)=2(\lambda_0-\mu_0).  
\end{equation}
Using $\mu_0-j=\pi_2$ and $\mu_0j+2=\pi_2\mu_0$,
we have
\begin{equation*}
  \begin{split}
  \wt\sigma_2(\mu_0)(\lambda_0\mu_0-2)
&=\frac{1}{\pi_2}(2+\lambda_0 j)(\lambda_0\mu_0-2) \\
&=\frac{2}{\pi_2}(\lambda_0(\mu_0-j)-(\mu_0 j+2))
= 2(\lambda_0-\mu_0),
  \end{split}
\end{equation*}
\ie (\ref{eq:commutator}) and (\ref{eq:abel-r=2-3}) are verified.

To verify (\ref{eq:abel-r=2-4}--\ref{eq:abel-r=2-5}) we show
\begin{align}
\label{eq:(2)-1}
u\wt\sigma_1(u)\alpha_1&=\wt\sigma_2(\alpha_1)
\intertext{and}
\label{eq:(2)-2}
\alpha_2^{-1}\wt\sigma_2(u)u&=\wt\sigma_1(\alpha_2)^{-1}.
\end{align}
The right handsides of these equations are
\begin{align}
\label{eq:right-side}
  \wt\sigma_2(\alpha_1)=\frac{1}{\pi_2}(\lambda_0-j)j \qt{and}
  \wt\sigma_1(\alpha_2)^{-1}=
\frac{1}{\sigma_1(\pi_1)^2\sqrt{3}}i^{-1}(\sigma_1(\pi_1)i-\bar\mu_0),
\end{align}
the latter following from
$\wt\sigma_1(\alpha_2)=\sqrt{3}(\mu_0+\sigma_1(\pi_1)i)i$.
Since $\pi_2\wt\sigma_1\wt\sigma_2(j)=\wt\sigma_1(\lambda_0)j$, we get
\begin{equation*}
  \wt\sigma_1(u)\alpha_1
=\frac{1}{2\pi_1\pi_2}(\wt\sigma_1(\lambda_0)\bar\mu_0-2)\mu_0j
=\frac{1}{\pi_1\pi_2}(\pi_2\wt\sigma_1\wt\sigma_2(j)-\mu_0j),
\end{equation*}
and by (\ref{eq:abel-r=2-3}),
\begin{equation*}
    u\wt\sigma_1(u)\alpha_1 = \frac{1}{\pi_1\pi_2}(\pi_2\wt\sigma_2\wt\sigma_1(j)u-u\mu_0j)
= \frac{1}{\pi_1\pi_2}(v\bar\lambda_0u-u\mu_0j).
\end{equation*}
Using 
\[
\bar\lambda_0u =\frac{-1}{\sigma_1(\pi_1)}(\bar\lambda_0+\bar\mu_0) \qt{and}
u\mu_0=\frac{1}{\sigma_1(\pi_1)}(\lambda_0-\mu_0),
\]
this %
implies
\begin{equation*}
  \begin{split}
    u\wt\sigma_1(u)\alpha_1 
=\frac{1}{\pi_1\pi_2\sigma_1(\pi_1)}(-j(\bar\lambda_0+\bar\mu_0)-(\lambda_0-\mu_0)j) \\
= \frac{1}{2\pi_2}(2\lambda_0+\bar\mu_0-\mu_0)j
= \frac{1}{\pi_2}(\lambda_0-j)j,
  \end{split}
\end{equation*}
which together with (\ref{eq:right-side}) proves (\ref{eq:(2)-1}).
Since $\mu_0^{-1}=\frac{1}{2}\bar\mu_0$,
we get
$$\wt\sigma_2\wt\sigma_1(i)^{-1}=-\frac{i^{-1}\wt\sigma_2(\bar\mu_0)}{\sigma_1(\pi_1)}, $$
hence
\begin{equation*}
  \alpha_2^{-1}\wt\sigma_2(u)
=\frac{1}{2\sigma_1(\pi_1)\sqrt{3}}i^{-1}\lambda_0^{-1}(\lambda_0\wt\sigma_2(\bar\mu_0)-2)
=\frac{1}{2\sqrt{3}}(-\wt\sigma_2\wt\sigma_1(i)^{-1}+\pi_1 i^{-1}\lambda_0^{-1}).
\end{equation*}
By (\ref{eq:abel-r=2-3})
and $\wt\sigma_1\wt\sigma_2(i)=\frac{\mu_0}{\pi_1}i=i\frac{\bar\mu_0}{\pi_1}$,
\begin{equation*}
  \alpha_2^{-1}\wt\sigma_2(u)u
=\frac{1}{2\sqrt{3}}(-u\wt\sigma_1\wt\sigma_2(i)^{-1}+\pi_1 i^{-1}\lambda_0^{-1}u)
=\frac{1}{\sigma_1(\pi_1)\sqrt{3}}(u\bar\mu_0^{-1}i^{-1}-i^{-1}\lambda_0^{-1}u).
\end{equation*}
Using 
\[
u\bar\mu_0^{-1} =\frac{1}{2\sigma_1(\pi_1)}(\lambda_0-\mu_0) \qt{and}
\lambda_0^{-1}u=\frac{1}{2\sigma_1(\pi_1)}(\bar\lambda_0+\bar\mu_0),
\]
this implies
\begin{equation*}
  \alpha_2^{-1}\wt\sigma_2(u)u
= \frac{1}{2\sigma_1(\pi_2)^2\sqrt{3}}i^{-1}(\lambda_0-\bar\lambda_0-2\bar\mu_0)
= \frac{1}{\sigma_1(\pi_2)^2\sqrt{3}}i^{-1}(\sigma_1(\pi_1)i-\bar\mu_0),
\end{equation*}
which together with (\ref{eq:right-side}) proves (\ref{eq:(2)-2}).
Thus (\ref{eq:abel-r=2-4}--\ref{eq:abel-r=2-5}) are verified.

\subsection{A noncrossed product division algebra with index $16$ and exponent $8$}
\label{sec:An-example-with-16}

The noncrossed product example from \S~\ref{sec:An-example-with-8} 
can be modified
to increase its index while the exponent remains the same
according to Example \ref{ex:twisted-scal-ext}.
This will be carried out for the function field and is analogous for the Laurent series ring.
Let $B=D(x_1,x_2;\wt\sigma_1,\wt\sigma_2;u)$ be the noncrossed product division algebra with index and exponent $8$ from \S~\ref{sec:An-example-with-8}.
Let $F=Z(B)=\Q(t_1,t_2)$ and $L=F(\sqrt{37})$.
Since $\sqrt{37}\notin K=\Q(\sqrt{3},\sqrt{-7})$,
the fields $\Q(\sqrt{37})$ and $K$ are linearly disjoint over $\Q$.
By Proposition~\ref{prop:gen-scal-ext-iterated},
\begin{align*}
A&:=B\otimes_F L=D(x_1,x_2;\wt\sigma_1,\wt\sigma_2;u)\otimes_{\Q(t_1,t_2)} \Q(\sqrt{37})(t_1,t_2) \\
&\cong D'(x_1,x_2;\wt\sigma'_1,\wt\sigma'_2;u'),
\end{align*}
with $D'=D\otimes_\Q \Q(\sqrt{37})$,
$\wt\sigma'_i=\wt\sigma_i\otimes\id$ for $1\leq i\leq 2$,
and $u'=u\otimes 1$.
We have $D'=(\frac{a,b}{K(\sqrt{37})})$.
It was shown in Example \ref{ex:noncr-prod-abel-2}
that $D'$ is a division algebra,
and that $D'$ does not contain a maximal subfield which is Galois over $\Q(\sqrt{37})$.
Thus, $A\cong D'(x_1,x_2;\wt\sigma'_1,\wt\sigma'_2;u')$ is a division algebra by Theorem~\ref{thm:iter-twist-properties}
and a noncrossed product by Corollary \ref{cor:iter-alg-nonr-prod}.
Let $\Gal(L/F)=\gen{\sigma_3}$,
\ie $\sigma_3(\sqrt{37})=-\sqrt{37}$,
$\sigma_3(t_1)=t_1$ and $\sigma_3(t_2)=t_2$.
Let $\wt\sigma_3$ be the automorphism $\wt\sigma_3=\id\otimes\sigma_3$ of $A=B\otimes_F L$.
Then, as an automorphism of $D'(x_1,x_2;\wt\sigma'_1,\wt\sigma'_2;u')$,
$\wt\sigma_3$ fixes $x_1,x_2$ and the subalgebra $D\otimes 1$ of $D'$.
By Example \ref{ex:twisted-scal-ext},
$$ A(x_3;\wt\sigma_3)=D'(x_1,x_2;\wt\sigma'_1,\wt\sigma'_2;u')(x_3;\wt\sigma_3) $$
is a noncrossed product division algebra with index $2\ind B=16$
and exponent $\lcm(2,\exp B)=8$.
The indeterminate $x_3$ commutes with $x_1$ and $x_2$ and with the subalgebra $D\otimes 1$ of $D'$,
and it holds $x_3\sqrt{37}=-\sqrt{37}x_3$.

\backmatter

\chapter*{Appendix}

\section{$p$-Algebras}
\label{sec:p-Algebras}

Let $F$ be a field with $\charak F=p$, a prime number.
By a $p$-extension we mean a field extension $K/F$
such that $[K:F]$ is a $p$-power
and by a $p$-algebra we mean an $A\in\calA(F)$
such that $\deg A$ is a $p$-power.

Let $K/F$ be a Galois $p$-extension of degree $p$.
Then $K=F(\alpha)$ for some $\alpha\in K$ with 
$\alpha^p-\alpha\in F\bs\calP(F)$,
where $\calP(F)=\sett{x^p-x}{x\in F}$.
Conversely, any $K$ of this form is a Galois extension of $F$.
The Galois group is generated by the automorphism given by $\alpha\mapsto\alpha+1$.
This is the well known Artin-Schreier construction of $p$-extensions of degree $p$
(\cf \cite[Chapter~VI, Theorem 6.4]{lang:algebra}).

Let $A\in\calA(F)$, $\deg A=p$,
and suppose there is an element $\beta\in A\bs F$ such that $\beta^p\in F$.
Then $A$ is generated over $F$ by $\beta$ and a second element $\alpha\in A$
that satisfy
\begin{equation}
  \label{eq:p-alg-gen}
\alpha^p-\alpha=a\in F, \quad \beta^p=b\in F, \quad \beta\alpha=(\alpha+1)\beta. 
\end{equation}
In this case we write $A=[a,b)$.
The goal of this section is to prove 

\begin{prop}
\label{prop:normal-galois-appendix}
Let $K/k$ be any finite separable field extension and $D\in\calD(K)$.
If $D$ contains a maximal subfield that is normal over $k$,
then $D$ also contains a maximal subfield that is Galois over $k$.
\end{prop}

This is essentially a statement about $p$-algebras and was proved in \cite[Lemma 3]{saltman:noncr-prod-small-exp} for $K=k$.
The same proof as given there can be used for the case that $K/k$ is finite separable.
However, we give a proof here for completeness up to the Lemma \ref{lem:p-alg-a-b} below, whose proof is quite difficult.
The following two lemmas are due to Saltman.

\begin{lemma}
\label{lem:p-alg-a-p}
Let $A=[a,b)$, $a,b\in F$.
Then $A=[a^p,b)$.
\end{lemma}
\begin{proof}
Let $A$ be generated by elements $\alpha,\beta$ satisfying (\ref{eq:p-alg-gen}).
Then $A$ is also generated by $\alpha^p=\alpha+a$ and $\beta$,
and these elements satisfy (\ref{eq:p-alg-gen}) where $a$ is replaced by $a^p$.
Thus $A=[a^p,b)$.
\end{proof}

\begin{lemma}
\label{lem:p-alg-a-b}
Let $A=[a,b)$, $a,b\in F$.
Suppose $F'\subseteq F$ is a subfield, $[F:F']<\ift$,
and $b\in F'$.
Then $A=[a',b')$ for some $a'\in F',b'\in F$.
\end{lemma}
\begin{proof}
The proof is quite complicated and shall be omited here.
It can be found in \cite[Lemma 4.4.16]{jacobson:fin-dim-div-alg} 
or \cite[Lemma 6]{saltman:splittings-cycl-p-alg}.
\end{proof}

\begin{proof}[Proof of Proposition \ref{prop:normal-galois-appendix}]
Suppose that $D$ contains a maximal subfield $L$ that is normal over $k$.
Then choose such an $L$ with maximal separable degree $[L:k]_s$.
Let $L_0$ be the maximal separable subextension of $L/k$,
\ie $[L_0:k]=[L:k]_s$.
Since $L/k$ is normal, 
$L_0/k$ is Galois, and there is a purely inseparable subfield $P$ of $L/k$ with $L=L_0P$
(\cf \cite[Chapter~V, Proposition~6.11]{lang:algebra}).
$K/k$ is separable by hypothesis, thus $K\subseteq L_0$.
Let $[P:k]=p^n$ and assume $n\geq 1$, in particular $\charak k=p$.
There is a subfield $P'$ of $P/k$ with $[P:P']=p$, $[P':k]=p^{n-1}$
and $P=P'(\beta)$ with $\beta^p=b\in P'$.
Consider $D^*:=C_D(L_0P')$,
which is a $p$-division algebra of degree $p$ over $Z(D^*)=L_0P'$
with maximal subfield $L=L_0P'(\beta)$.
Since $\beta^p=b\in P'$,
we have $D^*=[a,b)$ for some $a\in L_0P'$.
By Lemma \ref{lem:p-alg-a-b},
there are $a'\in P'$ and $b'\in L_0P'$ such that $D^*=[a',b')$.
Since ${a'}^{p^{n-1}}\in k$, 
repeated application of Lemma \ref{lem:p-alg-a-p} shows
that we can assume $a'\in k$.
Let $\alpha\in D^*\bs Z(D^*)$ with $\alpha^p-\alpha=a'\in k$.
Then $L_0P'(\alpha)=L_0P'\cdot k(\alpha)$ is a maximal subfield of $D^*$ and $D$
that is normal over $k$
with $$[L_0P'(\alpha):k]_s=[L_0(\alpha):k]>[L_0:k]=[L:k]_s,$$
since $\alpha\notin L_0$.
This contradicts the maximality of $[L:k]_s$,
hence $n=1$, $P=k$, and $L$ is Galois over $k$.
\end{proof}

\nocite{*}
\newcommand{\etalchar}[1]{$^{#1}$}
\providecommand{\bysame}{\leavevmode\hbox to3em{\hrulefill}\thinspace}

\chapter*{List of Symbols}

\begin{tabbing}
$\N$ \hspace{3cm} \= $\{1,2,3,\ldots\}$ \\
$\N_0$ \> $\{0,1,2,\ldots\}$ \\
$\mathfrak{Z}_n$ \> cyclic group of order $n$ \\
$\exp G$ \> exponent of abelian group \\
\\
$\charak F$ \> characteristic of  $F$ \\
$F_{alg}$ \> algebraic closure of  $F$ \\
$\mu_n(F)$ \> group of $n$-th roots of unity contained in $F$\\
$\mu_n\subset F$ \> $F$ contains all $n$-th roots of unity of $F_{alg}$ \\
$\zeta_n$ \> primitive $n$-th root of unity \\
\\
$\Gal(K/F)$ \> Galois group of $K/F$ \\
$\Fix(\cdot)$ \> fixed field \\
$\No_{K/F}$ \> Norm of $K/F$ \\
$\Tr_{K/F}$ \> Trace of $K/F$ \\
\\
$\Val(F)$ \> set of all valuations on $F$ \\
$\V0(F)$ \> set of all non-archimedian valuations on $F$ \\
$F_v$ \> completion of $F$ with respect to $v\in\Val(F)$ \\
$F^h$ \> Henselization of $F$ with respect to a fixed $v\in\Val(F)$ \\
$\bar F_v$ \> residue field of $F$ with respect to $v\in\V0(F)$\\
$\bar F$ \> residue field of $F$ with respect to a fixed $v\in\V0(F)$\\
\\
$w|v$ \> the valuation $w$ extends the valuation $v$ \\
$G_v(K/F)$ \> decomposition group of $v\in\V0(F)$ in a Galois extension\\
$Z_v(K/F)$ \> decomposition field of $v\in\V0(F)$ in a Galois extension\\
$I_v(K/F)$ \> inertia group of $v\in\V0(F)$ when $v$ extends uniquely  \\
$T_v(K/F)$ \> inertia field of $v\in\V0(F)$ when $v$ extends uniquely  \\
$e_v(K/F)$ \> ramification index \\
$f_v(K/F)$ \> inertia degree \\
\\
$A^\times$ \> multiplicative group of units of $A$ \\
$A^{op}$ \> opposite algebra of $A$ \\
$Z(A)$ \> centre of $A$ \\
$C_B(A)$ \> centralizer of $B$ in $A$ \\
$\iota_a$ \> inner automorphism defined by $\iota_a(x):=axa^{-1}$ \\
$\calA(F)$ \> set of all finite-dimensional central simple $F$-algebras \\
$\calD(F)$ \> set of all finite-dimensional $F$-division algebras \\
$A^K$ \> scalar extension, $A\otimes_F K$ \\
$A_K$ \> underlying division algebra of $A^K$ \\
$A\sim B$ \> $A$ and $B$ are similiar, $A\cong M_n(D)$ and $B\cong M_m(D)$ \\
$[A]$ \> similarity class of $A$ in $\calA(F)$ \\
$\Br(F)$ \> Brauer group of $F$ \\
$\res_{K/F}$ \> restriction map $\Br(F)\to\Br(K),[A]\mapsto[A^K]$ \\
$\deg A$ \> degree of $A$, $\sqrt{[A:Z(A)]}$ \\
$\ind A$ \> index of $A$, degree of the underlying division algebra of $A$ \\
$\exp A$ \> exponent of $A$, order of $[A]$ in $\Br(F)$ \\
\\
$V_D$ \> valuation ring of the valued division ring $D$ \\
$\Gamma_D$ \> value group of the valued division ring $D$ \\
$\bar D$ \> residue algebra of the valued division ring $D$ \\
$\pi_F$ \> canonical residue map \\
\\
$A_v$ \> the completion of $A$ with respect to $v$, $A\otimes_F F_v$ \\
$D^h$ \> Henselization of $D\in\calD(F)$, $D\otimes_F F^h$ \\
$\inv_v A$ \> local invariant, $\inv A_v$ \\
$\VBr(F)$ \> ``valued part'' of $\Br(F)$ \\
$\IBr(F)$ \> ``inertial part'' of $\Br(F)$ \\
$\SBr(F)$ \> ``inertially split part'' of $\Br(F)$ \\
$\Br(V_F)$ \> Brauer group of the valuation ring $V_F$ \\
\\
$(\frac{a,b}{F})$ \> quaternion algebra over $F$ generated by $i,j$ with \\
\> $i^2=a, j^2=b,ji=-ij$ \\
$(\frac{a,b}{F,\zeta_n})$ \> symbol algebra over $F$, $\zeta_n\in F$, generated by $i,j$ with \\
\> $i^n=a, j^n=b,ji=\zeta_n ij$ \\
$[a,b)$ \> $p$-algebra generated by $\alpha,\beta$ with \\
\> $\alpha^p-\alpha=a,\beta^p=b,\beta\alpha=(\alpha+1)\beta$ \\
\\
$\Aut_F(A)$ \> group of $F$-algebra automorphisms of $A$ \\
$(\omega,f)$ \> factor set \\
$(\omega,f)\sim (\eta,g)$ \> $(\omega,f)$ and $(\eta,g)$ are comohologous \\
$\calF(G,A^\times)$ \> set of all factor sets of $G$ in $A^\times$ \\
$\calH(G,A^\times)$ \> quotient of $\calF(G,\ug A)$ by $\sim$ \\
$Z^2(G,K^\times)$ \> set of all $2$-cocycles of $G$ in $K^\times$ \\
$B^2(G,K^\times)$ \> set of all $2$-coboundaries of $G$ in $K^\times$ \\
$(\wt\sigma,\alpha)$ \> cyclic factor set \\
$(\wt\sigma,u,\alpha)$ \> abelian factor set \\
$(A,G,(\omega,f))$ \> generalized crossed product \\
$(A,\wt\sigma,\alpha)$ \> cyclic generalized crossed product \\
$(A,\wt\sigma,u,\alpha)$ \> abelian generalized crossed product \\
$A(x;\wt\sigma)$ \> twisted function field \\
$A(\!(x;\wt\sigma)\!)$ \> twisted Laurent series ring \\
$A(x;\wt\sigma;u)$ \> iterated twisted function field \\
$A(\!(x;\wt\sigma;u)\!)$ \> iterated twisted Laurent series ring \\
\end{tabbing}

\end{document}